\documentclass[a4paper,reqno,english]{amsart}

\usepackage{graphicx} 
\usepackage[utf8]{inputenc}
\usepackage[pagebackref,colorlinks,linkcolor=red,citecolor=blue,urlcolor=blue,hypertexnames=false]{hyperref}
\usepackage{csquotes}

\usepackage{color}

\usepackage{enumerate}

\usepackage{graphicx}
\usepackage{mathrsfs}
\usepackage{amsfonts}

\usepackage{amsmath,amsthm,amssymb,amsopn}
\usepackage{amsrefs}
\usepackage{amscd}
\input xy
\xyoption{all}
\newdir{ >}{{}*!/-5pt/@{>}}
\usepackage{mathtools}
\let\overfence\overbrace 
\let\downfencefill\downbracefill 
\patchcmd{\overfence}{\downbracefill}{\downfencefill}{}{}
\patchcmd{\downfencefill}{\braceru \bracelu}{}{}{}

\usepackage{etoolbox}
\apptocmd{\sloppy}{\hbadness 10000\relax}{}{}

\newcommand{\comment}[1]{}

\newcommand\fA{\mathfrak{A}}

\def\N{\mathbb{N}} 
\def\Z{\mathbb{Z}} 
\def\Q{\mathbb{Q}} 
\def\R{\mathbb{R}} 
\def\C{\mathbb{C}} 
 
\def\F{\mathbb{F}}

\def\topdf{\texorpdfstring}
\def\inf{\mathrm{inf}}

\def\sotimes{\overset{\sim}{\otimes}}
\theoremstyle{plain}
\newtheorem{thm}[equation]{Theorem}
\newtheorem{lem}[equation]{Lemma}
\newtheorem{coro}[equation]{Corollary} 
\newtheorem{prop}[equation]{Proposition}

\def\s{\mathfrak{s}}

\theoremstyle{definition}
\newtheorem{defi}[equation]{Definition} 

\newtheorem{ex}[equation]{Example}

\newtheorem{stand}[equation]{Standing assumption}
\newtheorem{conj}[equation]{Conjecture}
\newtheorem{quest}[equation]{Question}

\theoremstyle{remark} 

 \newtheorem{rem}[equation]{Remark}

\newtheorem*{ack}{Acknowledgements}

\newcommand{\cA}{\mathcal A}
\newcommand{\cB}{\mathcal B}
\newcommand{\cC}{\mathcal C}

\newcommand{\cG}{\mathcal G}
\newcommand{\cH}{\mathcal H}

\newcommand{\cK}{\mathcal K}

\newcommand{\cO}{\mathcal O}

\newcommand{\cR}{\mathcal R}

\newcommand{\cU}{\mathcal U}
\newcommand{\cV}{\mathcal V}

\def\fA{\mathfrak{A}}

\def\fB{\mathfrak{B}}

\def\fC{\mathfrak{C}}

\def\fF{\mathfrak{F}}

\newcommand{\BF}{\fB\fF}

\def\ab{\mathfrak{Ab}}

\def\gr{\operatorname{gr}}

\def\alg{\mathrm{Alg}}
\def\sotimes{\overset{\sim}{\otimes}}
\newcommand{\aha}{{\alg_{\K}}}

\newcommand{\lra}{\longrightarrow}
\newcommand{\iso}{\overset{\sim}{\lra}}

\newcommand{\ol}{\overline}

\newcommand{\K}{\mathsf{k}}
\newcommand{\T}{{\mathbb{T}}}
\newcommand{\Perf}{\operatorname{Perf}}
\newcommand{\DD}{\operatorname{\bf D}}
\newcommand{\sg}{\operatorname{sgn}}
\newcommand{\singg}{\operatorname{sng}}

\newcommand\Gr[1][]{{\operatorname{{Gr}-}}}
\newcommand{\Modd}{\operatorname{Mod-}}
\newcommand{\M}{\operatorname{\mathbb M}}
\newcommand{\Pgrp}{\operatorname{P^{gr}-\!}}
\def\a{\alpha}
\def\TT{\mathcal{T}}
\def\LL{\mathscr{L}}
\newcommand{\FK}{\operatorname{FK}}
\newcommand{\FKbar}{\operatorname{\mathrm{F}\ol{\mathrm{K}}}}

\def\reg{\operatorname{reg}}
\def\sing{\operatorname{sing}}
\def\sink{\operatorname{sink}}
\def\inf{\operatorname{inf}}
\def\sour{\operatorname{sour}}
\def\triqui{\vartriangleleft}
\def\quitri{\vartriangleright}
\def\mspan{\operatorname{span}}
\def\supp{\operatorname{supp}}

\def\rk{\operatorname{rk}}

\def\Gl{\operatorname{GL}}

\def\Idem{\operatorname{Idem}}
\def\ev{\operatorname{ev}}
\def\id{\operatorname{id}}

\newcommand{\hltimes}{{ \ \widehat{\ltimes}\ }}

\def\Path{\operatorname{Path}}

\newcommand{\coker}{{\rm Coker}}
\renewcommand{\ker}{{\rm Ker}}

\newcommand{\op}{\mathrm{op}}

\def\Ext{\operatorname{Ext}}

\def\tor{\operatorname{Tor}}
\def\Hom{\operatorname{Hom}}
\def\End{\operatorname{End}}
\def\ab{\operatorname{ab}}
\DeclareMathOperator*{\Tor}{Tor}

\def\comp{\operatorname{comp}}

\renewcommand{\top}{\operatorname{top}}

\makeatletter
\@namedef{subjclassname@2020}{\textup{2020} Mathematics Subject Classification}
\makeatother
\numberwithin{equation}{section}

\title{Classification Conjectures for Leavitt path algebras}

\author{Guillermo Corti\~nas}
\address{Guillermo Corti\~nas: Dep. Matem\'atica-IMAS\\ Facultad de Ciencias Exactas y Naturales\\
Universidad de Buenos Aires\\ Argentina}
\email{gcorti@dm.uba.ar}

\author{Roozbeh Hazrat}
\address{Roozbeh Hazrat: 
Centre for Research in Mathematics and Data Science\\
Western Sydney University\\
Australia} \email{r.hazrat@westernsydney.edu.au}
\thanks{The first named author is a CONICET researcher.~He was supported by grants 
PIP 11220200100423CO from CONICET, UBACyT 20020220300206BA from UBA and PICT 2021-I-A-00710 from ANPCyT. The second author acknowledges Australian Research Council
grant DP230103184.
}

\begin{document}

\begin{abstract} 
The theory of Leavitt path algebras is intrinsically related, via graphs, to the theory of symbolic dynamics and $C^*$-algebras where the major classification programs have been a domain of intense research in the last 50 years. In this article, we gather together current lines of research in the classification of Leavitt path algebras, questions, conjectures, and some of the results about them that have been obtained so far. 

\end{abstract}

\maketitle

\section{Introduction}\label{sec:intro}
\numberwithin{equation}{section}
Leavitt path algebras are the discrete, purely algebraic version of graph $C^*$-algebras. The classification of these algebras has received a lot of attention, as they are closely related to the classification of graph $C^*$-algebras and those of symbolic dynamics. However, while graph $C^*$-algebras have been classified by means of $K$-theoretic invariants ~\cite{errs3}, an analogous classification for Leavitt path algebras remains conjectural.

A $C^*$-algebra is \emph{simple} if it has no non-trivial closed ideals. Simple graph $C^*$-algebras are completely classified by the  Grothendieck group $K_0$ ~\cites{ror,rordam111,rordam222}. 
For graphs with unital non-simple $C^*$-algebras, it became clear that one way to preserve enough information is to further consider the $K$-groups of the ideals and their subquotients and how they are related to each other via the long exact sequence of $K$-theory groups.
 This is now called filtered $K$-theory; after the pioneering work of  R\o rdam~\cite{ror} and Restorff~\cite{restorff}, it was further investigated by Eilers, Restorff, Ruiz and S\o rensen~\cites{errs4,errs2}. In the major article~\cite{errs3} they show that filtered $K$-theory is a complete invariant for unital graph $C^*$-algebras. 

 However, at the level of Leavitt path algebras, it is not known whether $K_0$ or other $K$-theoretical invariants classify simple Leavitt path algebras. It has now been over 15 years since the question on whether the Grothendieck group $K_0$ classifies simple purely infinite Leavitt path algebras of finite graphs was posed~\cite{question}. Similarly it is already over a decade since it was conjectured that the graded Grothendieck group $K_0^{\gr}$ is a complete invariant of the graded isomorphism class of \emph{all} unital Leavitt path algebras~\cite{haz2013}. Despite a flurry of activity, the questions and conjectures above remain unsolved; in particular no counterexamples to them have been found.

 In this survey paper, we present in one place several results on classification which are scattered in the literature  and further discuss various $K$-theoretical invariants which conjecturally could classify Leavitt path algebras.  We look at classical $K$-theory, graded $K$-theory, filtered $K$-theory and bivariant $K$-theory.  

The rest of this article is organized as follows. 

Section \ref{sec:preli} recalls some basic facts about monoids, graphs and $C^*$-algebras. 

Section \ref{sec:leavitt} introduces our main object of study, the Leavitt path algebra,  which is a $\K$-algebra $L(E)$ associated to a directed graph $E$, with coefficients in a field $\K$. This section also recalls  some of its main properties. 

Section \ref{sec:elliott} concerns Elliott's $K$-theoretic classification of unital ultramatricial algebras \cite{elliott}, which appears as Theorem \ref{thm:elliott}, after a brief introduction to the Grothendieck group $K_0$ and ultramatricial algebras, highlighting some examples related to Leavitt path algebras.  

Section \ref{sec:KP} is about the classification question for simple purely infinite (SPI) Leavitt path algebras, which aims to classify them up to isomorphism/Morita equivalence
using the Grothendieck group $K_0$. The question is inspired by R\o rdam's classification of SPI graph $C^*$-algebras of finite graphs \cite{ror} which is a particular case of the more general 
Kirchberg-Phillips' result \cite{P} classifying all Kirchberg algebras. After some prelimnaries on SPI rings (Subsection \ref{subsec:simple}), higher algebraic and topological 
$K$-theory (Subsection \ref{subsec:hik}) and Kasparov's bivariant $K$-theory (Subsection \ref{subsec:KK}), the Kirchberg-Phillips theorem \ref{thm:KP} is stated in Subsection 
\ref{subsec:KP}, and the question, posed in \cite{question}, of whether a similar theorem holds for SPI Leavitt path algebras of finite graphs, is raised. We then explain the different strategies 
of proof of these $C^*$-algebraic results, and the extent to which they can or cannot be translated to the purely algebraic setting. For example we recall Kirchberg's absorption 
theorems for $C^*$-algebras \cite{KP} (stated here as Theorem \ref{thm:KP}) and how it was proved in \cite{aratenso} that they fail to hold for Leavitt path algebras (Theorem 
\ref{thm:aratenso}). For another example, we explain how Franks' theorem \ref{thm:franks} was used in \cite{flow} to show that flow-equivalent graphs have Morita equivalent Leavitt 
path algebras (Theorem \ref{thm:flow}). Then we recall R\o rdam's theorem \cite{ror} that a finite SPI graph and its Cuntz splice have isomorphic $C^*$-algebras (Theorem 
\ref{thm:splice}), and review some of the research that has been published (\cite{johsor}, \cite{tokel2}, \cite{ac1}) as yet inconclusive, trying to determine whether the same is 
true for Leavitt path algebras (see Theorem \ref{thm:sgn=diag},  Corollary \ref{coro:tokel2} and Theorem \ref{thm:noexistis}). Kirchberg and Phillips use Kasparov's bivariant 
$K$-theory to obtain a classification result for Kirchberg algebras up to homotopy which they then transform to a classification up to isomorphism using analytic methods. After an 
introduction to bivariant algebraic $K$-theory \cite{ct}, and a description in Theorems \ref{thm:fundtri} and \ref{thm:kksing} of the $kk$-theoretic properties of Leavitt path algebras from \cite{cm1}, we explain how the latter were used in \cite{cm2} to establish that $K_0$ classifies SPI Leavitt path algebras of finite graphs up to 
homotopy (Theorem \ref{thm:cm2}), and how this classification ``passes to the completion" \cite{bullift} to a homotopy classification of SPI graph $C^*$-algebras (Theorem 
\ref{thm:bullift}).

Section \ref{sec:symb} reviews basic concepts and results from symbolic dynamics, such as shift spaces, conjugacy, Williams' theorem \cite{williams} that strong shift equivalence 
implies shift equivalence (Theorem \ref{willmnhfhf}), the fact that the converse does not hold \cites{kr1,kr2}, and Krieger's theorem \cite{krieger} that shift equivalence is classified by his dimension 
invariant (Theorem \ref{kriegerthm}). 

Section \ref{sec:gradK} concerns graded modules and graded $K$-theory $K^{\gr}$. For a ring graded over an abelian group $\Gamma$, the groups $K_n^{\gr}(A)$ carry a 
natural $\Z[\Gamma]$-module structure. As an application of graded $K$-theory we sketch a proof of a result from \cite{ac1},  stated in  Section \ref{sec:KP} as Theorem \ref{thm:noexistis}, asserting the non-existence of 
$\Z/m\Z$-graded unital homomorphisms between the Leavitt path algebras of the $n$-petal rose graph $\cR_n$ and of its Cuntz' splice for $n\ge 2$ and $m\in \{0\}\cup\N_{\ge 2}$.

Section \ref{sec:gradsymb} is about the graded classification conjectures for Leavitt path algebras and their connections to symbolic dynamics. These conjectures take into account 
the $\Z$-grading that Leavitt path algebras have, and aim to classify them up to graded isomorphism/graded Morita equivalence by means of the graded Grothendieck 
$\Z[x,x^{-1}]$-module $K_0^{\gr}$. In Subsection \ref{subsec:poly} we recall the definition of polycephaly graphs, and the result \cite{haz2013} that if $E$ and $F$ are polycephaly 
and $K_0^{\gr}(L(E))\cong K_0^{\gr}(L(F))$ as pointed partially ordered groups, then $L(E)\cong_{\gr}L(F)$ (Theorem \ref{mani543}). This motivated the graded classification 
conjecture  (Conjecture \ref{conjalgiso}), first formulated in \cite{haz2013}, which asserts that the latter classification result holds for general graphs. In Subsection 
\ref{subsec:ample} we recall the definition of amplified graphs and Eilers-Ruiz-Sims' result \cite{eilers2} that the conjecture is true for them (Theorem \ref{thm-main}). 
Subsection \ref{subsec:twist} recalls from \cite{skew} the description of the Leavitt path algebra of a finite graph as a Laurent polynomial ring over its zero-degree component, 
skewed by a corner isomorphism $\alpha$, and Ara-Pardo's result \cite{apgrad} (stated here as Theorem \ref{theor:Kiffgr-iso unital}) that the graded classification conjecture holds 
up to twisting $\alpha$ by a locally inner automorphism. Subsection \ref{subsec:full} recalls another conjecture from \cite{haz2013} which has been established independently by 
Arnone \cite{guido2023} and Va\v{s} \cite{vas} (Theorem \ref{thm:full}). The conjecture/theorem says that any homomorphism of pointed partially ordered $\Z[x,x^{-1}]$-modules 
$K_0^{\gr}(L(E))\to K_0^{\gr}(L(F))$ lifts to a unital graded homomorphism $L(E)\to L(F)$. In Subsection \ref{subsec:gradmor}  we review graded Morita equivalence of graded rings 
\cites{hazbook, abramsmori}. Subsection \ref{subsec:art} recalls results from \cite{hazd} relating graded Morita equivalence and the isomorphism class of $K_0^{\gr}$ to shift 
equivalence. Theorem \ref{h99} says that for finite $E$ and $F$, $K_0^{\gr}(L(E))\cong K_0^{\gr}(L(F))$ is equivalent to the adjacency matrices $A_E$ and $A_F$ being shift-equivalent. Proposition \ref{hgysweet} says that if $A_E$ and $A_F$ are strongly shift equivalent then $L(E)$ and $L(F)$ are graded Morita equivalent, and that if the latter holds, 
then $A_E$ and $A_F$ are shift equivalent. Thus graded Morita equivalence of Leavitt path algebras sits right in between strong shift equivalence and shift equivalence. Subsection 
\ref{subsec:c*ver} concerns a $C^*$-version of the graded classification conjecture (Conjecture \ref{conjanal}); it predicts that equivariant $K_0$ classifies graph $C^*$-algebras 
up to $\ast$-isomorphism preserving the gauge circle action. Subsection \ref{subsec:moritasing} reviews a result of Chen and Yang (Theorem \ref{thm:yache}) relating 
graded Morita equivalence of $L(E)$ and $L(F)$ to derived equivalence and equivalence between the singularity categories of certain truncated path algebras \cite{chen}, and states 
the graded Morita classification conjecture (Conjecture \ref{conjmogr}), which predicts that $K_0^{\gr}(L(E))\cong K_0^{\gr}(L(F))$ as partially ordered $\Z[x,x^{-1}]$-modules implies, among other things, that $L(E)$ and $L(F)$ 
are graded Morita equivalent and that $C^*(E)$ and $C^*(F)$ are gauge equivariantly Morita equivalent. Subsection \ref{talentedmoni} recalls the definition of the talented monoid of a graph introduced in \cite{hazli}, explains that it detects graph 
theoretic properties (Theorem \ref{conLm}) and how it was used in \cite{oberwol} to establish the graded Morita classification conjecture for meteor graphs 
(Theorem \ref{alergy1}). Subsection \ref{subsec:homotopy2} states Theorem \ref{thm:guidotopy} from \cite{guidotopy} which says that $K_0^{\gr}$ classifies Leavitt path algebras of 
finite primitive graphs up to graded homotopy. Subsection \ref{subsec:compa} concerns certain compatibility conditions that one may add to shift equivalence between the adjacency matrices of graphs $E$ and $F$. By \cite{recast} these conditions are implied by strong shift equivalence (Proposition \ref{prop:ssegse}) and imply that $L(E)$ and $L(F)$ are graded Morita equivalent (Theorem \ref{thm:recast}).

Section \ref{sec:filk} is the final section. It reviews filtered topological \cites{errs,errs2,errs3,errs4} and algebraic $K$-theory \cite{arahazli} and their relations to each other and to graded $K$-theory. It is explained how these theories were used in \cite{arahazli} to establish Proposition \ref{bfg1998d}, which says that shift equivalent graphs have Morita equivalent $C^*$-algebras. The section ends with a diagram exhibiting known and conjectural connections between various $K$-theoretic invariants and equivalences of algebras, graphs and matrices.

\begin{ack} We would like to thank Gene Abrams, Pere Ara, Efren Ruiz, and Lia Va\v{s} for carefully reading the paper and giving us feedback. Thanks also to the anonymous referee for a very careful reading of the article and for comments that improved the presentation of the paper. Part of the research for this article was carried out during a visit of the second named author to the Santal\'o Mathematics Research Institute (IMAS) of the Mathematics Department of the School of Exact and Natural Sciences of the University of Buenos Aires. He wholeheartedly thanks his host Willie Corti\~nas and the people at IMAS-DM for their hospitality. 
\end{ack}
\section{Preliminaries }\label{sec:preli}

Throughout $\N_0$ denotes the monoid of non-negative integers and $\N$ denotes the positive integers.

\subsection{Monoids}\label{subsec:mono}
\numberwithin{equation}{subsection}
Let $M$ be a commutative monoid, written additively,  and $\Gamma$ a group acting on $M$ by monoid automorphisms. For $\alpha \in \Gamma$ and $a\in M$, we denote the action of $\alpha$ on $a$ by ${}^\alpha a$. 
A monoid homomorphism $\phi:M_1 \rightarrow M_2$ is a $\Gamma$-\emph{monoid homomorphism} if it respects the action of $\Gamma$, i.e. if $
\phi({}^\alpha a)={}^\alpha \phi(a)$ $\forall a\in M$. Recall a \emph{pre-order} on a set $X$ is a reflexive and transitive relation. We define the \emph{algebraic}  pre-order on the monoid $M$ by $a\leq b$ if $b=a+c$, for some $c\in M$. 
Throughout we write $a \parallel b$ if the elements $a$ and $b$ are not comparable.  A monoid is \emph{conical} if $a+b=0$ implies $a=b=0$; it is \emph{cancellative} if $a+b=a+c$ implies $b=c$. A nonzero element $a\in M$ is an \emph{atom} if $a=b+c$ implies that $b=0$ or $c=0$. An element 
$ a\in M$  is called \emph{minimal} if $b\leq a\Rightarrow a\leq b$. When $M$ is conical and cancellative, an element $a\in M$ is an atom if and only if it is minimal if and only if $0\not = b\leq a\Rightarrow a=b$. A monoid of interest in the paper, the talented monoid of a graph $E$, denoted by $T_E$ (see Section~\ref{talentedmoni}) is conical and cancellative and thus the equivalence just mentioned holds.  

Throughout we assume that the group $\Gamma$ is abelian. Indeed in our setting of graph algebras, most of the time this group is the group of integers $\mathbb Z$. We call an element $0 \not = a\in M$ \emph{periodic} if there exists $0 \not = \alpha \in \Gamma$ such that ${}^\alpha a =a$. If $a\in M$ is not periodic, we call it \emph{aperiodic}. We denote the orbit of the action of $\Gamma$ on an element $a$ by $O(a)$, so $O(a)=\{{}^\alpha a \mid \alpha \in \Gamma \}$. 

We say that $\Gamma$ acts \emph{freely} on a monoid $M$ if ${}^\gamma m=m$, where $0\not = m \in M$,  implies that $\gamma$ is the identity of the group, i.e., all the isotropy groups of the action are trivial.

 A $\Gamma$-\emph{order-ideal} of a monoid $M$ is a  subset $I$ of $M$ such that for any $\alpha,\beta \in \Gamma$, ${}^\alpha a+{}^\beta b \in I$ if and only if 
$a,b \in I$. Equivalently, a $\Gamma$-order-ideal is a submonoid $I$ of $M$ which is closed under the action of $\Gamma$ and is
\emph{hereditary} in the sense that $a \le b$ and $b \in I$ imply $a \in I$. The set $\mathcal L(M)$ of $\Gamma$-order-ideals of $M$ forms a (complete) lattice. We say $M$ is a \emph{simple} 
$\Gamma$-\emph{monoid} if $M\ne 0$ and the only $\Gamma$-order-ideals of $M$ are $0$ and $M$.

\subsection{Graphs}\label{graph6464}

A (directed) graph $E$ is a tuple $(E^{0}, E^{1}, r, s)$, consisting of sets $E^{0}$ and $E^{1}$ and maps $s,r:E^1\to E^0$. The elements of $E^0$ and $E^1$ are respectively the \emph{vertices} and the \emph{edges} of $E$. A graph $E$ is \emph{finite} or \emph{countable} if both $E^0$ and $E^1$ are.

\begin{stand}\label{stand:count}
All graphs are assumed countable. 
\end{stand}

If $e\in E^1$ then $s(e)$ and $r(e)$ are the \emph{source} and the \emph{range} of $e$. We think of an edge as an arrow going from its source to its range.
We use the convention that a \emph{path} $\alpha$ in $E$ of \emph{length} $|\alpha|=n\geq 1$ is
a sequence 
\begin{equation}\label{eq:path}
\a=e_{1}e_{2}\cdots e_{n}    
\end{equation} 
with $e_{i}\in E^1$ such that
$r(e_{i})=s(e_{i+1})$ for $1\leq i\leq n-1$.  Set $s(\alpha) = s(e_1)$, and $r(\alpha) =r(e_{n})$. A vertex is viewed as a path in $E$ of length $0$. 
For $n\geq 2$, we define $E^n$ to be the set of paths in $E$ of length $n$ and $\Path(E)=\bigcup_{n\geq 1}E^n$, the set of all paths in $E$.

We shall also have occasion to consider \emph{infinite paths} $e_1e_2\cdots$ and \emph{bi-infinite paths} $\cdots e_{-1}e_0e_1\cdots$, defined in the obvious way. The \emph{support} of a path $\alpha$ is the set $\supp(\alpha)$ of the sources and ranges of its edges. 
We say that a vertex $u$ \emph{connects to a vertex} $v$, and  write $u\ge v$, if there is a path $\alpha$ with $s(\alpha)=u$ and $r(\alpha)=v$. We say that $u$ \emph{connects to a path} $\alpha$ if there exists $v\in\supp(\alpha)$ such that $u$ connects to $v$.

A \emph{closed path based} at a vertex $v\in E^0$ is a path $\alpha$ with $|\alpha|\ge 1$ such that $v=s(\alpha)=r(\a)$. A \emph{cycle} is a closed path \eqref{eq:path} such that $s(e_i) \neq s(e_j)$ for every $i \neq j$. A \emph{loop} is a cycle of length one. We say a graph is \emph{acyclic} if it has no cycles.

Let $v\in E^0$ and $c$ the cardinal of the set $s^{-1}(v)$ of edges it emits. We say that $v$ is a \emph{sink} if $c=0$, an \emph{infinite emitter} if $c$ is infinite, and  \emph{regular} if $1\le c<\infty$. We write $\sink(E)$, $\inf(E)$ and $\reg(E)$ for the sets of sinks, infinite emitters and regular vertices. Nonregular vertices are called \emph{singular}; put $\sing(E)=E^0\setminus\reg(E)$. A graph is \emph{regular} if $E^0=\reg(E)$ and \emph{row-finite} $\inf(E)=\emptyset$. A vertex $v\in E^0$ is a \emph{source} if $r^{-1}(v)=\emptyset$; we write $\sour(E)$ for the set of all sources of $E$. A graph $E$ is \emph{essential} if it is regular and $\sour(E)=\emptyset$.

There is a graph move called \emph{source removal}, consisting of removing a source $v$ as well as all edges $e\in s^{-1}\{v\}$. Applying source removal a finite number of times to a finite regular graph $E$, yields an essential graph. Other moves that appear in symbolic dynamics are reviewed in Section~\ref{sec:symb}.

The (reduced) \emph{adjacency  matrix} of a graph $E$  is the matrix 
$A_E\in \Z^{\reg(E)\times  E^0}$ with entries 
\begin{equation}\label{eq:ae}
(A_E)(v,w)=\#\{e\in E^1\,\colon\, s(e)=v,\, r(e)=w\}.
\end{equation}
Put
\begin{equation}\label{eq:I}
 I\in \Z^{E^0\times\reg(E)},\,\, I_{v,w}=\delta_{v,w}.   
\end{equation}
Observe that $I$ results from the identity matrix with the columns corresponding to the singular vertices removed.
The matrices $A^t_E$ and $I$  will be used to calculate the Grothendieck group $K_0$ of the Leavitt path algebra of $E$~(see \ref{ex:k0bf}).
For row-finite $E$, we shall also have occasion to consider the \emph{unreduced} adjacency matrix $A'_E\in\N_0^{E^0\times E^0}$ with entries
\[
(A'_E)_{v,w}=\#\{e\in E^1\,\colon\, s(e)=v,\, r(e)=w\}.
\]
Remark that $A'_E=A_E$ if and only if $E$ is regular.

\begin{rem}\label{rem:orderver}
As defined above, the adjacency matrix $A_E$ is a function $\reg(E)\times E^0\to \N_0$. It does not depend on any ordering of the vertices of $E$. However if we want to represent it graphically, we need to choose an order, and we do it so that the regular vertices appear first.
\end{rem}
\begin{rem}\label{rem:e(a)}
Let $n\in \N$ and $A\in\M_n(\N_0)$. Let $E(A)$ be the graph with vertices $\{1,\dots,n\}$ and exactly $A_{i,j}$ edges with source $i$ and range $j$. It is straightforward to check that, for the given order of the vertices, we have  $A'_{E(A)}=A$ and that there is a graph isomorphism $E(A'_E)\cong E$.
\end{rem}

Next we describe the notion of a covering graph. These graphs appear when we deal with the smash product of a Leavitt path algebra by the group $\mathbb Z$ and when we describe the graded Grothendieck group of a Leavitt path algebra. 
The \emph{covering graph}  $\ol{E}$ 
of $E$ is given by
\begin{gather}\label{eq:olE}
    \ol E^0 = \big\{v_n \mid v \in E^0 \text{ and } n \in \Z \big\},\qquad
    \ol E^1 = \big\{e_n \mid e\in E^1 \text{ and } n\in \Z \big\},\\
    s(e_n) = s(e)_{n-1},\qquad\text{ and } \qquad  r(e_n) = r(e)_{n}.\nonumber
\end{gather}

As examples, consider the following graphs
\begin{equation*}
{\def\labelstyle{\displaystyle}
E : \quad \,\, \xymatrix{
 u \ar@(lu,ld)_e\ar@/^0.9pc/[r]^f & v \ar@/^0.9pc/[l]^g
 }} \qquad \quad
{\def\labelstyle{\displaystyle}
F: \quad \,\, \xymatrix{
   u \ar@(ur,rd)^e  \ar@(u,r)^f
}}
\end{equation*}
Then
\begin{equation}\label{eq:times1}
\ol{E}: \quad \,\,\xymatrix@=15pt{
\dots  {u_{0}} \ar[rr]^-{e_1} \ar[drr]^(0.4){f_1} &&  {u_{1}} \ar[rr]^-{e_2} \ar[drr]^(0.4){f_2} && {u_{2}}  \ar[rr]^-{e_{3}} \ar[drr]^(0.4){f_{3}} && \cdots\\
\dots {v_{0}}   \ar[urr]_(0.3){g_1} && {v_{1}} \ar[urr]_(0.3){g_2}  && {v_{2}}  \ar[urr]_(0.3){g_{3}}&& \cdots
}
\end{equation}
and
\begin{equation*}
\ol{F}: \quad \,\,\xymatrix@=15pt{
\dots  {u_{0}} \ar@/^0.9pc/[r]^{f_1} \ar@/_0.9pc/[r]_{e_1}  &  {u_{1}} \ar@/^0.9pc/[r]^{f_2} \ar@/_0.9pc/[r]_{e_2} & {u_{2}}  \ar@/^0.9pc/[r]^{f_{3}}  \ar@/_0.9pc/[r]_{e_{3}} & \quad \cdots
}
\end{equation*}
Let $E$ be a row-finite graph. Recall that a subset $H \subseteq E^0$ is said to be \emph{hereditary} if
for any $e \in E^1$ we have that $s(e)\in H$ implies $r(e)\in H$. A hereditary subset $H
\subseteq E^0$ is called \emph{saturated} if whenever $v$ is not a sink, we have
$$\{r(e): e\in E^1 \text{~and~} s(e)=v\}\subseteq H\Rightarrow v\in H.$$ 
We let $\TT_E$ denote the set of hereditary saturated subsets of $E^0$, and order two hereditary saturated subsets $H$ and $H'$ by $H\leq H'$ if  
$H\subset H'$. It has been established that the ordered set $\TT_E$ is actually a lattice (see \cite{lpabook}*{Proposition~2.5.6}).
 
For a hereditary saturated subset $H\subset E^0$, the \emph{quotient graph} $E/H$ is the graph such that
$$(E/H)^0=E^0\setminus H,$$
$$(E/H)^1= \{e\in E^1\;|\; r(e)\notin H\}$$ 
with range and source maps defined as the restrictions of $r$ and $s$ to $(E/H)^1$.

For hereditary saturated subsets $H_1$ and $H_2$ of $E$ with  $H_1 \subseteq  H_2$, define the quotient graph $H_2 / H_1 $ as a graph such that 
$(H_2/ H_1)^0=H_2\setminus H_1$ and $(H_2/H_1)^1=\{e\in E^1\;|\; s(e)\in H_2, r(e)\notin H_1\}$. The source and range maps of $H_2/H_1$ are restricted from the graph $E$. If $H_2=E^0$, then $H_2/H_1$ is the \emph{quotient graph} $E/H_1$ defined above (see also \cite{lpabook}*{Definition~2.4.11}).

\subsection{\topdf{$\ast$}{*}- and \topdf{$C^*$}{C*}-algebras}\label{subsec:c*}
Let $\K$ be a field together with an involutive endomorphism $\ast:\K\to \K$. A \emph{$\ast$-algebra} $A$ over $\K$ is an algebra together with an involutive ring homomorphism $*:A\to A^{\op}$ that satisfies 
\[
(\lambda a)^*=\lambda^*a^*\,\, (\forall \,  \lambda\in\K,\, a\in A). 
\]
For the remainder of this subsection, we fix $\K=\C$ equipped with complex conjugation as involution. A $C^*$-algebra is a Banach algebra $\fA$ together with an isometric involution $\ast$ that makes it a $\ast$-$\C$-algebra and satisfies
\begin{equation}\label{eq:c*norm}
||a||^2=||a^*a|| \,\, (\forall \, a\in A).     
\end{equation}
For example the algebra $\cB(H)$ of bounded operators on a Hilbert space $H$, equipped with the operator norm and the adjoint as involution is a $C^*$-algebra, as is any closed $\ast$-subalgebra of $\cB(H)$. In fact the Gelfand-Naimark theorem \cite{david}*{Theorem I.9.12} says that any $C^\ast$-algebra is (isometrically) $\ast$-isomorphic to a closed $\ast$-subalgebra of $\cB(H)$ for some $H$. 

\begin{ex}\label{ex:b(x)}
Let $\fB$ be a $C^*$-algebra and $X$ a locally compact Hausdorff space. Let $\fB(X)$ be the algebra of continuous functions $X\to \fB$ that vanish at infinity, equipped with pointwise addition, multiplication and involution, and the supremum norm. Then $\fB(X)$ is a $C^*$-algebra, and for each $x\in X$, the \emph{evaluation map} $\ev_x:\fB(X)\to \fB$ is a $*$-homomorphism.     
\end{ex}

\section{Leavitt path algebras}\label{sec:leavitt}
\numberwithin{equation}{section}
The main focus of this survey article is on the classification of certain combinatorial algebras arising from directed graphs, namely, Leavitt path algebras. 

Let $\K$ be a field and $n\ge 2$. In \cite{vitt62} Leavitt considered the free unital associative $\K$-algebra $L_n=L_n(\K)$ generated by symbols $\{x_i,x^*_i \mid 1\leq i \leq n\}$ subject to the following relations 
\begin{equation}\label{jh54320}
x^*_ix_j =\delta_{ij}, \text{ for all } 1\leq i,j \leq n,  \text{  and  } \sum_{i=1}^n x_ix^*_i=1.
\end{equation} 
The relations guarantee that the right $L_n$-module homomorphism 
\begin{align}\label{is329ho}
\phi:L_n&\longrightarrow L_n^n\\
a &\mapsto (x^*_1a	,x^*_2a,\dots,x^*_na)\notag
\end{align}
has an inverse 
\begin{align}\label{is329ho9}
\psi:L_n^n&\longrightarrow L_n\\
(a_1,\dots,a_n) &\mapsto  x_1a_1+\dots+x_na_n, \notag 
\end{align}
so $L_n\cong L_n^n$ as right $L_n$-modules. Leavitt showed that 
$n$ is the smallest $m\ge 2$ such that $L_n\cong L_n^m$, and that $L_n$ is a simple ring. More precisely, he showed that for any nonzero $x\in L_n$ there exist $a,b\in L_n$ such that
\begin{equation}\label{eq:lnspi}
axb=1.    
\end{equation}
Note, however, that $L_n$ is not a division ring or even a domain; in fact we have 
\begin{equation}\label{eq:lndom}
x_2^*x_1=0.     
\end{equation}

It turned out the relations (\ref{jh54320}) are quite ubiquitous in mathematics and thus these rings have fascinating properties and appear in many areas. 
For example Cuntz' algebra $\cO_n$, introduced in \cite{On}, is a $C^*$-algebra completion of $L_n(\C)$. Cuntz and Krieger \cite{ck} introduced the $C^*$-algebra $C^*(E)$ of a finite graph $E$, so that $\cO_n=C^*(\cR_n)$ is the algebra of the graph consisting of a single vertex and $n$ loops, sometimes called the \emph{rose of $n$-petals}. Pictorially, $\cR_n$ is the following graph
\[
\xymatrix{
   \bullet \ar@{.}@(l,d) \ar@(ur,dr)^{x_{1}} \ar@(r,d)^{x_{2}} \ar@(dr,dl)^{x_{3}} 
\ar@(l,u)^{x_{n}}}
\]
The Cuntz-Krieger construction was later generalized, first to row-finite graphs \cite{kpr} and then to all \cite{dritom}.  
The Leavitt path algebra $L_\K(E)$ over a field $\K$ was introduced \cites{aap05,amp} as a purely algebraic counterpart to $C^*(E)$. Leavitt's algebra is recovered as $L_\K(\cR_n)=L_n(\K)$.

We briefly recall the construction of Leavitt path algebras and refer the reader to \cite{lpabook} for a comprehensive treatment. 

Let $E$ be a graph and $\K$ a field. The \emph{Cohn path algebra} $C_\K(E)$ of $E$ is the quotient of the free associative $k$-algebra generated by the set $E^0\cup E^1 \cup \{e^*\;|\;  e\in  E^1\}$, subject to the relations:
\begin{itemize}
\item[(0)] $v\cdot w = \delta_{v, w}$ for $v, w\in E^0$;
\item[(1)] $s(e)\cdot e = e = e \cdot r(e)$ for $e\in E^1$;
\item[(2)]  $r(e) \cdot e^* = e^* = e^*\cdot s(e)$ for $e\in E^1$;
\item[(3)]  $e^*\cdot f = \delta_{e, f} r(e)$ for $e, f\in E^1$.
\end{itemize} 
The elements of the form $e^*$ ($e\in E^1$) are called \emph{ghost edges}.
Suppose an involution $\ast:\K\to \K$ is given (e.g. the identity, or, if $\K\subset\C$, complex conjugation). Then $C_\K(E)$ is a $*$-algebra with the involution that fixes the vertices and maps $e$ to $e^*$ for $e\in E^1$. Throughout for a path
$\a=e_{1}e_{2}\cdots e_{n}$,  we write  $\a^*=e_{n}^*e_{n-1}^*\cdots e_{1}^*$.

Let $v\in\reg(E)$; put
\begin{equation}\label{eq:qv}
q_v:=v-\sum_{v\in s^{-1}(v)}ee^*.    
\end{equation}
Consider the ideal $\mathcal{K}(E)= \big \langle q_v\;|\; v\in \reg(E)\big \rangle$ of $C_\K(E)$. 
The \emph{Leavitt path algebra} $L_\K(E)$ of $E$ over the field $\K$ (see \cite{lpabook}*{Definition~1.2.3}) is the quotient algebra $C_\K(E)/\mathcal{K}(E)$. Hence there is an extension of $\K$-algebras
\begin{align}
 \label{sescohn}
 \CD
 0@>>> \mathcal{K}(E)@>{l}>>C_{\K}(E)@>{p}>>L_{\K}(E)@>>>0.
\endCD
\end{align} 

One may also consider the Cohn and Leavitt path algebras $C_R(E)$ and $L_R(E)$ over any unital ring $R$, defined by the same generators and relations as above. Remark that the algebra $C_{R}(E)$ is unital if and only if $L_{R}(E)$ is, if and only if $E^0$ is finite, in which case $1_E=\sum_{v\in E^0}v$ is the unit element. 

Throughout this paper, if $\K$ is an arbitrary field, we simply write $C(E)$ instead of $C_\K(E)$ and $L(E)$ instead of $L_\K(E)$. We keep the subscript when we want to emphasize the ground field or the fact that the ground ring under consideration is not a field.

The algebras above are naturally equipped with a $\Z$-grading (the canonical grading) where the vertices have degree zero, the edges have degree $1$ and the ghost edges have degree $-1$  (see \cite{lpabook}*{Section 2.1}).
This graded structure plays an important role in the theory of Leavitt path algebras as well as in the current paper. We recall one fact on the graded structure of these algebras, showcasing how germane is the $\mathbb Z$-grading of Leavitt path algebras.

Let $E$ be a row-finite graph. Write $\LL^{\gr}\big(L(E)\big)$ for the lattice of graded (two-sided) ideals of $L(E)$. There is a lattice isomorphism between the set $\TT_E$ of hereditary saturated subsets of $E^0$ and the set $\LL^{\gr}\big(L(E)\big)$ (\cite{lpabook}*{Theorem~2.5.8}). The correspondence maps a hereditary saturated set of vertices to the two-sided ideal it generates:
 \begin{align}\label{latticeisosecideal}
 \Phi: \TT_E&\longrightarrow \LL^{\gr}\big(L(E)\big),\\ 
 H &\longmapsto \langle H\rangle. \notag 
 \end{align} 

\begin{rem}\label{rem:infideal}
There is also a structure theorem for graded ideals of $L(E)$ that holds for an arbitrary, not necessarily row-finite graph $E$, see \cite{lpabook}*{Theorem 2.5.8}.
\end{rem}

\section{Elliott's \topdf{$K$}{K}-theoretic classification of ultramatricial algebras }\label{sec:elliott}

\subsection{The Grothendieck group}\label{subsec:k0}
\numberwithin{equation}{subsection}
Let $A$ be a ring; write $\Idem(A)$ for the set of idempotent elements in $A$. We say that two elements $p,q\in\Idem(A)$ are \emph{orthogonal}, and write $p\perp q$, if $pq=qp=0$; in that case $p+q$ is again an idempotent. We say that $e,f\in \Idem(A)$ are \emph{Murray-von Neumann equivalent} and write $e\sim f$, if there are elements $x\in eAf$, $y\in fAe$ such that $xy=e$ and $yx=f$. Borrowing terminology from \cite{B13}*{Cap\'\i tulo 7, Secci\'on 1}, we say that the ring $A$ has \emph{enough room} if given $p,q\in \Idem(A)$ there are $p',q'\in\Idem(A)$ with $p'\sim p$, $q'\sim q$ and $p'\perp q'$. For example, if $R$ is unital, then the non-unital ring 
\[
\M_\infty (R)=\bigcup_{n\ge 1}\M_n(R)
\]
of finite matrices of arbitrary size with coefficients in $R$ has enough room.  If $A$ has enough room then the quotient set $\Idem(A)/\sim$ becomes a monoid under orthogonal sum: to add the classes $[p]$ and $[q]$, one finds orthogonal representatives $p'\in [p]$ and $q'\in [q]$ and sets $[p]+[q]=[p'+q']$. In particular, for $\Idem_\infty(R)=\Idem(\M_\infty (R))$,  $\cV(R)=\Idem_\infty(R)/\sim$ is an abelian monoid.  An element $p\in\Idem_\infty R$ determines an $R$-endomorphism of $R^{(\N)}=\bigoplus_{n\ge 1}R$ and its image $pR^{(\N)}$ is a finitely generated projective right $R$-module. 
The assignment $p\to pR^{(\N)}$ induces an isomorphism of monoids 
\begin{equation}\label{map:vvv}
\cV(R)\iso \mathbb P^{\mathrm{iso}}(R)    
\end{equation}
with the set of isomorphism classes of finitely generated projective right $R$-modules, equipped with the direct sum. The inclusion of the category of abelian groups into that of abelian monoids has a left adjoint $M\mapsto M^+$, called \emph{group completion}. The \emph{Grothendieck group} of $R$ is $K_0(R)=\cV(R)^+$. Thus $K_0(R)$ comes equipped with a monoid homomorphism 
$\iota:\cV(R)\to K_0(R)$ that is universal in the sense that if $G$ is another abelian group and $f:\cV(R)\to G$ is a monoid homomorphism, then there exists a unique group homomorphism $\bar{f}$ making the following diagram commute
\[
\xymatrix{\cV(R)\ar[rr]^{\iota}\ar[dr]^f&& K_0(R)\ar@{.>}[dl]^{\bar{f}}\\
           &G&}
\]
Consider the submonoid $K_0(R)_+=\iota(\cV(R))$; define a pre-order between elements of $K_0(R)$ by $x\le y$ if $y-x\in K_0(R)_+$. This pre-order is compatible with the group structure in the sense that $x\le y\Rightarrow x+z\le y+z$. Thus $(K_0(R),\le)$ is a pre-ordered group. The class $[1]$ of the unit (which corresponds to that of the free module $R$ of rank one under the isomorphism \eqref{map:vvv}) is an \emph{order unit} of $K_0(R)$ in the sense that for any $x\in K_0(R)$ there exists $n\in\N$ such that $x\le n[1]=[I_n]$; here $I_n\in M_n(R)$ is the identity matrix. Thus $K_0(R)$ is a \emph{pointed pre-ordered} abelian group. 

\begin{rem}\label{rem:lemono}
We saw in Subsection \ref{subsec:mono} that any abelian monoid is canonically pre-ordered. In particular, $\cV(R)$ is pre-ordered and $[1]$ is an order unit with respect to this pre-order. The canonical monoid homomorphism $\iota:\cV(R)\to K_0(R)$ is a morphism of pre-ordered monoids. 
\end{rem}

\begin{rem}\label{rem:elepropk0}
 Let $R,S$ be Morita equivalent unital rings and let ${}_RP_S$ be a bimodule implementing the equivalence. Then $-\otimes_RP$ induces an isomorphism
 $\cV(R) \cong  \mathbb P^{\mathrm{iso}}(R) \iso  \mathbb P^{\mathrm{iso}}(S) \cong \cV(S) $ and therefore an isomorphism $K_0(R)\iso K_0(S)$. In particular, $K_0$ is Morita invariant. Similarly, it is straightforward to see that if $A$ and $B$ are unital rings, then the canonical map $\cV(A\times B)\to \cV(A)\times\cV(B)$ is an isomorphism. Hence $K_0$ commutes with finite sums. Let now $\{\sigma_n:R_n\to R_{n+1}\,\colon\, n\ge 1\}$ be an inductive system of unital ring homomorphisms, and let $R=\varinjlim_n R_n$. It is straightfoward to show that  $\cV(R)=\varinjlim_n \cV(R_n)$. It follows from this and the universal property of group completion that $K_0(R)=\varinjlim_n K_0(R_n)$; thus $K_0$ commutes with inductive limits.
\end{rem}
\begin{ex}\label{ex:k0easy}
Let $R$ be a unital ring. There is a canonical map $\N_0\to  \mathbb P^{\mathrm{iso}}(R)$, $n\mapsto [R^n]$. This map is an isomorphism, for instance, when $R$ is a division ring, a local ring or a PID; in particular $K_0(R)=\Z$ in all these cases. Moreover by Remark \ref{rem:elepropk0}, in all these cases, if $r,n_1,\dots,n_r\in\N$, then 
$K_0(\M_{n_1}(R)\oplus\dots \oplus\M_{n_r}(R))=\Z^r$.
\end{ex}

\begin{rem}\label{rem:k0nuni}
The definition of the Grothendieck group extends to not necessarily unital rings as follows. The \emph{unitalization} of a ring $A$ is $\tilde{A}=A\oplus\Z$ equipped with the product $(a,n)(b,m)=(ab+nb+ma,nm)$. Remark that $\tilde{A}$ is a unital ring and that the projection $\pi:\tilde{A}\to \Z$ is a unital ring homomorphism, so that $K_0(\pi)$ is defined. Set $K_0(A)=\ker(K_0(\pi))$. One shows that this definition agrees with the previous one in the unital case, and that $K_0$ is split-exact, so that if $A$ embeds as a two-sided ideal in a ring $R$ and the projection $p:R\to R/A$ is split by a ring homomorphism, then $K_0(A)\cong \ker(K_0(p))$ (see \cite{rosen}*{Chapter 1, Section 5}). 
It follows from this that $K_0$ still commutes with inductive limits when considered as a functor on the category of nonunital rings and ring homomorphisms. Hence, by Example \ref{ex:k0easy}, if
$R$ is a division ring or a PID and $n_1,n_2,\dots\in\N$ is an infinite sequence, then 
\[
K_0\big (\bigoplus_{i=1}^\infty \M_{n_i}(R)\big )=\Z^{(\N)}.
\]
\end{rem}

\subsection{Ultramatricial algebras}\label{subsec:ultra}
A $\K$-algebra $R$ is \emph{matricial} if there exist $r,n_1,\dots n_r\in\N$ and an isomorphism of $\K$-algebras 
$$R\cong\bigoplus_{i=1}^r\M_{n_i}(\K).$$ 
We say that an algebra $R$ is \emph{ultramatricial} if there is a direct system $\{\sigma_n:R_n\to R_{n+1}\}$ of matricial algebras and not-necessarily unital algebra homomorphisms such that $R\cong\varinjlim_n R_n$. An ultramatricial algebra is unital if and only if it can be written as an inductive limit of matricial algebras and unital homomorphisms. 

\begin{rem}\label{rem:locamat}
Ultramatricial algebras are a particular case of \emph{locally matricial} algebras \cite{lpabook}*{Definition 2.6.19}. An algebra $A$ is locally matricial if there exists an upwards directed partially ordered set $I$ and a system $\{\sigma_{i,j}:A_i\to A_j\,\colon\, i\le j\in I\}$ of matricial algebra homomorphisms such that $A=\varinjlim_IA_i$. If $A$ admits such a description for some countable index set $I$, then $A$ is ultramatricial. This follows from the fact that any countable upwards directed set admits a (countable) cofinal chain.
\end{rem}
\begin{ex}\label{ex:ultrale}
Let $E$ be a countable graph. Then $L(E)$ is matricial if and only if $E$ is finite and acyclic \cite{lpabook}*{Theorem 2.6.17}, in which case it is a direct sum of matrix algebras indexed by $\sink(E)$, and thus $K_0(L(E))=\Z^{(\sink(E))}$. If $E$ is infinite acyclic then $L(E)$ is ultramatricial by \cite{lpabook}*{Proposition 2.6.20} and Remark \ref{rem:locamat}. For example, if $E$ is row-finite, then the covering graph $\ol {E}$ of \eqref{eq:olE} is acyclic. Let $E_n$ be the graph with 
\[
E_n^0=\{v_i\,\colon\, -n-1\le i\le n\},\quad E_n^1=\{e_i\,\colon\, -n\le i\le n\}
\]
equipped with the restriction of the source and range maps of $\ol{E}$. Then $L(\ol E)=\varinjlim_n L(E_n)$. Every vertex of $E_n$ connects to a sink and 
\[
\sink(E_n)=\{v_i\,\colon\, v\in\sink(E)\,\, -n-1\le i\le n\}\sqcup\{v_n\,\colon\, v\in\reg(E)\}.
\]

Hence by  Example \ref{ex:k0easy} and Remark \ref{rem:k0nuni},
\[
K_0(L(E_n))=\Z^{(\bigsqcup_{i=-n-1}^n\sink(E)_i)}\oplus \Z^{(\reg(E))}.
\]
One checks that the map $K_0(L(E_n))\to K_0(L(E_{n+1}))$ is the inclusion on the first summand and left multiplication by $A_E^t$ on the second. Thus if $E$ is regular, $K_0(L(\ol{E}))$ is the inductive limit of the system 
\begin{equation}\label{seq:colibare}
\xymatrix{\Z^{E^0}\ar[r]^{A_E^t}&\Z^{E^0}\ar[r]^{A_E^t}&\dots}    
\end{equation}
For example, setting $E=\cR_n$ we obtain that $K_0(L(\ol{\cR_n}))=\Z[1/n]$.
\end{ex}

\begin{ex}\label{ex:kle0}
Let $E$ be a row-finite graph and $L(E)_0$ the degree-zero component of $L(E)$. For each $v\in E^0$ and $i\ge 0$, let $E^iv$ be the set of all paths of length $i$ ending in $v$. The subspace of $L(E)_0$ spanned by those $\alpha\beta^*$ with $\alpha,\beta\in E^i$ and $r(\alpha)=r(\beta)=v$ is a subalgebra isomorphic to $\M_{E^iv}(\K)$. The subspace $L(E)_{0,n}\subset L(E)_0$ generated by all $\alpha\beta^*$ with $\alpha,\beta\in E^i$, with $r(\alpha)=r(\beta)$ and $i\le n$ is also a subalgebra and we have an isomorphism
\[
L(E)_{0,n}\cong \left(\bigoplus_{v\in \sink(E)}\bigoplus_{i=0}^n \M_{E^iv}(\K)\right)\oplus\bigoplus_{v\in\reg(E)}\M_{E^nv}(\K).
\]
Hence by  Example \ref{ex:k0easy},
\[
K_0(L(E)_{0,n})=\left(\bigoplus_{v\in\sink(E)}\Z^{n+1}\right)\oplus\Z^{(\reg(E))}.
\]
At the $K_0$ level, the map induced by the inclusion $L(E)_{0,n}\subset L(E)_{0,n+1}$ is multiplication by  $A^t_E$ on $\Z^{\reg(E)}$ and the inclusion on the summand corresponding to sinks. In particular, if $E$ is regular, $K_0(L(E)_0)$ is the inductive limit of the system \eqref{seq:colibare}. Thus for regular $E$, we have  $K_0(L(\ol{E}))=K_0(L(E)_0)$. 
\end{ex}

The family of classification problems and results that this article is about started with a 1976 paper of Elliott~\cite{elliott}. Elliott's Theorem, as described by Goodearl in the language of $K$-theory, is as follows.

\begin{thm}[\cite{goodearl}*{Theorem 15.26 and Corollary 15.27}]\label{thm:elliott}
Let $A$ and $B$ be unital ultramatricial $\K$-algebras. Then
\begin{enumerate}[\upshape(1)]

\item $A$ and $B$ are Morita equivalent  if and only if $K_0(A) \cong K_0(B)$ as partially ordered abelian groups.

\item $A$ and $B$ are isomorphic if and only if $K_0(A)\cong K_0(B)$ as pointed partially ordered abelian groups. 

\end{enumerate}
\end{thm}

\section{Algebraic Kirchberg-Phillips problem}\label{sec:KP}

\subsection{Simplicity and pure infinite simplicity}\label{subsec:simple}
Let $R$ be a nonzero unital ring. We say that $R$ is \emph{simple} if for every nonzero element $x\in R$, there exist $n\ge 1$ and $a,b\in R^n$ such that 
\[
\sum_{i=1}^na_ixb_i=1.
\]
For example, any division ring is simple. We say that $R$ is \emph{simple purely infinite} (SPI) if it is not a division ring and for every nonzero $x\in R$, there exist $a,b\in R$ such that
\[
axb=1. 
\]
\begin{rem}\label{rem:simpnuni}
Simplicity and pure infinite simplicity are also defined for nonunital rings, see \cite{lpabook}*{page 68 and Definition 3.1.8}.
\end{rem}
\begin{thm}[Ara, Goodearl and Pardo \cite{agop}*{Proposition 2.1 and Corollary 2.2}]\label{thm:agop0} 
Let $R$ be an SPI unital ring. Then the subsemigroup $\cV(R)^*=\cV(R)\setminus\{0\}\subset\cV(R)$ is a group, and the monoid homomorphism $\cV(R)\to \cV(R)^*$ defined as the identity on $\cV(R)^*$ is a group completion. In particular $K_0(R)=\cV(R)^*$.
\end{thm}
\begin{rem}\label{rem:agop}
It follows from Theorem \ref{thm:agop0} that if $R$ is SPI, then $K_0(R)_+=K_0(R)$. In particular, if $G$ is any partially ordered abelian group, then any group homomorphism $G\to K_0(R)$ is order-preserving. 
\end{rem}
\begin{ex}
By \eqref{eq:lnspi}, if $n\ge 2$, then the Leavitt algebra $L_n$ is SPI. In general, the simplicity and pure infinite simplicity of a Leavitt path algebra $L(E)$ over a field $\K$ are characterized in terms of certain graphic conditions on $E$ \cite{lpabook}*{Theorems 2.9.7 and 3.1.10}. To explain them, we need some notation and vocabulary.
An edge $f$ is an \emph{exit} of a cycle $\alpha=e_1\cdots e_n$ if there exists $1\le i \le n$ such that $s(f)=s(e_i)$ and $f\ne e_i$.
The algebra $L(E)$ is simple if and only if $E$  satisfies the following two conditions \cite{lpabook}*{Theorem 2.9.7}

 \begin{enumerate}[\upshape(1)]

\item Every cycle of $E$ has an exit.
\item Every vertex in $E$ connects to every singular vertex (i.e., every sink or infinite emitter) and every infinite path.

By \cite{lpabook}*{Theorem 3.1.10}, $L(E)$ is SPI if and only if $E$ satisfies (1), (2) and

\item $E$ has at least one cycle. 

\end{enumerate}

We say that $E$ is \emph{simple} if it satisfies (1) and (2). A simple graph $E$ is \emph{SPI} if it satisfies (3). The graph $C^*$-algebra $C^*(E)$ is simple (respectively SPI) if and only if $E$ is \cite{dritom2}*{Corollaries 2.14 and 2.15}.
For example, if $n\ge 2$, then the graph $\cR_n$ satisfies (1)--(3), so as already mentioned, $L_n=L_n(\K)$ is SPI as is the Cuntz algebra $\cO_n$.
\end{ex}

\subsection{Higher algebraic and topological \topdf{$K$}{K}-theory}\label{subsec:hik}
For any ring $A$ there are defined \emph{algebraic $K$-theory} groups $K_n(A)$ for any $n\in\Z$ \cite{rosen}. For example if $A$ is unital, $K_1(A)=\Gl(A)_{\ab}$ is the abelianization of the group of invertible matrices of arbitrary finite size. If $\fA$ is a $C^*$-algebra, then $\Gl(\fA)$ is a topological group, so one can consider its homotopy groups $\pi_n\Gl(\fA)$ $(n\ge 0)$.   The Bott periodicity theorem says that the latter are periodic of period two; we have $\pi_{2n}\Gl(\fA)=\pi_{0}(\Gl(\fA))$ and $\pi_{2n+1}(\fA)=K_0(\fA)$ for all $n\ge 0$. The \emph{topological $K$-theory} groups $K^{\top}_n(\fA)$, $n\in\Z$, are defined to be 
\begin{equation}
K_n^{\top}(\fA)=\left\{\begin{matrix}
K_0(\fA)& n \text{ even, }\\
\pi_0\Gl(\fA)& n \text{ odd.}
\end{matrix}\right.    
\end{equation}
There is a natural map 
\begin{equation}\label{map:compak}
    c_n:K_n(\fA)\to K_n^{\top}(\fA)
\end{equation} 
defined for all $n\in\Z$. It is the identity for $n=0$ and surjective for $n=1$, but it is not an isomorphism in general. For example $K_1(\C)=\C^*$ and $K_1^{\top}(\C)=0$. 

\begin{ex}\label{ex:keasy}
There is a canonical map $R_{\ab}^*=\Gl_1(R)_{\ab}\to \Gl(R)_{\ab}=K_1(R)$. The latter map is an isomorphism if $R$ is a division ring, or a local ring, or a PID. If $R$ is regular supercoherent (e.g. if it is noetherian regular) then $K_n(R)=0$ for all $n\le 0$.
\end{ex}

\begin{ex}\label{ex:k0bf}  
Let $E$ be a graph, and let $A_E$ and $I$ be as in \eqref{eq:ae} and \eqref{eq:I}. Put 
\[
\BF(E)=\coker(I-A_E^t).
\]
Then we have
\begin{equation}\label{eq:kc*e}
K_0(C^*(E))=\BF(E),\,\, K_1^{\top}(C^*(E))=\ker(I-A_E^t).
\end{equation}

Remark that if $E$ is finite and regular, then $I-A_E^t$ is a square matrix, so $\ker(I-A_E^t)$ is free of rank $r=\rk(\BF(E))$. Hence $\BF(E)$ contains all the information about  $K^{\top}_*(C^*(E))$.

We also have
\begin{equation}\label{eq:k01le}
K_0(L(E))=\BF(E),\, K_1(L(E))=\BF(E)\otimes \K^*\oplus\ker(I-A_E^t).
\end{equation}
Moreover if $E^0$ is finite then the isomorphisms for $K_0$ above carry $[1_{C^*(E)}]$ and $[1_{L(E)}]$ to the class of
\[
1_E:=\sum_{v\in E^0}v.
\]
\end{ex}

For all $n\in\Z$, there is a short exact sequence
\begin{equation}\label{seq:kle}
0 \to \BF(E)\otimes K_{n}(\K)\to K_n(L(E))\to \ker((I-A_E^t)\otimes K_{n-1}(\K))\to 0.    
\end{equation}
In particular $K_n(L(E))=0$ for $n<0$. The formulas \eqref{eq:kc*e} were proven in \cite{ck2} for finite regular graphs and then extended in \cite{kpr} for row-finite graphs; the general case was done in \cite{dritom2}, using desingularization. The latter is a method for replacing any countable graph by a row-finite one with Morita equivalent $C^*$ and Leavitt path algebras. The exact sequence \eqref{seq:kle} for Leavitt path algebras over a field appears already in \cite{arajazz}; it was obtained for all unital coefficient rings and row-finite graphs in \cite{abc}. The general case follows from the latter using graph 
desingularization.

\begin{ex}\label{ex:klecomp}
Next we use \eqref{seq:kle} to compute $K_*(L(E))$; we assume that $E$ is finite and regular and that $\det(I-A_E^t)\ne 0$. In this case $\BF(E)$ is finite and 
\[
0\to \Z^{E^0}\overset{I-A_E^t}{\lra}\Z^{E^0}\to \BF(E)\to 0
\]
is a free resolution. In particular 
\begin{equation}\label{eq:kktor}
\ker((I-A_E^t)\otimes K_{n-1}(\K))=\tor_1^\Z(\BF(E), K_{n-1}(\K)).    
\end{equation}

If for example $\K$ is algebraically closed and of characteristic zero, then $K_n(\K)$ is a divisible group for all $n\ge 1$, torsionfree if $n$ is even and with torsion $\Q/\Z$ if $n$ is odd 
\cite{chuk}*{Chapter VI, Theorem 1.6}. Recall that if $M$ is a torsion abelian group, then $M\otimes D=0$ for any divisible group $D$, and $\Tor_1^\Z(M,\Q/\Z)=M$. It follows from this together with \eqref{eq:kktor} and 
\eqref{seq:kle} that $K_n(L(E))=\BF(E)$ for even nonnegative $n$ and zero otherwise. In particular, $K_n(L(E))=K_n(C^*(E))$ for nonnegative $n$. For another example, consider the case $\K=\F_q$, 
the finite field of $q$ elements. In this case we have $K_{2i}(\K)=0$ for $i>0$ and $K_{2i-1}(\K)=\Z/(q^i-1)$ \cite{chuk}*{Chapter IV, Corollary 1.13}, so we get 
$K_{2i-1}(L(E))=\coker(q^i-1:\BF(E)\to \BF(E))$ for $i\ge 1$ and  $K_{2i}(L(E))=\ker(q^i-1:\BF(E)\to \BF(E))$ for $i\ge 0$. In particular if $q=p^r$ with $p$ prime and $\BF(E)$ is a $p$-group, then $K_n(L(E))=0$ for all nonzero $n$. 
\end{ex}

\subsection{Kasparov's bivariant \topdf{$K$}{K}-theory}\label{subsec:KK}

Let $C^*$ be the category of separable $C^*$-algebras and
$*$-homomorphisms. A $*$-homomorphism $H:\fA\to\fB[0,1]$ is called a \emph{homotopy} between $\ev_0\circ H$ and $\ev_1\circ H$. Two $*$-homomorphisms $\phi,\psi:\fA\to\fB$ are \emph{homotopic} if there is a homotopy between them, in which case we write $\phi\sim\psi$. Let $\cC$ be a category. A functor 
\begin{equation}\label{map:fun}
 F:C^*\to \cC   
\end{equation} 
is \emph{homotopy invariant} if it sends homotopic maps to the same map. 

Let $\ell^2$ be the Hilbert space of $2$-summable sequences of complex numbers and $\cB=\cB(\ell^2)$ the $C^*$-algebra of bounded linear operators. The ideal $\cK\triqui\cB$ of \emph{compact operators} is the norm closure of the ideal of finite rank operators $\ell^2\to \ell^2$. For example $\epsilon_{i,j}:\ell^2\to \ell^2$, $\epsilon_{i,j}(a)_s=\delta_{s,i}a_j$ is an operator of finite rank, and thus a compact operator. The algebra $\cK$ is an example of a \emph{nuclear} $C^*$-algebra; for any $\fA\in C^*$, there is a unique norm $||\,||$ on the algebraic tensor product  $ \fA\otimes_\C\cK$ satisfying \eqref{eq:c*norm} and  such that $||x\otimes a||=||x||\cdot ||a||$ for all $x\in\cK$ and $a\in \fA$. For general $\fA,\fB$ there may be infinitely many such norms; we write $\fA\sotimes\fB$ for the completion of $\fA\otimes_\C\fB$ with respect to the smallest one. We have a $\ast$-homomorphism $\iota_\fA:\fA\to \cK\sotimes\fA$, $\iota_\fA(a)=\epsilon_{1,1}\otimes a$. A functor \eqref{map:fun} is \emph{stable} if $F(\iota_\fA)$ is an isomorphism for every $\fA$. 

An \emph{extension}
of $C^*$-algebras is a sequence of $*$-homomorphisms
\begin{equation}\label{ext:c*}
0\to \fA\to \fB\to \fC\to 0    
\end{equation}
which is exact as a sequence of vector spaces and which admits a completely positive splitting (for a definition of completely positive see \cite{david}*{p. 6}). If \eqref{map:fun} takes values in a triangulated category $\cC$, and $C\mapsto C[-n]$ is the $n$-fold suspension in $\cC$, then $F$ is \emph{excisive} if it sends any extension \eqref{ext:c*} to a distinguished triangle
\[
F(\fC)[+1]\overset{\partial}{\lra} F(\fA)\to F(\fB)\to F(\fC).
\]
The map $\partial$ is required to depend naturally on \eqref{ext:c*} and to satisfy certain additional compatibility conditions. 

There exist a triangulated category $KK$, whose objects are those of $C^*$, and a functor $k:C^*\to KK$, defined as the identity on objects, that is homotopy invariant, stable and excisive. Furthermore $k$ is universal initial with respect to the latter properties, which roughly means that if \eqref{map:fun} is another homotopy invariant, stable and excisive functor, then there is a unique triangulated functor $\bar{F}:KK\to\cC$ making the following diagram commute
\[
\xymatrix{C^*\ar[r]^k\ar[d]_F& KK\ar@{.>}[dl]^{\bar{F}}\\
\cC}
\]
We write 
\[
KK_n(\fA,\fB)=\hom_{KK}(k(\fA),k(\fB)[n]), \,\, KK(\fA,\fB)=KK_0(\fA,\fB).
\]
A major feature of Kasparov's theory is that setting the first variable equal to $\C$ recovers topological $K$-theory; we have
\begin{equation}\label{eq:kkca=ka}
KK_n(\C,\fA)=K_n^{\top}(\fA).     
\end{equation}

In this bivariant context, the Bott periodicity theorem says that there is a natural isomorphism $k(\fA)[+1]\cong k(\fA)[-1]$ for all $\fA\in C^*$. For more on $KK$, see \cite{cmr}.

\subsection{The Kirchberg-Phillips theorem}\label{subsec:KP}

Let $\fA$ and $\fB$ be separable $C^*$-algebras. Using composition in the category $KK$ and the isomorphism \eqref{eq:kkca=ka}, one obtains an evaluation map
\begin{equation}\label{map:evKK}
\ev: KK(\fA,\fB)\to 
\bigoplus_{i=0}^1\Hom(K^{\top}_i(\fA),K^{\top}_i(\fB))
\end{equation}
There is also a map \cite{roscho}
\begin{equation}\label{map:brown}
 \ker(\ev)\to \bigoplus_{i=0}^1\Ext_1^{\Z}(K^{\top}_i(\fA),K^{\top}_{i+1}(\fB)).
\end{equation}
The pair $(\fA,\fB)$  satisfies the \emph{Universal Coefficient Theorem} (UCT) if $\fA$ is nuclear, \eqref{map:evKK} is surjective and \eqref{map:brown} is an isomorphism. It follows that in this case we have an exact sequence 
\begin{equation}\label{seq:uct}
0\to \bigoplus_{i=0}^1\Ext_1^{\Z}(K^{\top}_i(\fA),K^{\top}_{i+1}(\fB))\to KK(\fA,\fB)\to \bigoplus_{i=0}^1\Hom(K^{\top}_i(\fA),K^{\top}_i(\fB))\to 0.
\end{equation}
A nuclear separable $C^*$-algebra $\fA$ is in the \emph{UCT class} if $(\fA,\fB)$ satisfies the UCT for every separable $C^*$-algebra $\fB$. The UCT class turns out to be quite large \cite{roscho}.

\begin{rem}\label{rem:uctis}
If both $\fA$ and $\fB$ are in the UCT class, then $k(\fA)\cong k(\fB)$ if and only if $K^{\top}_*(\fA)\cong K^{\top}_*(\fB)$ 
\cite{roscho}*{Proposition 7.3}.
\end{rem}

\begin{ex}\label{ex:uctc*e}
 The $C^*$-algebra of a countable graph $E$ is in the UCT class. Hence it follows from \eqref{eq:kc*e} and \eqref{seq:uct} that if $E$ and $F$ are countable graphs, then there is an exact sequence
 \begin{multline}\label{seq:kkcecf}
0\to \Ext^1_\Z(\BF(E),\ker(I-A_F^t))\to KK(C^*(E),C^*(F))\to \\\hom(\BF(E),\BF(F))\oplus \hom(\ker(I-A_E^t),\ker(I-A_F^t))\to 0
\end{multline}
Thus for example, if $E$ and $F$ are finite and regular, $\det(I-A_E^t)\ne 0$ and $r=\rk(\BF(F))$, then the above sequence simplifies to
\[
0\to \hom(\BF(E),\Q/\Z)^r\to KK(C^*(E),C^*(F))\to \hom(\BF(E),\BF(F))\to 0.
\]
\end{ex}

A $C^*$-algebra $\fA$ is a \emph{Kirchberg} $C^*$-algebra if it is unital, SPI, separable, nuclear and in the UCT class. 

\begin{ex}\label{ex:spic*}
 Let $E$ be a countable graph with finitely many vertices. Then the graph $C^*$-algebra $C^*(E)$ is separable, unital, nuclear and in the UCT class. It is SPI if and only if $E$ is 
 \cite{dritom}*{Corollaries 2.14 and 2.15}. 
\end{ex}

We are in a position to state the deep result of Kirchberg and Phillips on the $K$-theoretic classification of Kirchberg algebras. 

\begin{thm}[Kirchberg-Phillips \cite{P}*{Theorem 4.2.4}]\label{thm:KP}
 Let $\fA$ and $\fB$ be Kirchberg $C^*$-algebras. Assume there is an isomorphism of graded abelian groups $\xi:K_*(\fA)\to K_*(\fB)$ such that 
 $\xi([1_\fA])=[1_{\fB}]$. Then there is a $\ast$-isomorphism $\phi:\fA\iso \fB$ such that $K_*(\phi)=\xi$.
\end{thm}
\begin{rem}\label{re:k=kto} In Theorem \ref{thm:KP} above, as throughout this paper, $K_*$ stands for algebraic $K$-theory. The natural map \eqref{map:compak} is an isomorphism for any properly infinite unital $\ast$-algebra $\fA$ \cite{chriswi}*{Theorem 3.2} and in particular for any Kirchberg algebra. Thus $K_*^{\top}$ may be substituted for $K_*$ in Theorem \ref{thm:KP}. 
\end{rem}

Specializing to the case of graph $C^*$-algebras we have the following complete classification result. 

\begin{coro}\label{coro:KPR}
Let $E$ and $F$ be countable SPI graphs with finitely many vertices. Assume there exist group isomorphisms $\xi_0:\BF(E)\iso \BF(F)$ and $\xi_1:\ker(I-A_E^t)\iso \ker(I-A_F^t)$ such that $\xi_0([1_E])=[1_F]$. Then $C^*(E)\cong C^*(F)$ as $C^*$-algebras.
\end{coro}

\begin{rem}\label{rem:ror}
As explained in Example \ref{ex:k0bf}, for finite regular graphs $E$ and $F$, $\BF(E)\cong \BF(F)$ implies that $\ker(I-A_E^t)\iso \ker(I-A_F^t)$. Hence because finite SPI graphs are regular, the particular case of Corollary \ref{coro:KPR} when both $E$ and $F$ are finite says that if $\xi:\BF(E)\iso \BF(F)$ is an isomorphism such that $\xi([1_E])=[1_F]$, then $C^*(E)\cong C^*(F)$ as $C^*$-algebras. This was proved independently by R\o rdam in \cite{ror}. We give a sketch of the proof of this result using the notion of Cuntz splice after Theorem~\ref{thm:splice}. 
\end{rem}

We conclude this section with one of the main open questions in the theory of Leavitt path algebras (as of 2024). The following is referred to in \cite{lpabook}*{Section~ 7.3.1} as 
``What is generally agreed to be the most compelling unresolved question in the subject of Leavitt path algebras". 

\begin{quest}[\cite{question}]\label{quest:flow}
Let $\K$ be any field and $L=L_{\K}$ the Leavitt path algebra over $\K$.  Let $E$, $F$  be finite SPI graphs. Assume there exists an isomorphism $\xi:\BF(E)\iso \BF(F)$ such that $\xi([1_E])=[1_F]$. Does it follow that $L(E)\cong L(F)$?
\end{quest}

\begin{rem}\label{rem:flow} Ruiz and Tomforde proved \cite{ruto}*{Proposition 10.4} that if $E$ and $F$ are both SPI with finitely many vertices and infinitely many edges, then $L_\K(E)\cong L_\K(F)$ as rings if and only if an isomorphism $\xi$ as in Question \ref{quest:flow} exists and, in addition, $|\sing(E)|=|\sing(F)|$. Recall that $\sing(E)$ is the set of all sinks and infinite emitters (i.e. all singular vertices) of the graph $E$. They also showed that under additional hypothesis on $\K$ the latter condition can be replaced by the condition that $K_1(L_\K(E))\cong K_1(L_\K(F))$. However, in \cite{ruto}*{Example 11.2} they gave an example of vertex-finite $E$ and $F$ such that an isomorphism $\xi:\BF(E)\iso \BF(F)$ as in the question above exists, and such that furthermore $K_1(L_\Q(E))\cong K_1(L_\Q(F))$, but $K_2(L_\Q(E))\not\cong K_2(L_\Q(F))$. In particular $L_\Q(E)$ and $L_\Q(F)$ are not Morita equivalent, and thus of course not isomorphic. 
\end{rem}

\subsection{Tensor products of Leavitt path algebras}\label{subsec:tenso}
Perhaps the first thing that comes to mind as a strategy to address Question \ref{quest:flow} is to try to algebraize the original proof of Theorem \ref{thm:KP}. The latter relies heavily on the following absorption theorem.

\begin{thm}[Kirchberg \cite{KP}*{Theorems 3.8 and 3.15}]\label{thm:geneva}
 Let $\fA$ be a separable unital nuclear $C^*$-algebra. If $\fA$ is simple, then there is a $*$-isomorphism $\cO_2\sotimes \fA\cong \cO_2$. If furthermore $\fA$ is SPI, then there is a
 $*$-isomorphism $\cO_\infty\sotimes\fA\cong\fA$.
\end{thm}

The analogue of Theorem \ref{thm:geneva} with $L_n$ substituted for $\cO_n$ and the algebraic tensor product of $\K$-vector spaces $\otimes$ substituted for $\sotimes$ fails to hold. In fact we have

\begin{thm}[Ara and Corti\~nas \cite{aratenso}*{Theorem 5.1 and Proposition 5.3}]\label{thm:aratenso}
\hfill

\begin{enumerate}[\upshape(1)]
\item Let $E_1,\dots,E_m$ and $F_1,\dots,F_n$ be finite graphs. Assume that $m\ne n$ and that each of the $E_i$ and of the $F_j$ has at least one nontrivial closed path. Then the algebras $L(E_1)\otimes\dots\otimes L(E_m)$ and $L(F_1)\otimes\dots\otimes L(F_n)$ are not Morita equivalent.

\item Let $E$ be a finite graph with at least one nontrivial closed path. Then $L_\infty\otimes L(E)$ and $L(E)$ are not Morita equivalent.
\end{enumerate}

\end{thm}

\begin{proof}[Sketch of the proof of Theorem \ref{thm:aratenso}] One first shows \cite{aratenso}*{Corollary 4.5} that if $E$ is finite and has at least one closed path, then its Hochschild homology groups $HH_i(L(E))$ are nonzero for $i=0,1$ and vanish in higher dimension. Hence by the K\"unneth formula, the nonvanishing range for $HH_i$ of a tensor product of $m$ such algebras is precisely $0\le i\le m$. Since $HH_i$ is invariant under Morita equivalence, we get (1). Part (2) is proved similarly, writing $L_\infty$ as an inductive limit of Leavitt path algebras of finite graphs, and using that Hochschild homology commutes with such limits. 
\end{proof}

\subsection{Flow invariants and the Cuntz splice}\label{subsec:flowsplice}

Theorem \ref{thm:aratenso} shows that the original argument of Theorem \ref{thm:KP} does not adapt to the purely algebraic case. R\o rdam's proof of Corollary \ref{coro:KPR} for finite graphs uses a different argument which at least in part, translates nicely to the algebraic case. 

The starting point of the argument is Theorem \ref{thm:franks} below. To state it, we need some vocabulary.

Whenever a graph $E$ has vertices $v,w$ such that there is a single edge $e$ with $s(e)=v$ and $r(e)=w$, one can replace $E$ by a graph $E'$ where $v$ and $w$ are identified and $e$ has been contracted to the zero length path $v$; this graph alteration is called a \emph{contraction} and its inverse an \emph{expansion}. For example a graph consisting of a single closed path can be contracted to a single vertex with no edges after a finite number of contractions. There are also other, more complicated graph moves, called \emph{in-splitting} and \emph{out-splitting}, and their inverses, \emph{in-amalgamation} and \emph{out-amalgamation} \cite{flow}; they will be discussed in Section \ref{sec:symb} along with the related notion of strong shift equivalence. We say that two finite essential graphs $E$ and $F$ are \emph{flow equivalent} if one can get from one to the other via a finite sequence of splittings, amalgamations, expansions and contractions. The \emph{trivial flow class} is the flow equivalence class of the graph with a single vertex and no edges; its elements are the graphs consisting of a single closed path. A graph $E$ is \emph{strongly connected} if for any $v,w\in E^0$ there is a path $\alpha$ with $s(\alpha)=v$ and $r(\alpha)=w$.
Remark that an essential graph that is not in the trivial flow class is strongly connected if and only if it is SPI.

\begin{thm}[Franks \cite{franks}*{Theorem}]\label{thm:franks}
 Let $E$ and $F$ be finite, essential SPI graphs. Then the following are equivalent.
\begin{enumerate}[\upshape(1)]
 
\item $E$ and $F$ are flow equivalent; 
\item $\BF(E)\cong \BF(F)$ and $\det(I-A_E)=\det(I-A_F)$. 

\end{enumerate}

\end{thm}

Here is the first application of Franks' theorem to the classification of Leavitt path algebras.

\begin{thm}[Abrams, Louly, Pardo and Smith \cite{flow}*{Theorems 1.25 and 2.5}]\label{thm:flow}\,\negthickspace
Let $E$ and $F$ be as in Theorem \ref{thm:franks}. If condition (2) of Theorem \ref{thm:franks} holds, then $L(E)$ and $L(F)$ are Morita equivalent. If moreover there is an isomorphism $\BF(E)\iso \BF(F)$ such that $[1_E]\mapsto [1_F]$, then $L(E)\cong L(F)$.
\end{thm}

\begin{proof}[Sketch of the proof] We sketch the proof of the first assertion; for the proof of the second, which uses some results of Huang \cite{huang}, see \cite{flow}*{Theorem 2.5}. In \cite{flow} the authors first show that each of the graph moves that define flow equivalence yield Morita equivalent Leavitt path algebras. Then they apply the implication (2)$\Rightarrow$ (1) of Theorem \ref{thm:franks}. 
\end{proof}

The analogue of Theorem \ref{thm:flow} for graph $C^*$-algebras also holds \cite{ror}*{page 1}. To conclude his argument, R\o rdam uses an idea of Cuntz, which consists of considering another graph alteration, called the \emph{Cuntz splice}. For a finite regular graph $E$, let $E^0=\{v_1,\dots,v_n\}$ be an enumeration of its vertices. The Cuntz splice of $E$ at $v_n$ is the graph $E^-$
whose adjacency matrix is 
\begin{equation}
A_{E^-} \, = \, \begin{bmatrix}
  &   &  &  & 0 & 0 \\
  & A_E  &  &  & \vdots & \vdots \\
  &   &  &  & 0 & 0 \\
  &   &  &  & 1 & 0 \\
 0 & \cdots  & 0 & 1 & 1 & 1 \\
 0 & \cdots  & 0 & 0 & 1 & 1 \\
\end{bmatrix} 
\qquad
E^{-} \ = \ \xy
(0,0)*[F**:<8pt>]++{\xy
    (0,10)*{};
    (0,-10)*{};
    (3,5)*{E};
    (20,0)*{}="vv";
    (10,7)*{}="vv1";
    (10,0)*{}="vv2";
    (10,-7)*{}="vv3";
    {\ar@{.} "vv1" ;"vv" };
    {\ar@{.} "vv2" ;"vv" };
    {\ar@{.} "vv3" ;"vv" };
\endxy };
(10,0)*{\bullet}="v";
(13,0)*{\scriptstyle v_n};
(20,0)*{\bullet}="v1";
(30,0)*{\bullet}="v2";
{\ar@/^/     "v" ;"v1" };
{\ar@/^/     "v1";"v"  };
{\ar@/^/     "v1";"v2" };
{\ar@/^/     "v2";"v1" };
{\ar@(ul,ur) "v1";"v1" };
{\ar@(ur,dr) "v2";"v2" };
\endxy
\end{equation}

The following proposition summarizes the key properties of the Cuntz splice.
\begin{thm}[\cite{ror}*{page 54}, \cite{flow}*{Proposition 2.12}]\label{prop:e-}
Let $E$ be a finite regular graph. Then 
\begin{equation}
\BF(E)\cong \BF(E^-),\,\text{ and }\, \det(I-A_{E^-})=-\det(I-A_E).
\end{equation}
Moreover $E$ is SPI if and only if $E^-$ is. 
\end{thm}

\begin{thm}[R\o rdam \cite{ror}*{Lemma 6.4 and Theorem 7.2}]\label{thm:splice}
Let $E$ be a finite SPI graph. Then there is a $*$-isomorphism $C^*(E)\cong C^*(E^-)$. 
\end{thm}
\begin{proof}[Structure of the proof of Theorem \ref{thm:splice}] By an argument of Cuntz, explained in \cite{ror}*{Theorem 7.2}, it suffices to establish the theorem for the case $E=\cR_2$; this is done in \cite{ror}*{Lemma 6.4}.
\end{proof}

\begin{proof}[Proof of Corollary \ref{coro:KPR} in the finite case] Let $E$ and $F$ be finite SPI graphs and $\xi:\BF(E)\iso \BF(F)$ an isomorphism such that $\xi([1_E])=[1_F]$. Then by the Smith normal form, we have $|\det(I-A_E)|=|\det(I-A_F)|$. Hence we must either have $\det(I-A_F)=\det(I-A_E)$ or $\det(I-A_F)=-\det(I-A_E)=\det(I-A_{E^-})$. By the $C^*$-analogue of Theorem \ref{thm:flow}, we have $C^*(F)\cong C^*(E)$ in the first case and  $C^*(F)\cong C^*(E^-)$ in the second. Now apply Theorem \ref{thm:splice}.
\end{proof}

\begin{quest}\label{quest:splice}
Let $E$ be a finite SPI graph. Is $L(E)\cong L(E^-)$?
\end{quest}

Note that if the answer to Question \ref{quest:splice} is positive, then the argument above would apply verbatim to show that the answer to Question \ref{quest:flow} is positive too. 

Let $n\ge 2$; write $L_{n^-}=L(\cR_n^-)$. The following is a particular case of Question \ref{quest:splice}.

\begin{quest}\label{quest:l2}
Are $L_2$ and $L_{2}^-$ isomorphic?
\end{quest}

As explained above, Cuntz showed that a positive answer to the analogous question of whether $\cO_2\cong\cO_{2}^-$ (later confirmed by R\o rdam) would imply $\cO(E)\cong \cO(E^-)$ for all finite SPI graphs $E$. A partial analogue of Cuntz' result in the purely algebraic context was obtained in \cite{flow}*{Theorem 2.13}; it says that if the answer to Question \eqref{quest:l2} is positive and the isomorphism satisfies certain properties, then the answer to \eqref{quest:splice} is positive also. 
Remark however that since $\BF(\cR_{2}^-)=\BF(\cR_2)=0$, if the answer to \ref{quest:l2} is negative, then the answer to Question \ref{quest:flow} is also negative.

Next we list a couple of results that impose restrictions on extra properties that an isomorphism $L_2\iso L_{2}^-$ or $L(E)\iso L(E^-)$, if it exists, may have.

If $R$ is a commutative ring with unit and $E$ is a finite graph, then the \emph{diagonal} of $L_R(E)$ is the subalgebra
\[
D_R(E)=\mspan_R\{\alpha\alpha^*\,\colon\, \alpha\in \Path(E)\}.
\]
A \emph{unital $\ast$-subring} of $\C$ is a subring $R\subset \C$ containing $\Z$ and closed under complex conjugation.
\begin{thm}[Johansen and S\o rensen \cite{johsor}*{Theorem 3.6}]\label{thm:sgn=diag} 
Let $E$ and $F$ be finite essential SPI graphs and let $R\subset \C$ be a unital $\ast$-subring. Let $\phi:L_R(E)\to L_R(F)$ be a $\ast$-homomorphism
such that $\phi(D_R(E))\subset D_R(F)$. Then $\sg(\det(I-A_E))=\sg(\det(I-A_F))$.
\end{thm}

In view of Proposition \ref{prop:e-} and Theorem \ref{thm:sgn=diag}, if a $\ast$-isomorphism $L_R(E)\to L_R(E^-)$ exists, it cannot preserve the diagonal subalgebra.
Let $R$ be a unital $*$-subring of $\C$. We say that $R$ is \emph{kind} if for $n\ge 2$ and $x\in R^n$,
\[
x_1=\sum_{j=1}^n|x_j|^2\Rightarrow x_2=\dots=x_n=0.
\]
\begin{ex}\label{ex:kindstein}
A first basic example of a kind ring is $\Z$. Other examples include $\Z[X]$ for any set $X\subset \R$ algebraically independent over $\Q$, as well as $\Z[\sqrt{n}]$ and $\Z[i\sqrt{n}]$ for $n\in \Z_{\ge 2}$. A kind ring cannot contain $1/n$ for any $n\ge 2$; in particular, it cannot be a field. For proofs of these assertions and permanence properties of the class of kind rings see 
\cite{steinkind}*{Proposition 3}.
\end{ex}

Let $A$ be a ring with involution $\ast$. An element $a\in A$ is \emph{self-adjoint} if $a^\ast=a$. A \emph{projection} in $A$ is a self-adjoint idempotent. 
\begin{ex}\label{ex:proje}
Let $E$ be a graph and $\alpha\in \Path(E)$. Let $R$ be a ring with involution. Then $\alpha\alpha^*\in L_R(E)$ is a projection. In particular $D_R(E)$ is $R$-linearly generated by projections. 
\end{ex}
\begin{prop}[Carlsen \cite{tokel2}*{Proposition 4}]\label{prop:tokel2}
Let $R$ be commutative unital ring with involution, $E$ a graph and $p\in L_R(E)$ a projection. If $R$ is kind, then $p\in D_R(E)$.
\end{prop}
\begin{coro}\label{coro:tokel2}
Let $R$ be a kind ring, $E$ and $F$ graphs, and $\phi:L_R(E)\to L_R(F)$ a homomorphism of $\ast$-algebras. Then $\phi(D_R(E))\subset D_R(F)$.
\end{coro}
\begin{proof}
By Example \ref{ex:proje}, $D_R(E)$ is $R$-linearly generated by projections. Because $\phi$ is a $\ast$-homomorphism, it maps projections to projections. Because it is a homomorphism of $R$-algebras, it is $R$-linear. Hence $\phi(D_R(E))\subset D_R(F)$, by Proposition \ref{prop:tokel2}. 
\end{proof}

\begin{rem}\label{rem:kind}
Proposition \ref{prop:tokel2} was first proved for finite graphs in \cite{johsor}*{Theorem 5.6} and for $\ast$-subrings $R\subset\C$ such that for
$n\ge 2$ and $x\in R^n$,
\[
1=\sum_{j=1}^n|x_j|^2\Rightarrow \exists !\, 1\le j\le n \text{ such that } x_j\ne 0. 
\]
It was shown in \cite{tokel2}*{Section 3} that any kind ring satisfies the above property. In fact the converse is also true, as was proved later by Steinberg in \cite{steinkind}*{Theorem 2}. The latter theorem further shows that kind rings are precisely the unital $\ast$-subrings of $\C$ such that for any Hausdorff ample groupoid $\cG$, every projection in the groupoid algebra $\cA_R(\cG)$ belongs to the diagonal subalgebra. Since Leavitt path algebras are particular cases of such groupoid algebras \cite{gpdgen}*{Proposition 4.3}, Proposition \ref{prop:tokel2} follows from \cite{steinkind}*{Theorem 2}.
\end{rem}

Next we turn our attention to graded homomorphisms. It was noted in \cite{apgrad}*{Example 4.2} that there is no $\Z$-graded isomorphism $L_2\iso L_{2}^-$. One may also consider, for $m\ge 2$, the associated $\Z/m\Z$-grading 
\[
L(E)_{\ol{r}}=\bigoplus_{q\in \Z}L(E)_{mq+r},\,\, (r\in \{0,\dots,m-1\}). 
\]
\begin{thm}[Arnone and Corti\~nas \cite{ac1}*{Theorem 1.1}]\label{thm:noexistis}
Let $R$ be a field or a principal ideal domain, $n\in \N_{\ge 2}$ and $m\in \{0\}\cup \N_{\ge 2}$. Then there is no unital $\Z/m\Z$-graded ring homomorphism $L_n(R)\to L_{n^-}(R)$ nor in the opposite direction. 
\end{thm}

The proof of Theorem \ref{thm:noexistis} uses graded $K$-theory; it will be discussed in Section \ref{subsec:gradK}. 

\subsection{Bivariant algebraic~\topdf{$K$}{K}-theory}\label{subsec:kk}
Let $\K$ be a field and let $\phi,\psi:A\to B$ be $\K$-algebra homomorphisms. An \emph{elementary homotopy} $H:\phi\to \psi$ is an algebra homomorphism $H:A\to B[t]$ such that $H(a)(0)=\phi(a)$ and $H(a)(1)=\psi(a)$, for all $a\in A$. We say that $\phi$ and $\psi$ are \emph{homotopic} and write $\phi\sim\psi$ if there are a finite sequence of algebra homomorphisms $\phi=\phi_0,\dots,\phi_n=\psi$ and elementary homotopies $H_i:\phi_i\to \phi_{i+1}$, where  $0\le i\le n-1$. 
Let $\aha$ be the category of associative algebras over $\K$ and $\K$-linear ring homomorphisms, $\cC$ a category, and 
\begin{equation}\label{map:funh}
\cH:\aha\to\cC    
\end{equation}
a functor. Call $\cH$ \emph{homotopy invariant} if it sends homotopic maps to the same map. Let $\iota_A:A\to \M_\infty(A)$, $a\mapsto \epsilon_{1,1}a$; $\cH$ is called \emph{matricially stable} if $\cH(\iota_A)$ is an isomorphism for all $A\in\aha$. If $\cH$ is matricially stable,  then $\cH$ sends $A\to \M_n(A)$, $a\mapsto \epsilon_{1,1}a$, to an isomorphism, for $1\le n<\infty$ \cite{friendly}*{Section 2.2}. An \emph{extension} of algebras is a sequence of algebra homomorphisms
\[
0\to A\to B\to C\to 0
\]
which is exact as a sequence of vector spaces. The functor \eqref{map:funh} is \emph{excisive} if $\cC$ is triangulated and $\cH$ maps extensions to distinguished triangles. There exists a triangulated category $kk$, whose objects are those of $\aha$, and a functor $j:\aha\to kk$ defined as the identity on objects. The functor $j$ is homotopy invariant, matricially stable and excisive, and universal initial with respect to the latter properties. We write
\[
 kk_n(A,B)=\hom_{kk}(j(A),j(B)[n]),\,\, kk(A,B)=kk_0(A,B).
\]
Furthermore setting the first variable equal to the ground ring, recovers homotopy algebraic $K$-theory 
\begin{equation}\label{eq:kh=kk}
kk_n(\K,B)=KH_n(B).
\end{equation}
Homotopy algebraic $K$-theory $KH$ was defined by Weibel in \cite{kh}; see \cite{friendly}*{Section 5} for an elementary introduction. It follows from \eqref{eq:kh=kk} that 
$KH$ satisfies excision, is homotopy invariant and matricially stable on all $\K$-algebras. Ordinary (i.e. Quillen's) algebraic $K$-theory $K$ has none of those properties, at least not for all algebras and all algebra extensions; so $K$ and $KH$ are different in general. 
There is a natural comparison map 
\[
c_j:K_j(B)\to KH_j(B) \,\, (j\in\Z).
\]
The ring $B$ is called \emph{$K_n$-regular} if the natural map $K_n(B)\to K_n(B[t_1,\dots,t_m])$ is an isomorphism for all $m$, and \emph{$K$-regular} if it is $K_n$-regular for all $n$. If $B$ is $K_n$-regular then $c_j$ is an isomorphism for all $j\le n$ \cite{kh}*{Proposition 1.5}. Thus if $R$ is $K$-regular then $c_j$ is an isomorphism for all $j$.

\begin{ex}\label{ex:khlpa}
If $R$ is regular supercoherent (e.g. if it is regular noetherian) and $E$ is a graph, then $L_R(E)$ is $K$-regular (i.e. $K_n$-regular for all $n\in\Z$) and thus $c_j:K_j(L_R(E))\to KH_j(L_R(E))$ is an isomorphism for all $j\in\Z$ (see \cite{classinvo}*{Lemma 4.3}).
\end{ex}

\subsection{Leavitt path algebras in \topdf{$kk$}{kk}}\label{subsec:leavkk}

\begin{thm}[Corti\~nas and Montero \cite{cm1}*{Theorem 5.4 and Example 5.5}]\label{thm:fundtri} Let $E$ be a countable graph with $|E^0|<\infty$. Then there is a distinguished triangle in $kk$
\[
j(L(E))[1]\to j(\K)^{\reg(E)}\overset{I-A_E^t}{\lra} j(\K)^{E^0}\lra j(L(E)).
\]
\end{thm}

\begin{proof}[Sketch of the proof] The ideal $\cK(E)\triqui C(E)$ of \eqref{sescohn} is the direct sum over $v\in \reg(E)$ of the ideal $\cK(E)_v$ generated by the element $q_v$ of \eqref{eq:qv}. The elements $\alpha q_v\beta^*$ where $\alpha$ and $\beta$ are paths with $r(\alpha)=r(\beta)=v$ form a basis of $\cK(E)_v$, and the assignment $\alpha q_v\beta^*\mapsto \epsilon_{\alpha,\beta}$ induces an isomorphism from $\cK(E)_v$ to the algebra of finitely supported matrices indexed by the set of all paths in $E$ ending in $v$. Since $E$ is countable, so is the latter set. Hence by matricial stability, the algebra homomorphism 
\[
\K \reg(E)=\bigoplus_{v\in \reg(E)}\K\to \cK(E), \, v\mapsto q_v
\]
is a isomorphism in $kk$. A more laborious argument \cite{cm1}*{Theorem 4.2} shows that the inclusion $\K E^0$ into the Cohn algebra $C(E)$ is a $kk$-isomorphism as well. Thus up to isomorphism, in $kk$ the sequence \eqref{sescohn} gives rise to a triangle which, omitting the functor $j$, looks as follows
\begin{equation}\label{eq:triang}
L(E)[+1]\to \K \reg(E)\overset{f}{\lra}\K E^0\to L(E).    
\end{equation}

Here $f$ is induced by the inclusion $l$ of \eqref{sescohn}. Using \eqref{eq:kh=kk} and Example \ref{ex:khlpa}, we have
\begin{equation*}
kk(\K \reg(E),\K E^0)=kk(\K,\K)^{E^0\times \reg(E)}=KH_0(\K)^{E^0\times \reg(E)}
=K_0(\K)^{E^0\times \reg(E)}=\Z^{E^0\times \reg(E)}.
\end{equation*}
It is not hard to get from this and from \eqref{eq:qv} that $f=I-A_E^t$.
The theorem now follows using that $j$ is excisive. 
\end{proof}
Let $E$ be as in Theorem \ref{thm:fundtri}. In the next corollary and elsewhere we write
\[
\BF^\vee(E)=\coker(I^t-A_E).
\]

\begin{coro}\label{coro:fundtri} For any algebra $R$ there is an exact sequence
\[
0\to \BF^\vee(E)\otimes KH_1(R)\to kk(L(E),R)\to \hom(\BF(E),KH_0(R))\to 0.
\]   
\end{coro}
\begin{proof}
Apply $kk(-,R)$ to the triangle \eqref{eq:triang} and use \eqref{eq:kh=kk} to get an exact sequence
\begin{equation}\label{seq:uct1}
KH_1(R)^{E^0}\overset{I^t-A_E}{\lra} KH_1(R)^{\reg(E)}\lra kk(L(E),R)\to 
KH_0(R)^{E^0}\overset{I^t-A_E}{\lra} KH_0(R)^{\reg(E)}   
\end{equation}
The corollary follows upon interpreting the first map as the result of tensoring $I^t-A_E:\Z^{E^0}\to \Z^{\reg(E)}$ with $K_1(R)$ and the last map as resulting from applying $\hom_{\Z}(-,KH_0(R))$ to $I-A_E^t$.   
\end{proof}
\begin{rem}\label{rem:k-1=bfv}
The following computation follows from Theorem \ref{thm:fundtri} (see 
\cite{cm1}*{Section 6})
\begin{equation}\label{eq:k-1=bfv}
kk_{-1}(L(E),\K)=\BF^\vee(E).    
\end{equation}
It was shown in \cite{cm1}*{Lemma 7.21} that under the identifications \eqref{eq:k-1=bfv} and \eqref{eq:kh=kk}, the map $\BF^\vee(E)\otimes KH_1(R)\to kk(L(E),R)$ is given by composition in $kk$. 
\end{rem}

\begin{rem}\label{rem:uct2} 
By interpreting the first map in \eqref{seq:uct1} as the result of applying $\hom(-,KH_1(R))$
to $I-A_E^t:\Z^{\reg(E)}\to \Z^{E^0}$ and proceeding as in \cite{cm1}*{Theorem 7.12 and Corollary 7.20}, one obtains an exact sequence
\begin{equation*}
0\to \Ext^1_\Z(\BF(E),KH_1(R))\to kk(L(E),R)\to \\
\hom(\BF(E),KH_0(R))\oplus\hom(\ker(I-A_E^t),KH_1(R))\to 0    
\end{equation*}
Remark the similarities between the exact sequence above and that of the UCT for $KK$ of $C^*$-algebras \eqref{seq:uct}. Next we specialize to the case when $R=L(F)$ for some finite graph $F$; by Example \ref{ex:khlpa}, we may replace $KH$ by $K$, and use \eqref{eq:k01le}.
If $\K=\C$, then $\BF(F)\otimes \K^*$ is injective, so we obtain 
\begin{multline}\label{seq:kklelf}
0\to \Ext_\Z^1(\BF(E), \ker(I-A_F^t))\to kk(L(E),L(F))\to\\
\hom(\BF(E),\BF(F))\oplus \hom(\ker(I-A_E^t),\ker(I-A_F^t)\oplus \BF(F)\otimes\C^*)\to 0  
\end{multline}
Comparing the above sequence with \eqref{seq:kkcecf} we see that \eqref{seq:kklelf}  has an extra term $\hom(\ker(I-A_E^t), \BF(F)\otimes\C^*)$, which vanishes if $\BF(F)$ is finite. 
\end{rem}
\begin{rem}\label{rem:ojouct} 
It is not hard to show, using  Corollary \ref{coro:fundtri}, that if $E$ and $F$ are vertex-finite graphs and $\xi\in kk(L(E),L(F))$ is such that $K_i(\xi)$ is an isomorphism for $i=0,1$ then $\xi$ is an isomorphism in $kk$ \cite{cm1}*{Proposition 5.10}. However the existence of group isomophisms $\xi_i:K_i(L(E))\iso K_i(L(F))$ for $i=0,1$ does not imply that $j(L(E))\cong j(L(F))$, see \cite{cm1}*{Remark 5.11}. Thus the $kk$-analogue of Remark \ref{rem:uctis} does not hold. 
\end{rem}

In \cite{cm1}*{Theorem 6.10}, Theorem \ref{thm:fundtri} was used to establish a structure theorem for the $kk$-isomorphism class of Leavitt path algebras of vertex-finite graphs. Then the structure theorem was used to deduce the following. 

\begin{thm}[\cite{cm1}*{Corollary 6.11}]\label{thm:kksing} Let $E$ and $F$ be graphs with finitely many vertices. Then the following are equivalent. 
\item[i)] $j(L(E))\cong j(L(F))$.
\item[ii)] $K_0(L(E))\cong K_0(L(F))$ and  $kk_{-1}(L(E),\K)\cong kk_{-1}(L(F),\K)$.
\item[iii)] $\BF(E)\cong\BF(F)$ and $|\sing(E)|=|\sing(F)|$.
\end{thm}

\begin{rem} In \cite{ruto}*{Theorem 7.4}, Ruiz and Tomforde showed that if $E$ and $F$ are simple and both have infinitely many edges, then condition (iii) of Theorem \ref{thm:kksing} holds if and only if $L(E)$ and $L(F)$ are Morita equivalent. Thus for the Leavitt path algebras of such graphs, Morita equivalence is the same as $kk$-isomorphism. In particular for $E$ and $F$ as in \cite{ruto}*{Example 11.2}, $L_\Q(E)$ is not Morita equivalent to $L_\Q(F)$. This is proved in loc. cit. by showing that $K_2(L_\Q(E))\not\cong K_2(L_\Q(E))$. For an alternative proof, use  Theorem \ref{thm:fundtri} and the matrix calculations of \cite{ruto} to show that, over any field $\K$, $j(L(E))\cong j(\K)^2$ and $j(L(F))=j(\K)^2\oplus j(\K)[-1]$, so that $kk_{-1}(L(E),\K)=0$ and $kk_{-1}(L(F),\K)=\Z$. 
\end{rem}
\subsection{Homotopy Classification}\label{subsec:homotopy1}
 Let $B$ be a $\K$-algebra, where $\K$ is a field and  $\iota_2:B\to \M_2(B)$, $\iota_2(b)=\epsilon_{1,1}b$,  the upper left corner inclusion into the matrix algebra. We say that two algebra homomorphisms $\phi,\psi:A\to B$ are \emph{$\M_2$-homotopic}, and write $\phi\sim_{\M_2}\psi$, if $\iota_2\circ\phi\sim\iota_2\circ\psi$.

\begin{thm}[Corti\~nas and Montero \cite{cm2}*{Theorem 1.1}]\label{thm:cm2}
 Let $E$ and $F$ be finite SPI graphs and let $\K$ be a field.  Let 
 $\xi_0:K_0(L(E))\iso K_0(L(F))$ be an isomorphism of groups. Then there are algebra homomorphisms $\phi:L(E)\leftrightarrows L(F):\psi$ such that $\psi\circ\phi\sim_{\M_2}\id_{L(E)}$ and $\phi\circ\psi\sim_{\M_2}\id_{L(F)}$. If furthermore $\xi_0([L(E)])=[L(F)]$, then $\phi$ and $\psi$ can be chosen to be unital homomorphisms so that $\psi\circ\phi\sim\id_{L(E)}$ and $\phi\circ\psi\sim\id_{L(F)}$
\end{thm}
\begin{proof}[Sketch of the proof] 
If $R$ is $K$-regular (e.g. if it is a Leavitt path algebra) then the exact sequence of Corollary \ref{coro:fundtri} becomes
\begin{equation}\label{seq:uctle}
0\to \BF(E)^\vee\otimes K_1(R)\to kk(L(E),R)\to \hom(\BF(E),K_0(R))\to 0.    
\end{equation}
In particular, setting $R=L(F)$, we see that the element $\xi_0$ of the statement of the theorem lifts to $\xi\in kk(L(E),L(F))$. One checks that the first onto map above is induced by the functor $kk(\K,-)=K_0(-)$. In particular, we have a ring extension
\[
0\to \BF(E)^\vee\otimes K_1(L(E))\to kk(L(E),L(E))\to \hom(\BF(E),\BF(E))\to 0.
\]
It turns out that the kernel of the extension above is a square-zero ideal (this follows from \cite{cm1}*{Lemma 7.21} by the argument of the proof of \cite{bullift}*{Theorem 3.2}). In particular, any element of $ kk(L(E),L(E))$ mapping to the identity
of $\BF(E)$ is invertible. It follows from this that the element $\xi\in kk(L(E),L(F))$ above is a $kk$-isomorphism. Next observe that since $j$ is homotopy invariant and matricially stable, for any $A,B\in\aha$, the map $j:\hom_{\aha}(A,B)\to kk(A,B)$ factors through
\[
[A,B]_{\M_2}=\hom_{\aha}(A,B)/\sim_{\M_2}.
\]
The key part of the proof is to show that for $E$ and $F$ as in the theorem, the induced map is a bijection
\begin{equation}\label{map:homotokk}
[L(E),L(F)]_{\M_2}\setminus\{0\}\iso kk(L(E),L(F)).
\end{equation}
This holds more generally with $L(F)$ replaced by any SPI $K_1$-regular algebra $R$ \cite{cm2}*{Theorem 5.8}. As a first step for doing this, one shows that for any element $\eta_0\in\hom(\BF(E),K_0(R))$ there exists a nonzero homomorphism of algebras $\phi:L(E)\to R$ such that $K_0(\phi)=\eta_0$, which can be chosen to be unital if $\eta_0([L(E)])=[R]$. This is rather elementary using the relations defining $L(E)$ and the fact that any SPI algebra is properly infinite \cite{cm2}*{Theorem 3.1}, 
\cite{classinvo}*{Theorem 9.3}. If now $\eta\in kk(L(E),R)$ maps to $\eta_0$, then $\eta-j(\phi)$ lies in the image of $\BF(E)^\vee\otimes K_1(R)$. By definition, $K_1$ is the abelianization of the full general linear group $\Gl=\bigcup_{n\ge 1}\Gl_n$. The fact that $R$ is SPI implies that it coincides with the abelianization of $R^*=\Gl_1(R)$; $K_1(R)=R^*_{\ab}$ \cite{agop}*{Theorem 2.3}. From here one deduces  that $\eta=j(\psi)$ for some (unital) algebra homomorphism $\psi$ which results by twisting $\phi$ by an element $\zeta$ coming from $\BF^{\vee}(F)\otimes R^*_{\ab}$. This, and the injectivity of \eqref{map:homotokk} are the more delicate parts of the proof; both require certain understanding of the injection
in the exact sequence \eqref{seq:uctle}. In \cite{cm2}, this is done by adapting R\o rdam's argument in \cite{ror}; a different approach using Poincar\'e duality is applied in \cite{classinvo}. The latter establishes \cite{classinvo}*{Theorem 11.2} that if $E$ is a finite essential graph and $E_t$ is the transpose graph (so that $A_{E_t}=(A_E)^t$), then for every pair of algebras $(S,R)$, there is a canonical isomorphism
\[
kk(L(E)\otimes S,R)\cong kk_1(S,R\otimes L(E_t)).
\]
In particular, setting $S=\K$ and using \eqref{eq:kh=kk}, we obtain
\begin{equation}\label{eq:kk=kh1}
kk(L(E),R)\cong KH_1(R\otimes L(E_t)).    
\end{equation}

Moreover one shows \cite{classinvo}*{Lemma 12.3} that under the isomorphism \eqref{eq:kk=kh1}, the injection in \eqref{seq:uctle} corresponds to that obtained by applying $KH_1(-\otimes R)$ to the projection $C(E_t)\to L(E_t)$. This map is well-understood and allows one to understand the injection of \eqref{seq:uctle}, at least when $E$ is essential. By subsequently removing sources that are also regular vertices, one can get replace any finite graph by another without regular sources, and the resulting Leavitt path algebra is Morita equivalent to the original one (this is stated in \cite{flow}*{Proposition 1.4} for the case when $E$ is simple, but it holds in general, see  \cite{B13}*{Proposici\'on 6.1.1}). If the starting graph is SPI, then it has no sinks, so the process just described produces an essential graph. Since being SPI is preserved by Morita equivalence \cite{agop}*{Corollary 1.7}, one can reduce to the case when $E$ is an essential SPI graph and proceed as indicated above. 
\end{proof}

\begin{rem}\label{rem:homotoKK}
Let $\fA,\fB$ be Kirchberg $C^*$-algebras and let $[[\fA,\fB]]$ and $[[\fA,\fB]]_{M_2}$ be the sets of (continuous) homotopy equivalence classes and $\M_2$-homomotopy equivalence classes of $\ast$-homomorphisms $\fA\to\fB$. By \cite{bullift}*{Lemma A.5}, $[[\fA,\fB\sotimes\cK]]=[[\fA,\fB]]_{\M_2}$. The analogue of \eqref{map:homotokk}, stating that 
\begin{equation}\label{map:homotoKK}
[[\fA,\fB\sotimes\cK]]\setminus\{0\}\to KK(\fA,\fB)
\end{equation}
is bijective is part of \cite{P}*{Theorem 4.1.1}. The proof of Theorem \ref{thm:KP} uses the UCT \eqref{seq:uct} to see that $K_*(\fA)\cong K_*(\fB)$ implies that $\fA$ and $\fB$ are isomorphic in $KK$ and then \eqref{map:homotoKK} to derive the analogue of Theorem \ref{thm:cm2} for Kirchberg algebras.  However the argument to pass from a homotopy equivalence to an isomorphism in the $C^*$-setting is quite analytic and does not have an obvious algebraic counterpart that one can apply to derive a positive answer to Question \ref{quest:flow} from Theorem \ref{thm:cm2}. 
 \end{rem}

The following is a strong version of Question \ref{quest:flow}.
\begin{quest}\label{quest:starflow}
Let $E$ and $F$ be finite SPI graphs and let $\xi:\BF(E)\iso \BF(F)$ be an isomorphism such that $\xi[1_E]=[1_F]$. Is there a $\ast$-isomorphism $\phi:L_\C(E)\to L_\C(F)$ such that $K_0(\phi)=\xi$?
\end{quest}

 Remark that a $\ast$-homomorphism $\phi:L_\C(E)\to L_\C(F)$ passes to the completion, giving a $\ast$-homomorphism $\hat{\phi}:C^*(E)\to C^*(F)$. This construction preserves composition and identity maps, and thus $\hat{\phi}$ is an isomorphism whenever $\phi$ is. Thus a positive answer to Question \ref{quest:starflow} would directly imply R\o rdam's theorem that Corollary \ref{coro:KPR} holds for finite SPI graphs.

Theorem \ref{thm:bullift} addresses Question \ref{quest:starflow} up to homotopy. Its proof involves mapping bivariant algebraic $K$-theory of Leavitt path algebras to Kasparov's bivariant $K$-theory of their $C^*$-completions.

\begin{thm}[Corti\~nas \cite{bullift}*{Theorem 4.10}]\label{thm:bullift}
Let $E$ and $F$ be finite SPI graphs and let $\xi:C^*(E)\to C^*(F)$ be a unital $\ast$-homomorphism. Then there exists a unital 
$\ast$-homomorphism  $\phi:L_\C(E)\to L_\C(F)$ whose completion $\hat{\phi}:C^*(E)\to C^*(F)$ is $C^*$-homotopic to $\xi$. Moreover $\xi$ is a homotopy equivalence 
in the $C^*$-sense if and only if $\phi$ is a homotopy equivalence in the algebraic, polynomial sense. 
\end{thm}
\begin{proof}[Idea of the proof] One has to map $[L(E),L(F)]_{\M_2}$ to
$[[C^*(E),C^*(F)]]$ and show that it is onto and reflects homotopy equivalences. By \eqref{map:homotokk} and Remark \ref{rem:homotoKK} one may replace homotopy classes by $kk(L(E),L(F))$ and $KK(C^*(E),C^*(F))$. To do this, one uses Poincar\'e duality, both in $kk$ \cite{classinvo}*{Theorem 11.2} and in $KK$ \cite{kamiput}*{Section 4}. As explained in the sketch of the proof of Theorem \ref{thm:cm2}, the $kk$-version of Poincar\'e duality implies that for finite regular graphs $E$ and $F$ with $E$ essential with transpose graph $E_t$, there is a canonical isomorphism $KH_1(L(E_t)\otimes L(F))\cong kk(L(E),L(F))$; similarly $K_1^{\top}(C^*(E_t)\sotimes C^*(F))\cong KK(C^*(E),C^*(F))$. Using this and the canonical comparison map from algebraic homotopy to topological $K$-theory of a $C^*$-algebra, one obtains a map $\comp:kk(L(E),L(F))\to KK(C^*(E),C^*(F))$. It is shown in \cite{bullift}*{Theorem 3.2} that the latter map defines a full functor $\comp$ from the full subcategory of $kk$ whose objects are the Leavitt path algebras of finite regular graphs to $KK$, mapping $L(E)\mapsto C^*(E)$. It is further shown that $\comp$ reflects isomorphisms and $\comp(j(\phi))=k(\hat{\phi})$ whenever $\phi$ is a $\ast$-homomorphism. All this is used in \cite{bullift}*{Theorem 4.10} to prove the theorem.
\end{proof}

\section{Symbolic dynamics: conjugacy and (strong) shift equivalence}\label{sec:symb}
\numberwithin{equation}{section}
The central objects in the theory of symbolic dynamics are the shifts of finite type, also called topological Markov chains. A general reference for this section is \cite{lind}. In this article, by a \emph{dynamical system} we understand a pair $(X,\sigma)$ consisting of a Hausdorff topological space $X$ and a homeomorphism $\sigma$. Two dynamical systems $(X,\sigma)$ and $(Y,\tau)$ are \emph{conjugate} if there is an homeomorphism $f:X\to Y$ such that $\tau\circ f=f\circ\sigma$.

Let $E$ be a finite essential graph and give $E^1$ the discrete topology. Then  $(E^1)^\Z$ with the product topology is Hausdorff and compact, by Tychonov's theorem. The subset $X_E\subset (E^1)^\Z$ of all bi-infinite paths in $E$ is closed, and therefore a compact Hausdorff space. The map $\sigma:X_E\to X_E$, $\sigma(\alpha)_n=\alpha_{n+1}$ that shifts a path to the left is continuous, and so the pair $(X_E,\sigma)$ is a dynamical system; it is called \emph{an edge shift}. Although we will not define shifts of finite type, let us mention that any shift of finite type is conjugate to an edge shift. 

As a consequence of the decomposition theorem for conjugacies \cite{lind}*{Corollary 7.1.5} two shifts of finite type $X_E$ and $X_F$ are conjugate if and only if $F$ can be obtained from $E$ by a series of graphs moves, called \emph{in-splittings} and \emph{out-splittings}, and their inverses, called in and out \emph{amalgamations}. A theorem of Williams further characterizes the existence of a conjugacy between $X_E$ and $X_F$ in terms of the adjacency matrices of $E$ and $F$. Two square nonnegative integer matrices $A$ and $B$ are called {\it elementary shift equivalent}, and denoted by $A\sim_{ES} B$, if there are nonnegative matrices $R$ and $S$ such that $A=RS$ and $B=SR$. 
The equivalence relation $\sim_S$  on square nonnegative integer matrices generated by elementary shift equivalence is called {\it strong shift equivalence}.

\begin{thm}[Williams~\cite{williams}]\label{willmnhfhf}
Let $A$ and $B$ be two square nonnegative integer matrices and let $E$ and $F$ be their associated graphs. Assume that $E$ and $F$ are essential. Then $A$ is strongly shift equivalent to $B$ if and only if $F$ can be obtained from $E$ by a sequence of in-splittings, out-splittings, and their inverses.
\end{thm}

In the same article \cite{williams} Williams also introduced the weaker notion of shift equivalence, defined as follows. The square nonnegative integer matrices $A$ and $B$ are called {\it shift equivalent} if there exist $\ell\ge 1$ and nonnegative matrices $R$ and $S$ such that 
\begin{gather}
AR=RB,\, SA=BS,\label{eq:equifuera}\\
RS=A^\ell,\, SR=B^\ell.\label{eq:equidentro}
\end{gather}

Observe that for $\ell=1$, conditions \eqref{eq:equifuera} and \eqref{eq:equidentro} hold if and only if $A\sim_{ES}B$. It follows that strong shift equivalence implies shift equivalence. 

The problem of deciding whether the converse is true remained open for a long time and was known as Williams' conjecture. It was disproved by the counterexamples of Kim and Roush \cites{kr1,kr2}. 

The following question remains open at the time of writing. 

\begin{quest}\label{quest:semuove}
 Is there a list of graph moves so that two graphs have shift equivalent adjacency matrices if and only if one can transform one into the other by a finite sequence of moves in the list?   
\end{quest} 

An advantage of shift equivalence over strong shift equivalence is that it can be detected by a simple invariant, Krieger's dimension group, which we will introduce next.

A nonnegative integral $n\times n$ matrix $A$ gives rise to a direct system of free abelian groups with $A$ acting as an order-preserving group homomorphism
\[\mathbb Z^n \stackrel{A}{\longrightarrow} \mathbb Z^n \stackrel{A}{\longrightarrow}  \mathbb Z^n \stackrel{A}{\longrightarrow} \cdots
\]
We regard $\mathbb Z^n$ as a partially ordered group with positive cone $\mathbb N^n$. The direct limit of this system, $\Delta_A:= \varinjlim_{A} \mathbb Z^n$  is a partially ordered group whose positive cone $\Delta^+$ is the direct limit of the associated direct system of positive cones. Multiplication by $A$ induces an automorphism $\delta_A$ of partially ordered groups. The triple  $(\Delta_A, \Delta_A^+, \delta_A)$ is called \emph{Krieger's dimension group}. \index{Krieger's dimension group}

\begin{rem}\label{rem:DeltaA}
Krieger's dimension group $\Delta_A$ is in fact a module over the Laurent polynomial ring $\Z[x,x^{-1}]$ via $x\cdot v=\delta_A^{-1}(v)$. Moreover the submonoid $\Delta_A^+$ is closed under multiplication by $\N_0[x,x^{-1}]$. Hence we may regard the Krieger triple as a partially ordered $\Z[x,x^{-1}]$-module, and there is a module isomorphism
\begin{equation}\label{eq:DeltA}
    \Delta_A\cong\coker(I-xA:\Z[x,x^{-1}]^n\to \Z[x,x^{-1}]^n)
\end{equation}
mapping $\Delta_A^+$ isomorphically onto the image of $\N_0[x,x^{-1}]^n$. 
\end{rem}

\begin{ex}\label{ex:DeltaA}
Let $n\in\N$, $\cR_n$ the $n$-petal rose, $\cR_{n^-}$ its Cuntz' splice, and $A$ and $A^-$  their adjacency matrices. Then it is immediate from  \eqref{eq:DeltA} that 
$$
\Delta_A=\Z[x,x^{-1}]/\langle 1-xn\rangle=\Z[1/n].
$$
It also follows from \eqref{eq:DeltA} and another straightforward, but longer calculation \cite{ac1}*{Lemma 3.3} that
$$
\Delta_{A^-}\cong\Z[x,x^{-1}]/\langle x^3+(2n-1)x^2-(n+2)x+1\rangle.
$$
\end{ex}

The following theorem was proved by Krieger (\cite{krieger}*{Theorem~4.2}; see also~\cite{lind}*{Section 7.5} for a detailed algebraic treatment). 

\begin{thm}[Krieger \cite{krieger}]\label{kriegerthm}
Let $A$ and $B$ be two square nonnegative integer matrices. Then $A$ and $B$ are shift equivalent if and only if 
\[(\Delta_A, \Delta_A^+, \delta_A) \cong (\Delta_B, \Delta_B^+, \delta_B).\]
\end{thm}

\begin{rem}\label{rem:kriegerk0}
Let $E$ be a finite regular graph and $A_E$ its adjacency matrix. In view of Example \ref{ex:ultrale}, $\Delta_{A^t}=K_0(L(\ol{E}))$ and $\Delta_{A^t}^+=K_0(L(\ol{E}))_+$.
\end{rem}

\section{Graded rings and graded  \topdf{$K$}{K}-theory}\label{sec:gradK}

\numberwithin{equation}{subsection}

\subsection{Graded rings}\label{grringslablel}

Let $\Gamma$ be an abelian  group with identity denoted by $0$. A ring $A$ (possibly without unit)
is called a \emph{$\Gamma$-graded ring} if $ A=\bigoplus_{ \gamma \in \Gamma} A_{\gamma}$
such that each $A_{\gamma}$ is an additive subgroup of $A$ and $A_{\gamma}  A_{\delta}
\subseteq A_{\gamma+\delta}$ for all $\gamma, \delta \in \Gamma$. The group $A_\gamma$ is
called the $\gamma$-\emph{homogeneous component} of $A.$ When it is clear from context
that a ring $A$ is graded by the group $\Gamma,$ we simply say that $A$ is a  \emph{graded
ring}.  The set $\Gamma_A=\{ \gamma \in \Gamma \mid A_\gamma \not =
0 \}$ is called the \emph{support}  of $A$. We say that a $\Gamma$-graded ring $A$ is
\emph{trivially graded} if the support of $A$ is the trivial group
$0$---that is, $A_0=A$, so $A_\gamma=0$ for $\gamma \in \Gamma
\backslash \{0\}$. Any ring admits a trivial grading by any group. If $A$ is a $\Gamma$-graded ring that happens to be an algebra over a trivially graded ring $R$, then $A$ is called a \emph{graded algebra}
if $A_{\gamma}$ is an $R$-submodule for any $\gamma \in \Gamma$.
A $\Gamma$-graded ring $A=\bigoplus_{\gamma\in\Gamma}A_{\gamma}$ is called \emph{strongly graded} if
$A_{\gamma}A_{\delta}=A_{\gamma+\delta}$ for all $\gamma,\delta$ in $\Gamma$.

The elements of $$A_*:=\bigcup_{\gamma \in \Gamma} A_{\gamma}$$ in a graded ring $A$ are called
\emph{homogeneous elements} of $A.$ The nonzero elements of $A_\gamma$ are called
\emph{homogeneous of degree $\gamma$} and we write $|a| = \gamma$ for $a \in
A_{\gamma}\backslash \{0\}.$  An ideal $I$ of a graded ring $A$ is called a \emph{graded ideal} if it is generated by homogeneous elements. This is equivalent to  $I=\bigoplus_{ \gamma \in \Gamma} I\cap A_{\gamma}$. 
We say a $\Gamma$-graded ring $A$ has \emph{graded local units} if there is a set $\cU$ of homogeneous idempotents such that for any finite set of
homogeneous elements  $\{x_{1}, \cdots, x_{n}\}\subseteq A$, there exists $e\in \cU$ such that $\{x_{1}, \cdots, x_{n}\}\subseteq eAe$. Such a set $\cU$ is called a \emph{system of graded local units}. Remark that if $A$ has graded local units, then the set of all homogeneous idempotent elements of $A$ is a system of graded local units.

For a $\Gamma$-graded $\K$-algebra $R$ and $\ol{\a}=(\a_1, \cdots, \a_n)\in\Gamma^n$ we produce a $\Gamma$-grading on $\M_n(R)$ and denote this graded ring by $\mathbb{M}_{n}(R)(\ol{\alpha})$. 
For each $\gamma\in \Gamma$, the elements of $\mathbb{M}_{n}(R)(\ol{\alpha})_{\gamma}$ are the
$n\times n$-matrices over $R$ with the degree shifted as follows 
 \begin{equation}\label{eq:weightgrading}
     \mathbb{M}_{n}(R)(\ol{\a})_{\gamma}=
     {\begin{pmatrix}
 	R_{\gamma+\a_1-\a_1}& R_{\gamma+\a_2-\a_1}&\cdots& R_{\gamma+\a_n-\a_1}\\
  R_{\gamma+\a_1-\a_2}& R_{\gamma+\a_2-\a_2}&\cdots& R_{\gamma+\a_n-\a_2}\\
  \vdots& \vdots& \ddots& \vdots\\
  R_{\gamma+\a_1-\a_n}& R_{\gamma+\a_2-\a_n}&\cdots& R_{\gamma+\a_n-\a_n}
 \end{pmatrix}}.
 \end{equation} 

We have \[\mathbb{M}_{n}(R)(\ol{\a}) = \bigoplus_{\alpha \in \Gamma} \mathbb{M}_{n}(R)(\ol{\a})_{\gamma},\] and one can check that this makes 
$\mathbb{M}_{n}(R)(\ol{\a})$ a $\Gamma$-graded ring. Similarly, any infinite sequence $\ol{\a}\in\Gamma^\N$ defines a grading on $\M_\infty(R)$, and we write 
$\M_{\infty}(R)(\ol{\a})$ for the resulting graded ring.

If $A$ is a ring with local units, we call a right $A$-module $M$ \emph{unital} if $MA=M$, and write $\Modd A$ for the category of unital $A$-modules and module homomorphisms. If $A$ is $\Gamma$-graded with graded local units, we denote by $\Gr A$ the category of unital graded right $A$-modules with morphisms preserving the grading. For a graded right $A$-module $M$, we define the $\a$-\emph{shifted} graded right 
$A$-module $M(\a)$ as
\begin{equation*}M(\a)=\bigoplus_{\gamma\in \Gamma}M(\a)_{\gamma},
\end{equation*}
where $M(\a)_{\gamma}=M_{\a+\gamma}$.  That is, as an ungraded module, $M(\alpha)$ is a copy of
$M$, but the grading is shifted by $\alpha$. For $\a\in\Gamma$, the \emph{shift functor}
\begin{equation}\label{shiftshift}
\mathcal{T}_{\a}: \Gr A\longrightarrow \Gr A,\quad M\mapsto M(\a)
\end{equation}
is an isomorphism with the property that $\mathcal{T}_{\a}\mathcal{T}_{\beta}=\mathcal{T}_{\a+\beta}$ for $\a,\beta\in\Gamma$.

These shift functors are one of the main features of the graded theory that distinguishes it from the non-graded theory. The fact that one can shift a module (back and forth), introduces a dynamic into abstract algebra.

Let $\Z^{(\Gamma)}$ be the ring of finitely supported functions $\Gamma\to \Z$, equipped with pointwise operations. For example if $g\in\Gamma$ then the characteristic function $\chi_g$ of the set $\{g\}$ is in $\Z^{(\Gamma)}$, and $\{\chi_g\,\colon\, g\in \Gamma\}$ is a $\Z$-linear basis of $\Z^{(\Gamma)}$. Let $A$ be a $\Gamma$-graded ring.  The \emph{crossed (or smash) product} $\Gamma\hltimes A$ is the abelian group $\Z^{(\Gamma)}\otimes_\Z A$ equipped with the product defined below, where we write $\hltimes$ for $\otimes_\Z$ and $|a|$ for the degree of a homogeneous element. 
\[
(\chi_g\hltimes a)(\chi_h\hltimes b)=\delta_{h,g+|a|}\chi_g\hltimes ab\,\, (a\in A_*,\, b\in A).
\]
The group $\Gamma$ acts on $\Gamma\hltimes A$ by automorphisms, via
\[
g(\chi_h\hltimes a)=\chi_{g+h}\hltimes a.
\]
For a right $\Gamma\hltimes A$-module $N$, write $N^g$ for the right $\Gamma\hltimes A$-module with the same underlying group as $N$ and multiplication
\[
x\cdot_g a=xg(a)\,\, (x\in N,\, a\in A).
\]
\begin{prop}[\cite{ahls}*{Section 3.3} and \cite{ac}*{Section 3}]\label{prop:cross}
Let $A$ be a $\Gamma$-graded ring with graded local units. There is an equivalence of categories
\[
\Psi:\Gr A\to \Modd(\Gamma\hltimes A)
\]
with the property that, for each $g\in \Gamma$ and $M\in\Gr A$
\[
\Psi(M(g))=\Psi(M)^g.
\]    
\end{prop}
\begin{proof}[Sketch of the proof]
For $M\in \Gr A$, $\Psi(M)$ is the abelian group $M$, equipped with the following right $\Gamma\hltimes A$-module structure
\[
m\cdot(\chi_g\hltimes a)=m_ga
\]
If $f:M\to N$ is a homomorphism in $\Gr A$, then $\Psi(f)$ is the same function $f$. To show that $\Psi$ is an equivalence, one defines a functor $\Phi:\Modd(\Gamma\hltimes A)\to \Gr A$ as follows. Let $\cU$ be a system of graded local units of $A$. On an object $N$, $\Phi(N)$ is the same abelian group with the grading 
\[
N_g=\bigcup_{u\in\cU}N\cdot(\chi_g\hltimes u)
\]
and the $A$-module structure defined by 
\[
x\cdot a=\sum_{g\in \Gamma}x_g(\chi_g\hltimes a).
\]
If $f:N\to P$ is a homomorphism in $\Modd(\Gamma\hltimes A)$, and $x\in N$, define $\Phi(f)(x_g)=f(x)_g$. 
\end{proof}

As we will see the graded theory is the right setting to relate the theory of Leavitt path algebras to symbolic dynamics. These algebras have a natural $\mathbb Z$-graded structure and  have graded local units. We refer the reader to \cite{hazbook} for an account of graded ring theory.

\begin{ex}\label{ex:crocov}
Let $E$ be a row-finite graph and $\ol{E}$ its covering graph \eqref{eq:olE}. Then $\Z$ acts on $\ol{E}$ via $v_n\mapsto v_{n+1}$, $e_n\mapsto e_{n+1}$; this defines a $\Z$-action by $\ast$-automorphisms on $L(E)$, and the map
\[
L(\ol{E})\to \Z\hltimes L(E),\,\, x_n\mapsto \chi_{-n}\hltimes x\,\, (x\in E^0\sqcup E^1)
\]
is an isomorphism of algebras. Moreover it is \emph{$\Z$-equivariant} in the sense that it intertwines the $\Z$-actions on its domain and codomain \cite{ac}*{Proposition 2.7}.
\end{ex}

Below we state Dade's theorem which says that for strongly graded rings, the category $\Gr A$ is equivalent to the module category of $A_{0}$.

\begin{thm}[Dade\,\cite{dade}*{Theorem 2.8}]\label{thm:dade}
Let $R$ be a unital $\Gamma$-graded ring. Consider the functors 
\begin{gather}
 -\otimes_{R_{0}}R:\Modd R_0\to \Gr R,\,\, M\mapsto M\otimes_{R_0}R\label{map:ugrtogr}\\
(-)_0:\Gr R\to \Modd R_0,\,\, M\mapsto M_0 \label{map:grtougr}
\end{gather}
The following are equivalent.

\begin{enumerate}[\upshape(1)]

\item $R$ is strongly graded;

\item The functors \eqref{map:ugrtogr} and \eqref{map:grtougr} are inverse category equivalences. 
\end{enumerate}

\end{thm}

In the case of Leavitt path algebras of finite graphs, strong gradedness is characterized by the following theorem.

\begin{thm}[Clark, Hazrat and Rigby,\, \cite{strogra}*{Corollary 4.4}]\label{thm:lestrogra}
Let $E$ be a finite graph and $R$ a unital commutative ring. Then $L_R(E)$, equipped with its canonical $\Z$-grading, is strongly graded if and only if $E$ is regular. 
\end{thm}

For a characterization of strong gradedness of $L_R(E)$ for arbitrary graphs, see \cite{strogra}*{Theorem 4.2}.

\subsection{Graded \topdf{$K$}{K}-theory}\label{subsec:gradK}

Let $A$ be a $\Gamma$\!-graded unital ring and let $\mathcal V^{\gr}(A)$ denote the monoid of graded isomorphism classes of graded finitely generated   right projective modules over $A$ with direct sum as the addition operation. For a graded finitely generated projective $A$-module $P$, we denote the graded isomorphism class of $P$ by $[P]$, which is an element of $\mathcal V^{\gr}(A)$.  Thus for $[P], [Q] \in \mathcal V^{\gr}(A)$, we have $[P]+[Q]=[P\oplus Q]$. 
Recall that the shift functor $\mathcal T_\alpha:\Gr A\rightarrow \Gr A$, $M \mapsto M(\alpha)$ (see~\ref{shiftshift})  is an isomorphism which restricts to $\Pgrp A$, the  category of graded finitely generated  
right projective $A$-modules. Thus the abelian group $\Gamma$ acts on $\mathcal V^{\gr}(A)$ via 
\begin{equation}\label{hhffwwq}
(\alpha, [P]) \mapsto  [P(\alpha)].
\end{equation}
The \emph{graded Grothendieck group}, $K_0^{\gr}(A)$, is defined as the group completion of  the monoid $\mathcal V^{\gr}(A)$. 
Thus $K_0^{\gr}(A)$ naturally inherits a $\Gamma$\!-module structure via (\ref{hhffwwq}). Further, the image $K_0^{\gr}(A)_+$ of $\cV^{\gr}(A)$ is a $\Gamma$-submonoid, and thus the pre-order relation $x\ge y\iff x-y\in K_0^{\gr}(A)_+$ is preserved by the $\Gamma$-action. Hence  $K_0^{\gr}(A)$ is a \emph{pre-ordered $\Z[\Gamma]$-module}, \emph{pointed} by the class $[A]$.

One proves, using Proposition \ref{prop:cross}, that there is an isomorphism of pre-ordered $\Z[\Gamma]$-modules 
\begin{equation}\label{map:k0gr=k0hltimes}
K_0^{\gr}(A)\iso K_0(\Gamma\hltimes A).    
\end{equation}

\begin{rem}\label{ex:hikgr}
Let $A$ be a unital $\Gamma$-graded ring. The graded $K$-groups $K_n^{\gr}(A)$ ($n\in\Z$) are defined to be the $K$-groups of the category $\Pgrp A$; the shift action equips $K_n^{\gr}(A)$ with a $\Z[\Gamma]$-module structure. One uses Proposition \ref{prop:cross} to prove that $K_n^{\gr}(A)\cong K_n(\Gamma\hltimes A)$ as $\Z[\Gamma]$-modules \cite{ac}*{Theorem 3.4}.
\end{rem}

\begin{rem}
The unitalization $\tilde{A}$ of a not necessarily unital $\Gamma$-graded ring $A$ is again $\Gamma$-graded, with $\tilde{A}_\gamma=A_\gamma$ for $\gamma\ne0$ and $\tilde{A}_0=A_0\oplus \Z$. For $n\in\Z$, the graded $\Z[\Gamma]$ module $K_n^{\gr}(A)$ is defined as the kernel of $K_n^{\gr}(\tilde{A})\to K_n^{\gr}(\Z)$. The isomorphism of Remark \ref{ex:hikgr} holds whenever $A$ has graded local units.
\end{rem}

\begin{ex}\label{ex:kgrle}
Let $E$ be a row-finite graph, $\ol{E}$ its covering graph \eqref{eq:olE}, $\K$ a field, and $L=L_\K$. By \eqref{map:k0gr=k0hltimes} and Example \ref{ex:crocov}, 
$K_0^{\gr}(L(E))=K_0(L(\ol{E}))$, which is computed in Example \ref{ex:ultrale}. If $E$ is regular, it is the inductive limit \eqref{seq:colibare}; the fact, noted in Example \ref{ex:kle0}, that it agrees with $K_0(L(E)_0)$ is justified by Dade's theorem \ref{thm:dade}. 

The \emph{graded Bowen-Franks module} is the $\Z[x,x^{-1}]$-module
\[
\BF_{\gr}(E)=\coker (I-xA_E^t:\Z[x,x^{-1}]^{(\reg(E))}\to \Z[x,x^{-1}]^{(E^0)}).
\]
It is not hard to see from Example \ref{ex:ultrale} that if $E$ is any row-finite graph, then identifying $v_n$ with $vx^n$, one gets an isomorphism of $\Z[x,x^{-1}]$-modules 
\begin{equation*}\label{eq:k0grbfe}
K_0^{\gr}(L(E))\cong \BF_{\gr}(E).    
\end{equation*}
   
It was shown in \cite{ac}*{Corollary 5.4} that for any ring $R$ with local units and trivial $\Z$-grading, 
\[
K^{\gr}_n(L_R(E))\cong\BF_{\gr}(E)\otimes K_n(R) \,\, (n\in\Z).
\]
\end{ex}
\begin{ex}\label{ex:kmgrle}
Let $E$ be a finite regular graph, $m\ge 2$ and $\tau_m$ a generator of $\Z/m\Z$, written multiplicatively, so that $\Z[\Z/m\Z]=\Z[\tau_m]$. Put
\[
\BF_m(E)=\coker(I-\tau_mA_E^t:\Z[\tau_m]^{E^0}\to \Z[\tau_m]^{E^0}).
\]
If $R$ is a field or a PID, then by \cite{ac1}*{Lemma 4.2}, there is an isomorphism of $\Z[\tau_m]$-modules 
\[
K_0^{\Z/m\Z-\gr}(L_R(E))\cong\BF_m(E). 
\]
Further, the isomorphism maps $[L(E)]\mapsto [1_E]$. 
\end{ex}

\subsection{Non-existence of graded homomorphisms \topdf{$L_n\leftrightarrows L_{n-}$}{}}\label{subsec:nonex}

As a first application of graded $K$-theory, we sketch a proof of Theorem \ref{thm:noexistis}.
Let $m\ge 2$; as in \ref{ex:kmgrle}, we write $\Z[\tau_m]$ for the group ring of $\Z/m\Z$ over $\mathbb Z$. 
\begin{lem}[\cite{ac1}*{Lemmas 3.1 and 3.3}]\label{lem:bfmrn}
Let $m,n\ge 2$. Set
\[
\xi_n(x)=x^3+(2n-1)x^2-(n+2)x+1\in\Z[x].
\]
Then there are isomorphisms of $\Z[\tau_m]$-modules
\begin{gather}
\BF_m(\cR_n)\cong \Z/(n^m-1)\Z\label{eq:bfmrn}\\
\BF_m(\cR_{n^-})\cong \Z[\tau_m]/\langle\xi_n(\tau_m) \rangle.\label{eq:bfmrn-}
\end{gather}
The isomorphism \eqref{eq:bfmrn} sends $[1_{\cR_n}]\mapsto [1]$ and $\tau_m$ acts on the right hand side as multiplication by $n^{m-1}$. The isomorphism \eqref{eq:bfmrn-} sends $[1_{\cR_{n^-}}]\mapsto [1-n\tau_m]$.
\end{lem}

\begin{proof}[Sketch of the proof of Theorem \ref{thm:noexistis}] Since any $\Z$-graded homomorphism of $\Z$-graded algebras is also $\Z/2\Z$-graded, we may assume that $m\ge 2$. In view of Example \ref{ex:kmgrle} and Lemma \ref{lem:bfmrn}, it suffices to show that there are no $\Z[\tau_m]$-linear homomorphisms between the right-hand sides of \eqref{eq:bfmrn} and \eqref{eq:bfmrn-} preserving the distinguished element. This is done in \cite{ac1}*{Theorem 5.1 and Theorem 6.3}. One direction is easy; if $\phi:\Z[\tau_m]/\langle\xi_n(\tau_m)\rangle\to \Z/(n^m-1)\Z$ is a $\Z[\tau_m]$-module homomorphism such that  $\phi([1-n\tau_m])=[1]$, then 
\[
[1]=\phi([1])(1-n\tau_m)=\phi(1)(1-m^n)=0,
\]
a contradiction. The other direction is more involved; see \cite{ac1}*{Lemma 6.1, Proposition 6.2 and Theorem 6.3}.
\end{proof}

\begin{rem}\label{rem:noexistisinfty}
It follows from Theorem \ref{thm:noexistis} that $L_n$ and $L_{n^-}$ are not isomorphic as $\Z$-graded algebras. One can also derive this directly from Examples \ref{ex:DeltaA} and \ref{ex:kgrle}. Indeed, for $M=\Z[x,x^{-1}]/\langle x+1\rangle$, we have 
\[
K_0^{\gr}(L_n)\otimes_{\Z[x,x^{-1}]}M=\Z[1/n]\otimes_{\Z[x,x^{-1}]} M=\Z/(n+1)\Z
\]
while
\[
K_0^{\gr}(L_{n^-})\otimes_{\Z[x,x^{-1}]}M=\Z[x,x^{-1}]/\langle \xi(-1)\rangle =\Z/(3n+1)\Z.
\]
\end{rem}
\section{Graded classification and symbolic dynamics}\label{sec:gradsymb}

Since Leavitt path algebras are graded algebras, it is natural to take into account the grading of these algebras when trying to determine an invariant for them. In this section we will look at the graded Grothendieck group $K_0^{\gr}$ as such an invariant. The main conjectures pertaining to  Leavitt path algebras we consider in this section state that the graded Grothendieck group, along with its positive cone (and the position of the identity) is a complete invariant (Conjectures~\ref{conjalgiso}, \ref{conj:full} and \ref{conjmogr}).

\subsection{Polycephaly graphs and the graded classification conjecture}\label{subsec:poly}
We write $C_n$ for the graph consisting of a single cycle of length $n$. 
A graph having only one cycle $C$ such that all vertices are connected to some vertex in the support of $C$ is called a \emph{$C_n$-comet} if $|C|=n$. A \emph{multiheaded comet} is a graph consisting of a finite number of comets such that cycles are mutually disjoint and every vertex connects to at least one cycle. A finite graph $E$ is \emph{polycephaly} if each vertex connects to disjoint cycles, or to sinks or roses.   For formal definitions and more details  of  multiheaded comets and polycephaly graphs, see \cite{haz2013}*{page 296}. 

\begin{thm}[Hazrat~\cite{haz2013}*{Theorem 9}]\label{mani543}
Let $E$ and $F$ be polycephaly graphs. Then $L(E)\cong_{\gr} L(F)$ if and only if there is an order-preserving $\mathbb Z[x,x^{-1}]$-module isomorphism
\[\big (K_0^{\gr}(L(E)),[L(E)]\big ) \cong \big (K_0^{\gr}(L(F)),[L(F)]\big ).\] 
\end{thm}
\begin{proof}[Strategy of the proof] It was shown in \cite{haz2013}*{Theorem 2} that if $E$ is polycephaly, then $L(E)$ decomposes as a direct sum of matrix algebras, graded by certain sequences of weights as in \eqref{eq:weightgrading}, with one matrix algebra over $\K$ for each sink, one over $\K[x^n,x^{-n}]$ for each comet $C_n$ and one over $L_{n}$ for each $n$-petal rose in $E$ $(n\ge 2)$. Then $K_0^{\gr}$ was computed for these three types of algebras \cite{haz2013}*{Theorems 6, 7 and 8} as a pointed, partially ordered $\Z[x,x^{-1}]$-module. Finally the theorem was proved in \cite{haz2013}*{Theorem 9} using those calculations.     
\end{proof}

Theorem~\ref{mani543} motivated the following conjecture.

\begin{conj}[\emph{Graded Classification Conjecture} \cite{haz2013}]\label{conjalgiso}
Let $E$ and
$F$ be finite graphs. 

There is an order-preserving $\Z[x,x^{-1}]$-module
isomorphism $\phi: K_0^{\gr}(L(E)) \iso K_0^{\gr}(L(F))$ with 
$\phi([L(E)])=[L(F)]$ 
if and only if there exists a unital $\mathbb Z$-graded $\K$-isomorphism $\psi: L(E) \rightarrow L(F)$ such that $K^{\gr}_0(\psi) = \phi$.

\end{conj}

\begin{rem}\label{rem:nuniconjalgiso}
One can also formulate a version of Conjecture \ref{conjalgiso} for row-finite graphs with an infinite number of vertices, with the condition that $\phi$ preserves the order unit replaced by the condition that it preserves the generating interval \cite{hazlia}*{p. 450}. Then one may further strengthen the conjecture by requiring that $\psi$ be a $\ast$-homomorphism. This strong form of the conjecture was established by Hazrat and Va\v{s} in \cite{hazlia}*{Theorem 5.5} for a certain class $\cC$ of row-finite countable graphs and under some hypothesis on the ground field $\K$. A countable graph $E$ belongs to $\cC$ if and only if $L(E)$ decomposes as a direct sum of countably many summands, each of them a finite or countably infinite matrix algebra with coefficients in either $\K$ or $\K[x,x^{-1}]$. A graph theoretic description of such graphs is given in \cite{chain}*{Theorem 3.7}. The condition on ground field $\K$ is that it be $2$-proper and $\ast$-Pythagorean (see \cite{hazlia}*{p. 420} for a definition of these terms); for example $\C$, $\R$ and, in general, any quadratically closed $\ast$-subfield of $\C$ satisfy these conditions. 
\end{rem}

\subsection{The graded classification conjecture up to a twist}\label{subsec:twist}

Recall that an automorphism $g $ of a unital algebra $\mathcal A$ is said to be {\it locally inner}
if given any finite number of elements $x_1, \dots , x_n$ in $\mathcal A$ there is an inner automorphism $h$ of $\mathcal A$
such that $h(x_i)= \varphi (x_i)$ for $i=1, \dots , n$. It is well-known (\cite{goodearl}*{Lemma 15.23(b)}) that any automorphism $\varphi $ of a unital ultramatricial algebra $\mathcal A$,
such that $K_0(\varphi ) = \text{id}$, is locally inner.

Let $R$ be a ring with identity and $p$ an idempotent of $R$. Observe that if $\phi:R\to R$ is a not-necessarily unital ring homomorphism then for $p=\phi(1)$ we have $\phi(R)\subset pRp$. We call $\phi$
a \emph{corner isomorphism} if the induced homomorphism $\phi:R\lra pRp$ is an isomorphism. The \emph{corner skew Laurent polynomial ring}  \index{corner skew Laurent polynomial ring} associated to $(R,\phi)$, denoted by $R[t_{+},t_{-},\phi]$, is a unital ring constructed as follows. For $n\ge 0$, set $p_n=\phi^n(1)$. The elements of $R[t_{+},t_{-},\phi]$ are the formal expressions
\[t^j_{-}r_{-j} +t^{j-1}_{-}r_{-j+1}+\dots+t_{-}r_{-1}+r_0 +r_1t_{+}+\dots +r_it^i_{+},\,\, (i,j\ge 0).\]
Addition is componentwise, and multiplication is determined by the distributive law and the following rules.
\begin{equation}\label{oiy53} 
t_{-}t_{+} =1, \qquad t_{+}t_{-} =p, \qquad rt_{-} =t_{-}\phi(r),\qquad  t_{+}r=\phi(r)t_{+}.
\end{equation}

Corner skew Laurent polynomial rings are studied in~\cite{arabrucom}; their $K$-groups were calculated in \cite{abc}*{Theorem 3.6}. Assigning degrees $-1$ to $t_{-}$ and $1$ to $t_{+}$ makes $A:=R[t_{+},t_{-},\phi]$ a $\mathbb Z$-graded ring with $A=\bigoplus_{i\in \mathbb Z}A_i$, where for $p_i=\phi^i(1)$ $(i\ge 0)$ we have
\[
A_i=\left\{\begin{matrix}
Rp_it^i_{+} & \text{ for  } i>0,\\
t^{-i}_{-}p_{-i}R & \text{ for } i<0,\\
R & \text{ for } i=0.
\end{matrix}\right. 
\]

Setting $p=1$ and $\phi$ the identity map, $R[t_{+},t_{-},\phi]$ reduces to the familiar Laurent polynomial ring $R[t,t^{-1}]$.

  \begin{thm}[Ara, Gonz\'{a}lez-Barroso, Goodearl, Pardo~\cite{skew}*{Lemma 2.4}]\label{jhby67}
Let $A$ be a $\mathbb Z$-graded unital ring  possessing  a left invertible element $t_{+} \in A_1$. Then $t_{+}$ has a left inverse $t_{-}\in A_{-1}$, and for $p=t_+t_-$ and
\begin{align}\label{liyang}
\phi:A_0 &\longrightarrow pA_0p,\\
 a &\longmapsto t_{+}at_{-}\notag
\end{align} 
we have a ring isomorphism $A\cong A_0[t_{+},t_{-},\phi]$.
\end{thm}

Let $E$ be a finite graph without sources. Using Theorem~\ref{jhby67}, one can present $L(E)$ as a corner skew Laurent polynomial ring as follows. For each $v\in E^0$,  choose an edge $e_v$ such that $r(e_v)=v$ and consider $$t_{+}=\sum_{v\in E^0}e_v\in L(E)_1.$$ Then $t_{-}= t_+^*$ satisfies $t_-t_+=1$; set $p=t_+t_-$ and let
\begin{align*}
\alpha:L(E)_0 &\longrightarrow pL(E)_0 p\\
 a &\longmapsto t_{+}at_{-}.
\end{align*}
One checks that $\alpha$ is an isomorphism. Then thanks to Theorem~\ref{jhby67}, we have $$L(E)=L(E)_0[t_{+},t_{-},\alpha].$$
Now let $g$ be a locally inner automorphism of $L(E)_0$. Then we define $L^g(E)$ by
$$L^g(E)= L(E)_0[t_+,t_-, g\circ \alpha ].$$
The graded algebra $L^g(E)$ has the same graded $K$-theory as $L(E)$. Indeed the graded $K$-groups are the same because both are strongly graded with the same $0$-degree component, and the $\Z$-actions on them coincide because $g$ induces the identity in $K$-theory. We are now in a position to describe the following result of Ara and Pardo.

\begin{thm}[Ara-Pardo~\cite{apgrad}*{Theorem 4.1}]\label{theor:Kiffgr-iso unital}
If $E,F$ are finite essential graphs, then the following are equivalent:
\begin{enumerate}[\upshape(1)]
\item There is an order-preserving $\mathbb Z[x,x^{-1}]$-module isomorphism
$$\big (K_0^{\gr}(L(E)),[L(E)]\big ) \cong \big (K_0^{\gr}(L(F)),[L(F)]\big );$$ 
\item  There exists a locally inner automorphism $g$ of $L(E)_0$ such that  $L^g(E)\cong_{\gr}L(F)$.
\end{enumerate}
\end{thm}

\subsection{The fullness conjecture}\label{subsec:full}

One of the conjectures raised in~\cite{haz2013} is that the graded Grothendieck group $K_0^{\gr}$, as a functor  between the category of Leavitt path algebras and the category of 
pre-ordered $\mathbb Z[x,x^{-1}]$-modules with order unit is  full.

\begin{conj}[Hazrat, \cite{haz2013}]\label{conj:full}
For any order-preserving $\Z[x,x^{-1}]$-module homomorphism $\phi: K^{\gr}_0(L(E)) \rightarrow K^{\gr}_0(L(F))$ with 
$\phi([L(E)])=[L(F)]$, there exists a unital $\mathbb Z$-graded $\K$-homomorphism $\psi: L(E) \rightarrow L(F)$ such that $K^{\gr}_0(\psi) = \phi$.    
\end{conj}

 G. Arnone in \cite{guido2023} and L. Va\v{s} in \cite{vas} proved Conjecture~\ref{conj:full}, independently and with different approaches. In fact Va\v{s} considers graphs with finitely many vertices, but allows infinitely many edges, while Arnone requires that the graphs be finite, but shows that the lifting map can be chosen to be a diagonal preserving graded $*$-homomorphism.

\begin{thm}[Arnone~\cite{guido2023}, Va\v{s}~\cite{vas}]\label{thm:full}
Let $E$ and $F$ be finite graphs and  $\phi: K^{\gr}_0(L(E)) \rightarrow K^{\gr}_0(L(F))$ a pre-ordered $\mathbb Z[x,x^{-1}]$-module homomorphism with $\phi([L(E)])=[L(F)]$. Then there exists a unital $\mathbb Z$-graded homomorphism $\psi: L(E) \rightarrow L(F)$ such that $K^{\gr}_0(\psi) = \phi$. Furthermore if $\phi$ is injective, so is $\psi$.
\end{thm}

\subsection{\topdf{$C^*$}{C*}-version of the classification conjecture}\label{subsec:c*ver}
Write $\T$ for the circle group, $\gamma_E$ for the gauge circle action on $C^*(E)$, and $K^\T_0(C^*(E))$ for $\T$-equivariant $K$-theory~\cite{chrisp}.  Let $\ol{E}$ be as in \eqref{eq:times1}. There are canonical 
order-preserving isomorphisms  of $\Z[x,x^{-1}]$-modules (see~\cite{haz2013}*{p. 275}, \cite{haziso} and \cite{eilers2}*{Proof of Theorem A}) 
\begin{equation}\label{copenhag}
K^{\gr}_0(L(E)) \cong K_0(L(\ol{E})) \cong  K_0(C^*(\ol{E}))\cong K_0^{\mathbb T}(C^*(E)).
\end{equation}

Thus one can pose the analytic version of Conjecture~\ref{conjalgiso} as follows. 

\begin{conj}\label{conjanal}
Let $E$ and $F$ be finite graphs. Then there is an order-preserving $\mathbb Z[x,x^{-1}]$-module
isomorphism
\[ \phi: K^\T_0(C^*(E)) \iso K_0^\T(C^*(F)),\] 
with 
$\phi([C^*(E)])=[C^*(F)]$ 
 if and only if
 there is a $\T$-equivariant $\ast$-isomorphism $C^*(E) \iso C^*(F)$.
\end{conj}

In fact in Conjecture~\ref{conjalgiso} if the graded Grothendieck groups are isomorphic, one expects to have a graded $\ast$-isomorphism between the Leavitt path algebras. This stronger version of Conjecture~\ref{conjalgiso} implies Conjecture~\ref{conjanal}. For, if $K^\T_0(C^*(E)) \cong K_0^\T(C^*(F))$, then by (\ref{copenhag}), $K_0^{\gr}(L_{\mathbb C}(E)) \cong K_0^{\gr}(L_{\mathbb C}(F))$, so there is a graded $\ast$-isomorphism $L_{\mathbb C}(E)\cong_{\gr} L_{\mathbb C}(F)$. Passing to the completion, we get a $\T$-equivariant $\ast$-isomorphism $C^*(E) \iso C^*(F)$ (see~\cite{abramstomforde}*{Theorem~4.4}).

\subsection{The graded classification conjecture for amplified graphs}\label{subsec:ample}

A directed graph $E$ is called an \emph{amplified graph} if for any $v,w\in E^0$, the set of edges from $v$ to $w$ is either empty or infinite. The following theorem shows that the graded Grothendieck group precisely pinpoints the isomorphism class of amplified graphs; namely if the graded Grothendieck groups are isomorphic, then the associated graphs are isomorphic, and then of course so are their Leavitt path algebras.

\begin{thm}[Eilers, Ruiz, Sims~\cite{eilers2}*{Theorem A}]\label{thm-main}
Let $E$ and $F$ be countable amplified graphs and let $\K$ be a field.  Then the following are
equivalent.
\begin{enumerate}[\upshape(1)]
\item $E \cong F$;

\item There is an order-preserving $\mathbb Z[x,x^{-1}]$-module isomorphism $K_0^{\gr}(L_{\K}(E)) \cong K_0^{\gr}(L_{\K}(F));$
\item There is an order-preserving $\mathbb Z[x,x^{-1}]$-module isomorphism
    $K_0^{\mathbb T}(C^*(E)) \cong K_0^{\mathbb T}(C^*(F))$
    of $\mathbb T$-equivariant $K_0$-groups.
\end{enumerate}
\end{thm}

\begin{rem}\label{rem:isotoo}
It follows from Theorem \ref{thm-main} that statements (1), (2) and (3) are further equivalent to the existence of a $\Z$-graded isomorphism (of $\K$-algebras or of rings) $L_{\K}(E)\cong L_{\K}(F)$.
\end{rem}

\subsection{Graded Morita equivalence}\label{subsec:gradmor} 

The notion of graded Morita theory plays an important role in this subject. It relates the theory of Leavitt path algebras to the theory of symbolic dynamics. One of the main conjectures in the theory is around the notion of graded Morita equivalence (Conjecture~\ref{conjmogr}). 

\begin{defi} \label{grdeffsa} 
Let $A$ and $B$ be unital $\Gamma$-graded rings.

\begin{enumerate}

\item A functor $\phi:\Gr A \rightarrow \Gr B$ is called a \emph{graded functor} if $\phi \mathcal T_{\alpha} = \mathcal T_{\alpha} \phi$ for all $\alpha \in \Gamma$.

\item A graded functor $\phi:\Gr A \rightarrow \Gr B$ is called a \emph{graded equivalence} if there is a graded functor $\psi:\Gr B \rightarrow \Gr A$ such that $\psi \phi \cong 1_{\Gr A}$ and $\phi \psi \cong 1_{\Gr B}$.

\item If there is a graded equivalence between $\Gr A$ and $\Gr B$, we say $A$ and $B$ are {\it graded equivalent} or {\it graded Morita equivalent} and write 
$\Gr A \approx_{\gr} \Gr B$.

\item A functor $\phi: \Modd A \rightarrow \Modd B$ is called a \emph{graded functor} if  there is a graded functor $\phi': \Gr A \rightarrow \Gr B$ such that the following diagram, where the vertical functors are forgetful functors, commutes.
\begin{equation}\label{njhysi}
\xymatrix{
\Gr A \ar[r]^{\phi'} \ar[d]_{U}& \Gr B \ar[d]^{U}\\
\Modd A \ar[r]^{\phi}  & \Modd B.
}
\end{equation}

The functor $\phi'$ is called an \emph{associated graded functor} of $\phi$.

\item A functor $\phi:\Modd A \rightarrow \Modd B$ is called a \emph{graded equivalence} 
if it is graded and it is an equivalence.
\end{enumerate}
\end{defi}

Recall that for any ring $R$ we let $\M_\infty(R)$ denote the (nonunital) ring consisting of those countably infinite square matrices over $R$ that contain at most finitely many nonzero entries (see Section \ref{subsec:k0}). Furthermore, we say an idempotent $e$ in a ring $R$ is \emph{full}, if the two-sided ideal generated by $e$ is the whole ring $R$.

The equivalence between statements (1)-(5) in the following theorem can be found in \cite{hazbook}*{Theorem 2.3.8} and the equivalence of the latter with (6) is a recent result of Abrams, Ruiz and Tomforde~\cite{abramsmori}. They are all parallel to the classical results in the non-graded Morita theory, with the added complexity of shifts being present in the graded theory. Once we assume the grade group $\Gamma$ is trivial, we recover the classical results.

\begin{thm}\label{grmorim11} 
Let $A$ and $B$ be two unital $\Gamma$-graded rings. The following are equivalent/

\begin{enumerate}[\upshape(1)]

\item $\Gr A$ is graded equivalent to $\Gr B$.

\item $\Modd A$ is graded equivalent to $\Modd B$.

\item $B\cong_{\gr} \End_A(P)$ for a graded $A$-progenerator $P$.

\item $B\cong_{\gr} e \M_n(A)(\ol \delta) e$ for a full homogeneous idempotent $e \in \M_n(A)(\ol \delta)$, where $\ol \delta=(\delta_1,\dots,\delta_n)$, $\delta_i \in \Gamma$. 

\item There are graded finitely generated projective $(A,B)$-bimodule $P$ and $(B,A)$-bimodule $Q$, such that $P\otimes_B Q\cong_{\gr} A$ and $Q\otimes_A P\cong_{\gr} B$ as $(A,A)$ and $(B,B)$-bimodules, respectively.

\item  There exists a sequence $(\gamma_m)_{m\in \mathbb N}$ in $\Gamma$ such that $\M_\infty(A)$ is graded isomorphic
to $\M_\infty(B)((\gamma_m)_{m\in \mathbb N})$.

\end{enumerate}
 \end{thm}

\subsection{Graded Morita equivalence and shift equivalence }\label{subsec:art}

Recall the notion of shift equivalence from Section~\ref{sec:symb}. The theorem below provides a bridge between the theory of symbolic dynamics and Leavitt path algebras through graded Grothendieck groups. 

 \begin{thm}\label{h99}
Let $E$ and $F$ be finite regular graphs and let $A_E$ and $A_F$ be their adjacency matrices. Then
$A_E$ is shift equivalent to $A_F$ if and only if there is an order-preserving $\mathbb Z[x,x^{-1}]$-module  isomorphism
$K_0^{\gr}(L(E)) \cong K_0^{\gr}(L(F))$.
\end{thm}
\begin{proof}
Observe that two matrices are shift equivalent if and only if their transposes are. Hence in view of Remark \ref{rem:kriegerk0}, the theorem follows from Theorem \ref{kriegerthm}.
\end{proof}

The following Proposition  relies heavily  on  Williams' Theorem ~\ref{willmnhfhf} on graph moves. 
One can check that these graph moves preserve the associated Leavitt path algebras up to graded Morita equivalence. Thus we have the following proposition. 

\begin{prop}[Hazrat~\cite{hazd}*{Proposition 15}]\label{hgysweet}
Let $E$ and $F$ be finite graphs.
\begin{enumerate}[\upshape(1)]
\item If $E$ is essential and $F$ is obtained from an in-splitting or an out-splitting of the graph $E$, then 
$L(E)$ is graded Morita equivalent to $L(F)$.

\item For essential $E$ and $F$, if the adjacency matrices $A_E$ and $A_F$  are strongly shift equivalent then $L(E)$ is graded Morita equivalent to $L(F)$.

\item  If $E$ and $F$ are regular and $L(E)$ is graded Morita equivalent to $L(F)$, then the adjacency matrices $A_E$ and $A_F$  are shift equivalent.

\end{enumerate}
\end{prop}

\begin{rem}
Conjecture~\ref{conjmogr} says that the converse of assertion (3) in Proposition~\ref{hgysweet} also holds.     
\end{rem}

\begin{ex}
Consider the graph below on the left. Williams' out/in-splitting starts with a partition $E^1=\bigsqcup_{i=1}^nE^1_i$ of the set of edges. The latter induces, for each $v\notin\sink(E)$, a partition $\s^{-1}\{v\}=\bigsqcup_{i=1}^n(E^1_i\cap s^{-1}(v))$ of the set of outgoing  edges, and similary, if $v\notin\sour(E)$, a partition on the set $r^{-1}\{v\}$ incoming edges. If we start with the partition of $E^1$ which consists of one edge in each term, and apply Williams' out-splitting change of the graph, we obtain the graph on the right. Then our result implies that the Leavitt path algebras associated to these graphs are graded Morita equivalent. In fact it was proved in \cite{question}*{Theorem 2.8} that out-splitting yields graded isomorphic Leavitt path algebras.

\includegraphics[scale=0.15]{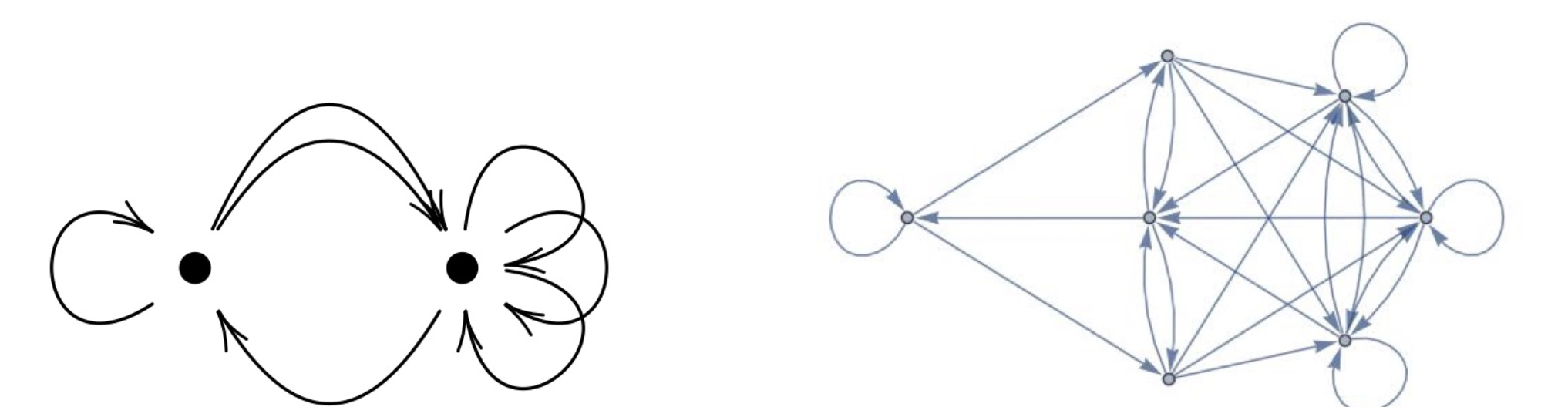}
    
\end{ex}

\subsection{Graded Morita equivalence, derived equivalence and the singularity category}\label{subsec:moritasing}

Let $\K$ be a field and $E$ a finite graph. The \emph{path algebra} $\K E$ is the tensor algebra of the $\K^{E^0}$-bimodule $\K^{E^1}$. Let $\K E\quitri J$ be the two-sided ideal generated by the paths of length greater or equal to $1$. Set $A(E)=\K E/J^2$. 
There are close relations between Leavitt path algebras and the representation theory of $A(E)$. For a unital algebra $A$, the {\it singularity category of $A$}, introduced by Buchweitz and Orlov,  is defined as the quotient category 
\[
\DD_{\singg}(A)=D^b(\Modd A)/\Perf(A),
\]
where $D^b(\Modd A)$ is the bounded derived category of $A$ and $\Perf(A)$ is the subcategory of perfect complexes. One can show that the global dimension of $A$  is finite if and only if $\DD_{\singg}(A)$ is trivial. 

The following theorem, due to Chen and Yang~\cite{chen}, uses results of Paul Smith~\cite{smith3} and Hazrat~\cite{haz2013}. Recall that two differential graded algebras (dgas) are \emph{derived equivalent} if the derived categories of their categories of differential graded modules are equivalent. Any $\Z$-graded algebra can be considered as a dga with trivial differential. In particular this applies to Leavitt path algebras; in part (2) of Theorem \ref{thm:yache}, derived equivalence should be understood in this sense.

\begin{thm}[Chen, Yang~\cite{chen}]\label{thm:yache}
Let $E$ and $F$ be finite regular graphs. Then the following statements are equivalent. 

\begin{enumerate}[\upshape(1)]

\item The Leavitt path algebras $L(E)$ and $L(F)$ are graded Morita equivalent.

\item The Leavitt path algebras $L(E)$ and $L(F)$, regarded as dgas with trivial differential, are derived equivalent.

\item The singularity categories $\DD_{\singg}(A(E))$  and $\DD_{\singg}(A(F))$  are triangulated equivalent.
\end{enumerate}
\end{thm}

Assembling all the results so far, we have the following conjecture:

\begin{conj}\label{conjmogr}
Let $E$ and $F$ be finite regular graphs. Then the following are equivalent.

\begin{enumerate}[\upshape(1)]

\item There is an isomorphism of partially ordered $\mathbb Z[x,x^{-1}]$-modules
$K_0^{\gr}(L(E)) \iso K_0^{\gr}(L(F))$;

\item There is a gauge-preserving Morita equivalence between $C^*(E)$ and   $C^*(F)$;

\item There is a graded Morita equivalence between $L(E)$ and $L(F)$;

\item The singularity categories $\DD_{\singg}(A(E))$  and $\DD_{\singg}(A(F))$  are triangulated equivalent;

\item The adjacency matrices $A_E$ and $A_F$ are shift equivalent. 

\end{enumerate}

\end{conj}

\begin{rem}\label{rem:conjemoves} 
By Theorem \ref{thm:yache}, conditions (3) and (4) of Conjecture \ref{conjmogr} are equivalent. By Proposition \ref{hgysweet}, (3) implies (5). Assuming the response to Question \ref{quest:semuove} is positive, (5) would further be equivalent to the existence of a finite sequence of moves that carries $E$ to $F$. 
\end{rem}

\subsection{Talented monoids and the Morita equivalence conjecture for meteor graphs}\label{talentedmoni}
 For a finite essential graph $E$, the Krieger monoid $\Delta_{A_E^t}^+$ is the positive cone of $K_0^{\gr}(L(E))$ (this follows by the same argument as in \ref{h99}). The latter is the \emph{talented monoid} $T_E$; $K_0^{\gr}(L(E))$ is its group completion. Therefore by Theorem \ref{kriegerthm}, from a shift equivalence of finite essential graphs $E$ and $F$ we obtain a $\mathbb Z$-monoid isomorphism of talented monoids $T_E\cong T_F$. This isomorphism gives us control over elements of the monoids (such as minimal elements, atoms, etc.) and consequently on the geometry of the graphs (the number of cycles, their lengths, etc.). 
 
We start with a definition of talented monoids via generators and relations. 

\begin{defi}
    For a row-finite graph $E$, the \emph{talented monoid} $T_E$ is the commutative monoid generated by $\{v(i) \mid v\in E^0, i\in \mathbb Z\}$ subject to
\begin{equation}\label{monoidrelation2}
v(i)=\sum_{e\in s^{-1}(v)}r(e)(i+1), 
\end{equation}
for every $v\in \reg(E)$ and $i \in \mathbb Z$.
\end{defi}

The above relations define a congruence relation  which also respects the action of $\mathbb Z$. Namely,  the monoid $T_E$ is equipped by a natural $\mathbb Z$-action: $${}^n v(i)=v(i+n)$$ for $n,i \in \mathbb Z$. Proposition~5.7 of \cite{ahls} relates this monoid to the theory of Leavitt path algebras: there is a  $\mathbb Z$-module isomorphism 
$T_E \cong \mathcal{V}^{\gr}(L(E))$.  In fact we have 
\begin{align*}
T_E &\cong\, \mathcal V^{\gr}(L(E)) \cong\mathcal V(L(\ol E)),\\
v(i) &\longmapsto   \big [(L(E)v\big) (i)]    \longmapsto [L(\ol E) v_i]
 \end{align*} 
Here, as in \eqref{eq:olE}, $\ol E$ is the covering graph of $E$. 
Thus the talented monoid $T_E$  is conical and cancellative (\cite{ahls}*{Section 5}). Here we use left $L(E)$-modules to construct $\mathcal V^{\gr}(L(E))$ so that it matches the definition of our talented monoid. 

The following theorem shows how the geometry of graphs translates into properties of their talented monoids. A graph $E$ is said to satisfy \emph{Condition (L)} if every cycle in $E$ has an exit. A closed path $\alpha=e_1\cdots e_n$ based at a vertex $v$ is \emph{simple} if $s(e_i)=v$ then $i=1$. We say that $E$ satisfies \emph{Condition (K)} if for every vertex $v$ in the support of a closed simple path there exist at least two distinct closed simple paths based at $v$.

\begin{thm}[\cite{hazli}]\label{conLm}
Let $E$ be a row-finite graph and $T_E$ the talented monoid associated to $E$. 

\begin{enumerate}[\upshape(1)]
\item The graph $E$ satisfies Condition (L) if and only if $\mathbb Z$ acts freely on $T_E$.

\item The graph $E$ satisfies Condition (K) if and only if $\mathbb Z$ acts freely on any quotient of $T_E$ by an $\Z$-order-ideal. 
\end{enumerate}
\end{thm}

A \emph{meteor graph} is an essential connected graph consisting of two disjoint cycles and the paths connecting these cycles. We can give a positive answer to Conjecture~\ref{conjmogr} for the case of meteor graphs.

\begin{thm}[Cordeiro, Gillaspy, Gon\c{c}alves, Hazrat~\cite{oberwol}] \label{alergy1}
    Let \(E\) and \(F\) be essential graphs, where $E$ is a meteor graph. Then the following are equivalent.
    \begin{enumerate}[\upshape(1)]

        \item The Leavitt path algebras \(L(E)\) and \(L(F)\) are graded Morita equivalent.

\item  The graph $C^*$-algebras $C^*(E)$ and $C^*(F)$ are equivariantly Morita equivalent.

        \item There is an order-preserving 
        $\mathbb Z[x,x^{-1}]$-module isomorphism
$K_0^{\gr}(L(E))\iso K_0^{\gr}(L(F))$.       
        \item The talented monoids \(T_E\) and \(T_F\) are \(\mathbb{Z}\)-isomorphic.

        \item The graphs \(E\) and \(F\) are shift equivalent.
        
        \item The graphs \(E\) and \(F\) are strongly shift equivalent.
        
    \end{enumerate}
    \label{thm:main}
\end{thm}

\begin{proof}[Sketch of the proof] An 
order-preserving $\mathbb Z[x,x^{-1}]$-module isomorphism
 $K_0^{\gr}(L(E))\iso K_0^{\gr}(L(F))$, gives an isomorphism on the level of talented monoids $T_E \cong T_F$. Then, a careful analysis of the monoid isomorphism allows us to show that $E$ and $F$ can be transformed, via in- and out-splitting and their inverses, 
into each other. Applying Williams' Theorem~\ref{willmnhfhf}, we obtain that the graphs $E$ and $F$ are strongly shift equivalent and, as a consequence of Proposition~\ref{hgysweet}, we get that $L(E)$ is graded Morita equivalent to $L(F)$.
\end{proof}

Since a meteor graph consists of only two cycles and paths emitting from one cycle lands exclusively into the other cycle, using in/out-splitting one can transform this graph into a ``standard form'' graph, without changing the $K$-theoretical data. This allows us to compare two meteor graphs by converting them into the standard form. It is not clear how to adopt this approach to graphs which have more than two cycles. Although Theorem~\ref{alergy1} very likely holds for graphs with disjoint cycles, this is yet to be established. 

\subsection{Graded homotopy classification}\label{subsec:homotopy2}

Let $B=\bigoplus_{i\in \mathbb Z}B_i$ be a $\Z$-graded algebra and $B[t]$ the polynomial ring graded so that $B_n[t]$ is the graded component of degree $n$. Let $\phi,\psi:A\to B$ be graded homomorphisms of $\Z$-graded algebras. We say that $\phi$ and $\psi$ are \emph{graded homotopic}, and write $\phi\sim_{\gr}\psi$, if there is a finite sequence of elementary homotopies $H_0:\phi\to \phi_1,\dots, H_n:\phi_n\to \psi$ such that each $H_i$ is a graded homomorphism. The following theorem is concerned with primitive graphs. A finite graph $E$ is \emph{primitive} if there is an $N$ such that for every pair of vertices $(v,w)$ of $E$ there is a path $\alpha$ of length $N$ with $s(\alpha)=v$ and $r(\alpha)=w$.

The talented monoid $T_E$ (Section \ref{talentedmoni}) of a finite graph $E$ is pointed by $1_E=\sum_{v\in E^0}v(0)$. 
\begin{thm}[Arnone \cite{guidotopy}*{Theorem 8.1}]\label{thm:guidotopy}  
Let $E$ and $F$ be two finite, primitive graphs. 
If there exists an isomorphism of $\Z$-monoids $T_E\iso T_F$ sending $1_E\mapsto 1_F$
then the algebras $L(E)$ and $L(F)$ are graded homotopy
equivalent.
\end{thm}
The proof of Theorem \ref{thm:guidotopy} uses graded bivariant algebraic $K$-theory techniques \cites{kkg,ac}. 
The hypothesis that $E$ and $F$ be primitive is needed so that certain idempotents of $L(E)$ arising from edges are full.

\subsection{Compatibility conditions on shift equivalence that imply graded Morita equivalence}\label{subsec:compa}

In what follows, if $X$ is a finite set we write $\K X$ for the $\K$-vector space with basis $X$. We will occasionally consider $\K X$ as a $\K$-algebra with the product
$x\cdot y=\delta_{x,y}x$. For example $\K E^0$ is isomorphic to the subalgebra of $L(E)$ generated by $E^0$. 

Let $E$ and $F$ be finite regular graphs. If $A=A_E$ and $B=A_F$ are shift equivalent with lag $\ell$, then by definition there are matrices $R\in\N_0^{E^0\times F^0}$ and $S\in\N_0^{F^0\times E^0}$ satisfying the identities \eqref{eq:equifuera} and \eqref{eq:equidentro}. Let $G=G(R)$ be a set with  
$$
|R|_1=\sum_{(v,w)\in E^0\times F^0}R_{v,w}
$$
elements, and let $s:G\to E^0$ and $r:G\to F^0$ be such that 
\[
R_{v,w}=\#\big(s^{-1}(v)\cap r^{-1}(w)\big).
\]
Next consider the $\K$-vector space $M_R=\K G(R)$. We give $M_R$ a $(\K E^0,\K F^0)$-bimodule structure as follows: for each $g\in G=G(E)$, $v\in E^0$ and $w\in F^0$, we put
\[
vg=\delta_{v,s(g)}g,\, gw=\delta_{r(g),w}g.
\]
Similarly we define the $\K$-vector space $M_S$ and make it into a $(\K{F^0},\K{E^0})$-bimodule. Let
\begin{gather*}
s,r:G(R)\times_{F^0}G(S)=\{(g,h)\in G(R)\times G(S)\,\colon\, r(g)=s(h)\}\to E^0\sqcup F^0\\
 s(g,h)=s(g),\,\, r(g,h)=r(h).
\end{gather*}
Then for each pair $(v_1,v_2)\in E^0\times E^0$
\[
\#\{(g,h)\in G(R)\times_{F^0}G(S)\,\colon\, s(g,h)=v_1,\, r(g,h)=v_2\}=(RS)_{v_1,v_2}=A^\ell_{v_1,v_2}=\#v_1E^\ell v_2.
\]
Hence from  $RS=A^\ell$ we obtain
\begin{equation}\label{eq:shifteqbimod1}
M_R\otimes_{\ell{F^0}}M_S=\bigoplus_{(g,h)\in G(R)\times_{F^0}G(S)}\K\cdot(g,h)\cong \K E^\ell\cong (\K{E^1})^{\otimes_{\K{E^0}}^\ell}.
\end{equation}
The isomorphism \ref{eq:shifteqbimod1} translates one of the conditions in the definition of shift equivalence to bimodule terms. We may similarly translate all of them; setting $M=M_R$, $N=M_S$ we obtain bimodule isomorphisms 
\begin{gather}\label{eq:gse}
M\otimes_{\K{F^0}}N\overset{\omega_E}{\iso} (\K{E^1})^{\otimes_{\K{E^0}}^\ell},\,\, N\otimes_{\K{E^0}}M\overset{\omega_F}{\iso} (\K{F^1})^{\otimes_{\K{F^0}}^\ell},\\
\K{E^1}\otimes_{\K{E^0}}M\overset{\sigma_M}{\iso}M\otimes_{\K{F^0}}\K{F^1},\,\,
 \K{F^1}\otimes_{\K{F^0}}N\overset{\sigma_N}{\iso}N\otimes_{\K{E^0}}\K{E^1}. \nonumber
\end{gather}

\begin{prop}[\cite{recast}*{Theorem 3.9}]\label{prop:recast}
Let $E$ and $F$ be regular graphs. The following are equivalent. 

 \begin{enumerate}[\upshape(1)]

\item  $E$ and $F$ are shift equivalent;

\item  There exist bimodules ${}_{\K{E^0}}M_{\K{F^0}}$ and ${}_{\K{F^0}}N_{\K{E^0}}$ and bimodule isomorphisms \eqref{eq:gse}.

\end{enumerate} 
\end{prop}
\begin{rem} 
Both the construction of $M$ and $N$ and  the isomorphisms \eqref{eq:gse} were first done in \cite{lens} in the setting of $C^*$-algebras, where the analogue of a bimodule is a $C^*$-correspondence. The analogue of Proposition \ref{prop:recast} in that context is \cite{lens}*{Proposition 3.5}.  
\end{rem}

To state the next theorem we need more notation. Let $E$, $F$ and $G$ be finite graphs and $\sigma:\K{E^1}\otimes_{\K{E^0}}M\to M\otimes_{\K{F^0}}\K{F^1}$ and 
$\psi:\K{F^1}\otimes_{\K{F^0}}N\to N\otimes_{\K{G^0}}\K{G^1}$ bimodule homormophisms. The composite
\[
\sigma\#\psi:\K{E^1}\otimes_{\K{E^0}}M\otimes_{\K{F^0}}N\overset{\sigma\otimes \id_{N}}{\lra} M\otimes_{\K{F^0}}\K{F^1}\otimes_{\K{F^0}}N\overset{\id_{M}\otimes\psi}{\lra} M\otimes_{\K{F^0}} N\otimes_{\K{G^0}}\K{G^1}
\]
is a homomorphism of $(\K{E^0},\K{G^0})$-bimodules. (Here and elsewhere, we omit the associativity isomorphisms of the tensor product.)

\begin{thm}[Abrams, Ruiz and Tomforde \cite{recast}*{Theorem 6.7}]\label{thm:recast} Let $E$ and $F$ be finite regular graphs. Assume there are bimodules and isomorphisms as in \eqref{eq:gse} such that the following diagrams commute
\begin{gather}
\xymatrix{
\K{E^1}{\otimes}_{\K{E^0}}M\otimes_{\K{F^0}}N\ar[rr]^{\sigma_M\#\sigma_N}\ar[dr]^{\id\otimes\omega_E}&&M\otimes_{\K{F^0}}N\otimes_{\K{E^0}}\K{E^1}\ar[dl]^{\omega_E\otimes\id}\\
&\K{E^1}^{\otimes_{\K{E^0}}m+1}&}\label{diag:M}\\
\xymatrix{
\K{F^1}{\otimes}_{\K{F^0}}N\otimes_{\K{E^0}}M\ar[rr]^{\sigma_N\#\sigma_M}\ar[dr]^{\id\otimes\omega_F}&&N\otimes_{\K{E^0}}M\otimes_{\K{F^0}}\K{F^1}\ar[dl]^{\omega_F\otimes\id}\\
&\K{F^1}^{\otimes_{\K{F^0}}m+1}&}\label{diag:N}
\end{gather}
Then $L(E)$ and $L(F)$ are graded Morita equivalent.
\end{thm}
\begin{proof}[Sketch of the proof] Consider the $(\K E^0,L(F))$-bimodule $P=M\otimes_{\K F^0}L(F)$ and the $(\K F^0,L(E))$-bimodule $Q=N\otimes_{K E^0}L(E)$. Both of these are $\Z$-graded, with the gradings induced by those of $L(E)$ and $L(F)$. Following an idea of  Ery\"{u}zl\"{u} \cite{eryu} in the $C^*$-algebra context, one uses $\sigma_M$ to extend the left (graded) $\K E^0$-module structure on $P$ to a left graded $L(E)$-module structure. For example multiplication by edges is defined by the composite of the maps induced by $\sigma_M$ and by multiplication in $L(F)$
\[
\K{E^1}\otimes_{\K{E^0}}M\otimes_{\K{F^0}}L(F)\overset{\sigma_M\otimes\id}{\iso}M\otimes_{\K{F^0}}\K{F^1}\otimes_{\K{F^0}}L(F)\overset{\id\otimes\cdot}{\lra}M\otimes_{\K{F^0}}L(F)
\]
Similarly, the left $\K F^0$-module structure on $Q$ is promoted to a left $L(F)$-module structure using $\sigma_N$. The composite
\[
P\otimes_{L(F)}Q=M\otimes_{\K{F^0}}N\otimes_{\K{E^0}}L(E)\overset{\omega_E\otimes\id_{L(E)}}{\lra}\K{E^{m+1}}\otimes_{\K{E^0}}L(E)\overset{\cdot}{\to}L(E).
\]
is clearly an ungraded isomorphism  of $(\K E^0,L(E))$-bimodules. Furthermore, using the commutativity of \eqref{diag:M} one checks that it is also left $L(E)$-linear and so an $(L(E),L(E))$-bimodule isomorphism. By symmetry, $\omega_F$ and \eqref{diag:N} 
give an  ungraded $(L(F),L(F))$-bimodule isomorphism $Q\otimes_{L(E)}P\cong L(F)$. It follows by standard Morita theory that $P\otimes_{L(F)}:\operatorname{Mod} L(F)\to \operatorname{Mod} L(E)$ is a graded equivalence. By Theorem \ref{grmorim11}, this implies 
that $L(E)\sim_{\gr}L(F)$. 
\end{proof}

\begin{prop}[\cite{recast}*{Proposition 6.9}]\label{prop:ssegse}
Let $E$ and $F$ be finite regular graphs. If $E$ and $F$ are strongly shift equivalent, then the hypotheses of Theorem \ref{thm:recast} are satisfied.
\end{prop}
\begin{proof}
We may assume there are matrices $R\in\N_0^{E^0\times F^0}$ and $S\in \N_0^{F_0\times E^0}$ such that $RS=A$ and $SR=B$. Let $M=M_R$ and $N=M_S$ as in the discussion above \eqref{eq:gse}. Choose bimodule isomorphisms
\[
\omega_E:M\otimes_{\K{F^0}}N\iso \K{E^1},\,\, \omega_F:N\otimes_{\K{E^0}}M\iso \K{F^1}.
\]
Set 
\[
\sigma_M=(\id_M\otimes\omega_F)\circ(\omega_E^{-1}\otimes\id_M),\, \sigma_N=(\id_N\otimes\omega_E)\circ(\omega_F^{-1}\otimes\id_F).
\]
A straightforward calculation shows that 
\[
\sigma_M\#\sigma_N=\omega_E^{-1}\otimes\omega_E,\,\, \sigma_N\#\sigma_M=\omega_F^{-1}\otimes\omega_F.
\]
The proposition is now immediate. 
\end{proof}

\begin{rem}
A different compatibility condition for shift equivalence was considered by Carlsen, Dor-On and Eilers in \cite{lens}. They expressed their condition in terms of the choices of bijections between $\K$-linear bases used above to define the isomorphisms \eqref{eq:gse}. In terms of the isomorphisms \eqref{eq:gse} themselves, it boils down to requiring that
\[
\xymatrix{\K E^\ell\otimes_{\K E^0}M\ar[r]^(.4){\id\otimes\sigma_M}&\K E^{\ell-1}\otimes_{\K E^0}M\otimes_{\K F^0}\K F^1
\ar[r]^(.9){\id\otimes\sigma_M\otimes\id}&\ar@{.}[r]&\ar[r]^(.3){\sigma_M\otimes\id}&M\otimes_{\K F^0}\K F^\ell\\
&M\otimes_{\K F^0}N\otimes_{\K E^0}M\ar[lu]^{\omega_E\otimes \id}\ar[rrru]^{\id\otimes\omega_F}}
\]
(where $\ell$ is the lag of the shift equivalence) and the analogous triangle with $F$ and $E$ and $N$ and $M$ switched, both commute. The authors show in \cite{lens}*{Theorem 7.3} that the latter, as well as another related condition, are both equivalent to requiring that $E$ and $F$ be strongly shift equivalent. On the other hand, it is not known whether the
hypotheses of Theorem~\ref{thm:recast} are equivalent to shift equivalence, or strong shift equivalence, or perhaps something in between those two conditions.
\end{rem}

\section{Filtered \topdf{$K$}{K}-theory, shift equivalence, and \topdf{$C^*$}{C*}-Morita equivalence}\label{sec:filk}
\numberwithin{equation}{section}
Following the early work of R\o rdam~\cite{ror}  and Restorff~\cite{restorff},  it became clear that one way to preserve enough information in the presence of ideals in a $C^*$-algebra, is to further consider the $K$-groups of the ideals, their subquotients, and how they are related to each other via the six-term excision sequence. Over the next ten years since \cites{ror,restorff} appeared, this approach, which is now called \emph{filtered $K$-theory},  was subsequently investigated and further developed by Eilers, Restorff, Ruiz and S\o rensen~\cites{errs4,errs2}, who showed that the lattice of gauge invariant prime ideals and their subquotient $K$-groups can be used as an invariant. For a graph $C^*$-algebra $C^*(E)$, we denote this invariant  by $\FK_{0,1}(C^*(E))$. In their major work~\cite{errs3}, Eilers, Restorff, Ruiz and S\o rensen proved that filtered $K$-theory is a complete invariant for unital graph $C^*$-algebras. Thus for finite graphs $E$ and $F$,
\begin{equation}\label{miser1}
    \FK_{0,1}(C^*(E))\cong \FK_{0,1}(C^*(F)) \Longrightarrow C^*(E) \cong_{\text{Morita eq.}} C^*(F).
\end{equation}

In~\cite{errs} the four authors introduced filtered $K$-theory in the purely algebraic setting and showed that if two Leavitt path algebras over $\mathbb C$ have isomorphic filtered algebraic $K$-theory then the associated graph $C^*$-algebras have isomorphic filtered $K$-theory.
\begin{equation}\label{miser2}
    \FK_{0,1}(L_{\mathbb C}(E))\cong \FK_{0,1}(L_{\mathbb C}(F)) \Longrightarrow \FK_{0,1}(C^*(E))\cong \FK_{0,1}(C^*(F)).
\end{equation}

 In~\cite{arahazli} it was shown that graded $K$-theory determines filtered $K$-theory up to a certain precisely defined quotient. 
 \begin{equation}\label{miser3}
     K_0^{\gr}(L_{\mathbb C}(E)) \cong K_0^{\gr}(L_{\mathbb C}(F)) \Longrightarrow
 \FKbar_{0,1}(L_{\mathbb C}(E))\cong \FKbar_{0,1}(L_{\mathbb C}(F)). 
 \end{equation}

  This shows the richness of the graded Grothendieck group as an invariant. Namely, the single group $K_0^{\gr}(L(E))$ of the Leavitt path algebra $L(E)$ associated to a graph $E$, with coefficients in a field $\K$, contains all the information about the $K_0$ groups and the quotient groups $\ol{K}_1$ of $K_1$  and how they are related via the long exact sequence of $K$-theory. The important link established in~\cite{arahazli}, was that the existence of an isomorphism at the level of this quotient filtered $K$-theory of Leavitt path algebras still implies the existence of an isomorphism at the level of the corresponding filtered $K$-theory of $C^*$-algebras (compare this with~(\ref{miser2})):
  \begin{equation}\label{miser4}
   \FKbar_{0,1}(L_{\mathbb C}(E))\cong \FKbar_{0,1}(L_{\mathbb C}(F)) 
  \Longrightarrow \FK_{0,1}(C^*(E))\cong \FK_{0,1}(C^*(F)).
     \end{equation}

  Using these connections one can prove the following proposition.

\begin{prop}[\cite{arahazli}]\label{bfg1998d}
Let $E$ and $F$ be finite regular graphs. If $E$ and $F$ are shift equivalent, then the $C^*$-algebras $C^*(E)$ and $C^*(F)$ are Morita equivalent. 
\end{prop}
\begin{proof}
Since $E$ and $F$ are shift equivalent, we have an isomorphism of Krieger's dimension groups 
\[ (\Delta_E, \Delta_E^+, \delta_E) \cong (\Delta_F, \Delta_F^+, \delta_F).\]
But by Corollary~\ref{h99}, Krieger's dimension group for the graph $E$ coincides with the graded Grothendieck group $K_0^{\gr}(L(E))$. We thus obtain an order-preserving $\Z[x,x^{-1}]$-module isomorphism $K_0^{\gr}(L(E))\cong K_0^{\gr}(L(F))$. By (\ref{miser2}) and (\ref{miser4}), the filtered $K$-theory of the corresponding graph $C^*$-algebras are  isomorphic. Now~(\ref{miser1})
gives that the $C^*$-algebras $C^*(E)$ and $C^*(F)$ are Morita equivalent. 
\end{proof}

The following diagram summarizes what is known and what could be considered as conjectural links.

\begin{equation*}
\resizebox{\hsize}{!}{
$
\xymatrix@=3pt{
&&&&&&&&& A_E \cong_{\text{shift eq.}}  \ar@{<->}[dd]  A_F\\
\\
    \textcolor{red}{L(E) \approx_{\text{gr Morita eq}} L(F)} \ar@{<->}[rrr] 
 &&& \DD_{\singg}(A(E)) \cong \DD_{\singg}(A(F)) \ar@{->}[rr]&&
  \textcolor{red}{K_0^{\gr}(L(E)) \cong   \ar@{->}[dd] K_0^{\gr}(L(F))}  {\ar@/_2pc/@{.>}[lllll]_{\bf Conjecture~\ref{conjmogr}}} \ar@{<->}[rr] &&
  \Delta_E \approx   \Delta_F  {\ar@{<->}[rr]} && 
 T_E \approx_{\Z} T_F \\
 \\
 &&&&&  \FKbar_{0,1}(L(E)) \cong  \ar@{->}[dd] \FKbar_{0,1}(L(F))\\
 \\
&&&&& \FK_{0,1}(C^*(E)) \cong \ar@{<->}[dd] \FK_{0,1}(C^*(F))\\
\\
&&&&& C^*(E) \cong_{\text{Morita eq.}} C^*(F)
}$
}
\end{equation*}


\begin{bibdiv}
\begin{biblist}



\bib{question}{article}{
   author={Abrams, G.},
   author={\'{A}nh, P. N.},
   author={Louly, A.},
   author={Pardo, E.},
   title={The classification question for Leavitt path algebras},
   journal={J. Algebra},
   volume={320},
   date={2008},
   number={5},
   pages={1983--2026},
   issn={0021-8693},
   review={\MR{2437640}},
   doi={10.1016/j.jalgebra.2008.05.020},
}
\bib{lpabook}{book}{
   author={Abrams, Gene},
   author={Ara, Pere},
   author={Siles Molina, Mercedes},
   title={Leavitt path algebras},
   series={Lecture Notes in Mathematics},
   volume={2191},
   publisher={Springer, London},
   date={2017},
   pages={xiii+287},
   isbn={978-1-4471-7343-4},
   isbn={978-1-4471-7344-1},
   review={\MR{3729290}},
}
\bib{aap05}{article}{
   author={Abrams, Gene},
   author={Aranda Pino, Gonzalo},
   title={The Leavitt path algebra of a graph},
   journal={J. Algebra},
   volume={293},
   date={2005},
   number={2},
   pages={319--334},
   issn={0021-8693},
   review={\MR{2172342}},
   doi={10.1016/j.jalgebra.2005.07.028},
}
\bib{chain}{article}{
    AUTHOR = {Abrams, Gene},
    author={Aranda Pino, Gonzalo},
    author={Perera, Francesc},
    author={Siles Molina, Mercedes},
     TITLE = {Chain conditions for {L}eavitt path algebras},
   JOURNAL = {Forum Math.},
    VOLUME = {22},
      YEAR = {2010},
    NUMBER = {1},
     PAGES = {95--114},
      ISSN = {0933-7741},
       DOI = {10.1515/FORUM.2010.005},
       URL = {https://doi.org/10.1515/FORUM.2010.005},
}
\bib{flow}{article}{
   author={Abrams, Gene},
   author={Louly, Adel},
   author={Pardo, Enrique},
   author={Smith, Christopher},
   title={Flow invariants in the classification of Leavitt path algebras},
   journal={J. Algebra},
   volume={333},
   date={2011},
   pages={202--231},
   issn={0021-8693},
   review={\MR{2785945}},
   doi={10.1016/j.jalgebra.2011.01.022},
}

\bib{recast}{article}{
 author={Abrams, Gene},
   author={Ruiz, Efren},
   author={Tomforde, Mark},
	title={Recasting the Hazrat conjecture: Relating shift equivalence to graded Morita equivalence},
eprint={arXiv:2311.02896},
}
\bib{abramstomforde}{article}{
   author={Abrams, Gene},
   author={Tomforde, Mark},
   title={Isomorphism and Morita equivalence of graph algebras},
   journal={Trans. Amer. Math. Soc.},
   volume={363},
   date={2011},
   number={7},
   pages={3733--3767},
   issn={0002-9947},
   review={\MR{2775826}},
   doi={10.1090/S0002-9947-2011-05264-5},
}



\bib{abramsmori}{article}{
   author={Abrams, Gene},
   author={Ruiz, Efren},
   author={Tomforde, Mark},
   title={Morita equivalence for graded rings},
   journal={J. Algebra},
   volume={617},
   date={2023},
   pages={79--112},
   issn={0021-8693},
   review={\MR{4513781}},
   doi={10.1016/j.jalgebra.2022.10.036},
}
\bib{arabrucom}{article}{
   author={Ara, Pere},
   author={Brustenga, Miquel},
   title={$K_1$ of corner skew Laurent polynomial rings and applications},
   journal={Comm. Algebra},
   volume={33},
   date={2005},
   number={7},
   pages={2231--2252},
   issn={0092-7872},
   review={\MR{2153218}},
   doi={10.1081/AGB-200063582},
}


\bib{arajazz}{article}{
	author = {Ara, Pere},
 author = {Brustenga, Miquel},
	title = {Module theory over Leavitt path algebras and K-theory},
	year = {2010},
	journal = {Journal of Pure and Applied Algebra},
	volume = {214},
	number = {7},
	pages = {1131 -- 1151},
	doi = {10.1016/j.jpaa.2009.10.001},
	
}

\bib{abc}{article}{
   author={Ara, Pere},
   author={Brustenga, Miquel},
   author={Corti\~{n}as, Guillermo},
   title={$K$-theory of Leavitt path algebras},
   journal={M\"{u}nster J. Math.},
   volume={2},
   date={2009},
   pages={5--33},
   issn={1867-5778},
   review={\MR{2545605}},
}

\bib{aratenso}{article}{
   author={Ara, Pere},
   author={Corti\~{n}as, Guillermo},
   title={Tensor products of Leavitt path algebras},
   journal={Proc. Amer. Math. Soc.},
   volume={141},
   date={2013},
   number={8},
   pages={2629--2639},
   issn={0002-9939},
   review={\MR{3056553}},
   doi={10.1090/S0002-9939-2013-11561-3},
}
\bib{skew}{article}{
   author={Ara, P.},
   author={Gonz\'{a}lez-Barroso, M. A.},
   author={Goodearl, K. R.},
   author={Pardo, E.},
   title={Fractional skew monoid rings},
   journal={J. Algebra},
   volume={278},
   date={2004},
   number={1},
   pages={104--126},
   issn={0021-8693},
   review={\MR{2068068}},
   doi={10.1016/j.jalgebra.2004.03.009},
}

\bib{agop}{article}{
   author={Ara, P.},
   author={Goodearl, K. R.},
   author={Pardo, E.},
   title={$K_0$ of purely infinite simple regular rings},
   journal={$K$-Theory},
   volume={26},
   date={2002},
   number={1},
   pages={69--100},
   issn={0920-3036},
   review={\MR{1918211}},
   doi={10.1023/A:1016358107918},
}
\bib{arahazli}{article}{
   author={Ara, Pere},
   author={Hazrat, Roozbeh},
   author={Li, Huanhuan},
   title={Graded $K$-theory, filtered $K$-theory and the classification of
   graph algebras},
   journal={Ann. K-Theory},
   volume={7},
   date={2022},
   number={4},
   pages={731--795},
   issn={2379-1683},
   review={\MR{4560379}},
   doi={10.2140/akt.2022.7.731},
}
\bib{ahls}{article}{
   author={Ara, Pere},
   author={Hazrat, Roozbeh},
   author={Li, Huanhuan},
   author={Sims, Aidan},
   title={Graded Steinberg algebras and their representations},
   journal={Algebra Number Theory},
   volume={12},
   date={2018},
   number={1},
   pages={131--172},
   issn={1937-0652},
   review={\MR{3781435}},
   doi={10.2140/ant.2018.12.131},
}

\bib{amp}{article}{
   author={Ara, P.},
   author={Moreno, M. A.},
   author={Pardo, E.},
   title={Nonstable $K$-theory for graph algebras},
   journal={Algebr. Represent. Theory},
   volume={10},
   date={2007},
   number={2},
   pages={157--178},
   issn={1386-923X},
   review={\MR{2310414}},
   doi={10.1007/s10468-006-9044-z},
}

\bib{apgrad}{article}{
   author={Ara, P.},
   author={Pardo, E.},
   title={Towards a K-theoretic characterization of graded isomorphisms
   between Leavitt path algebras},
   journal={J. K-Theory},
   volume={14},
   date={2014},
   number={2},
   pages={203--245},
   issn={1865-2433},
   review={\MR{3319704}},
   doi={10.1017/is014006003jkt269},
}

\bib{guidotopy}{article}{
 author={Arnone, Guido},
 title={Graded homotopy classification of Leavitt path algebras},
 eprint={arXiv:2309.06312},
}
\bib{guido2023}{article}{
   author={Arnone, Guido},
   title={Lifting morphisms between graded Grothendieck groups of Leavitt
   path algebras},
   journal={J. Algebra},
   volume={631},
   date={2023},
   pages={804--829},
   issn={0021-8693},
   review={\MR{4594882}},
   doi={10.1016/j.jalgebra.2023.05.018},
}

\bib{ac}{article}{
   author={Arnone, Guido},
   author={Corti\~{n}as, Guillermo},
   title={Graded $K$-theory and Leavitt path algebras},
   journal={J. Algebraic Combin.},
   volume={58},
   date={2023},
   number={2},
   pages={399--434},
   issn={0925-9899},
   review={\MR{4634308}},
   doi={10.1007/s10801-022-01184-5},
}
\bib{ac1}{article}{
   author={Arnone, Guido},
   author={Corti\~{n}as, Guillermo},
   title={Non-existence of graded unital homomorphisms between Leavitt
   algebras and their Cuntz splices},
   journal={J. Algebra Appl.},
   volume={22},
   date={2023},
   number={4},
   pages={Paper No. 2350084, 10},
   issn={0219-4988},
   review={\MR{4553219}},
   doi={10.1142/S0219498823500846},
}
\bib{tokel2}{article}{
   author={Carlsen, Toke Meier},
   title={$\ast$-isomorphism of Leavitt path algebras over $\mathbb Z$},
   journal={Adv. Math.},
   volume={324},
   date={2018},
   pages={326--335},
   issn={0001-8708},
   review={\MR{3733888}},
   doi={10.1016/j.aim.2017.11.018},
}

\bib{lens}{article}{
   author={Carlsen, Toke Meier},
   author={Dor-On, Adam},
   author={Eilers, S\o ren},
   title={Shift equivalences through the lens of Cuntz-Krieger algebras},
   journal={Anal. PDE},
   volume={17},
   date={2024},
   number={1},
   pages={345--377},
   issn={2157-5045},
   review={\MR{4702320}},
   doi={10.2140/apde.2024.17.345},
}

\bib{chen}{article}{
   author={Chen, Xiao-Wu},
   author={Yang, Dong},
   title={Homotopy categories, Leavitt path algebras, and Gorenstein
   projective modules},
   journal={Int. Math. Res. Not. IMRN},
   date={2015},
   number={10},
   pages={2597--2633},
   issn={1073-7928},
   review={\MR{3352249}},
   doi={10.1093/imrn/rnu008},
}

\bib{gpdgen}{article}{
   author={Clark, Lisa Orloff},
   author={Farthing, Cynthia},
   author={Sims, Aidan},
   author={Tomforde, Mark},
   title={A groupoid generalisation of Leavitt path algebras},
   journal={Semigroup Forum},
   volume={89},
   date={2014},
   number={3},
   pages={501--517},
   issn={0037-1912},
   review={\MR{3274831}},
   doi={10.1007/s00233-014-9594-z},
}
\bib{strogra}{article}{
   author={Clark, Lisa Orloff},
   author={Hazrat, Roozbeh},
   author={Rigby, Simon W.},
   title={Strongly graded groupoids and strongly graded Steinberg algebras},
   journal={J. Algebra},
   volume={530},
   date={2019},
   pages={34--68},
   issn={0021-8693},
   review={\MR{3938862}},
   doi={10.1016/j.jalgebra.2019.03.030},
}
\bib{oberwol}{article}{
author={Cordeiro, L.G.},
author={Gillaspy, E.},
author={Gon\c{c}alves, D.},
author={Hazrat, R},
title={Williams' Conjecture holds for meteor graphs}, 
eprint={arXiv:2304.05862},
}

\bib{friendly}{article}{
   author={Corti\~{n}as, Guillermo},
   title={Algebraic v. topological $K$-theory: a friendly match},
   conference={
      title={Topics in algebraic and topological $K$-theory},
   },
   book={
      series={Lecture Notes in Math.},
      volume={2008},
      publisher={Springer, Berlin},
   },
   isbn={978-3-642-15707-3},
   date={2011},
   pages={103--165},
   review={\MR{2762555}},
   doi={10.1007/978-3-642-15708-0\_3},
}
\bib{B13}{book}{
   author={Corti\~{n}as, Guillermo},
   title={\'Algebra II+1/2. Notas de teor\'\i a de \'algebras},
   series={Cursos y Seminarios de Matem\'atica, Serie B},
   volume={13},
   publisher={Departamento de Matemática, Facultad de Ciencias Exactas y Naturales,
Universidad de Buenos Aires},
date={2021},
}
\bib{bullift}{article}{
   author={Corti\~{n}as, Guillermo},
   title={Lifting graph $C^*$-algebra maps to Leavitt path algebra maps},
   journal={Bull. Lond. Math. Soc.},
   volume={54},
   date={2022},
   number={6},
   pages={2188--2201},
   issn={0024-6093},
   review={\MR{4528618}},
}
\bib{classinvo}{article}{
   author={Corti\~{n}as, Guillermo},
   title={Classifying Leavitt path algebras up to involution preserving
   homotopy},
   journal={Math. Ann.},
   volume={386},
   date={2023},
   number={3-4},
   pages={2107--2157},
   issn={0025-5831},
   review={\MR{4612414}},
   doi={10.1007/s00208-022-02436-2},
}
\bib{cm1}{article}{
author={Corti\~nas, Guillermo},
author={Montero, Diego},
title={Algebraic bivariant $K$-theory and Leavitt path algebras},
journal={J. Noncommut. Geom.},
volume={15},
   date={2021},
   number={1},
   pages={113--146},
   DOI={10.4171/jncg/397},
}
\bib{cm2}{article}{
   author={Corti\~{n}as, Guillermo},
   author={Montero, Diego},
   title={Homotopy classification of Leavitt path algebras},
   journal={Adv. Math.},
   volume={362},
   date={2020},
   pages={106961, 26},
   issn={0001-8708},
   review={\MR{4050584}},
   doi={10.1016/j.aim.2019.106961},
}
\bib{chriswi}{article}{
   author={Corti\~{n}as, Guillermo},
    author={Phillips, N. Christopher},
    title={Algebraic K-theory and properly infinite $C^\ast$-algebras},
    eprint={arXiv:1402.3197},
}
\bib{ct}{article}{
   author={Corti\~{n}as, Guillermo},
   author={Thom, Andreas},
   title={Bivariant algebraic $K$-theory},
   journal={J. Reine Angew. Math.},
   volume={610},
   date={2007},
   pages={71--123},
   issn={0075-4102},
   review={\MR{2359851}},
   doi={10.1515/CRELLE.2007.068},
}
\bib{On}{article}{
   author={Cuntz, Joachim},
   title={Simple $C\sp*$-algebras generated by isometries},
   journal={Comm. Math. Phys.},
   volume={57},
   date={1977},
   number={2},
   pages={173--185},
   issn={0010-3616},
   review={\MR{0467330}},
}
\bib{ck2}{article}{
   author={Cuntz, J.},
   title={A class of $C\sp{\ast} $-algebras and topological Markov chains.
   II. Reducible chains and the Ext-functor for $C\sp{\ast} $-algebras},
   journal={Invent. Math.},
   volume={63},
   date={1981},
   number={1},
   pages={25--40},
   issn={0020-9910},
   review={\MR{0608527}},
   doi={10.1007/BF01389192},
}
\bib{ck}{article}{
   author={Cuntz, Joachim},
   author={Krieger, Wolfgang},
   title={A class of $C\sp{\ast} $-algebras and topological Markov chains},
   journal={Invent. Math.},
   volume={56},
   date={1980},
   number={3},
   pages={251--268},
   issn={0020-9910},
   review={\MR{0561974}},
   doi={10.1007/BF01390048},
}
\bib{cmr}{book}{
   author={Cuntz, Joachim},
   author={Meyer, Ralf},
   author={Rosenberg, Jonathan M.},
   title={Topological and bivariant $K$-theory},
   series={Oberwolfach Seminars},
   volume={36},
   publisher={Birkh\"{a}user Verlag, Basel},
   date={2007},
   pages={xii+262},
   isbn={978-3-7643-8398-5},
   review={\MR{2340673}},
}
\bib{dade}{article}{
   author={Dade, Everett C.},
   title={Group-graded rings and modules},
   journal={Math. Z.},
   volume={174},
   date={1980},
   number={3},
   pages={241--262},
   issn={0025-5874},
   review={\MR{0593823}},
   doi={10.1007/BF01161413},
}
\bib{david}{book}{
   author={Davidson, Kenneth R.},
   title={$C^*$-algebras by example},
   series={Fields Institute Monographs},
   volume={6},
   publisher={American Mathematical Society, Providence, RI},
   date={1996},
   pages={xiv+309},
   isbn={0-8218-0599-1},
   review={\MR{1402012}},
   doi={10.1090/fim/006},
}
\bib{dritom}{article}{
   author={Drinen, D.},
   author={Tomforde, M.},
   title={The $C^*$-algebras of arbitrary graphs},
   journal={Rocky Mountain J. Math.},
   volume={35},
   date={2005},
   number={1},
   pages={105--135},
   issn={0035-7596},
   review={\MR{2117597}},
   doi={10.1216/rmjm/1181069770},
}
\bib{dritom2}{article}{
   author={Drinen, D.},
   author={Tomforde, M.},
   title={Computing $K$-theory and ${\rm Ext}$ for graph $C^*$-algebras},
   journal={Illinois J. Math.},
   volume={46},
   date={2002},
   number={1},
   pages={81--91},
   issn={0019-2082},
   review={\MR{1936076}},
}
\bib{eilers2}{article}{
   author={Eilers, S\o ren},
   author={Ruiz, Efren},
   author={Sims, Aidan},
   title={Amplified graph $C^*$-algebras II: Reconstruction},
   journal={Proc. Amer. Math. Soc. Ser. B},
   volume={9},
   date={2022},
   pages={297--310},
   review={\MR{4446255}},
   doi={10.1090/bproc/112},
}

\bib{errs2}{article}{
   author={Eilers, S\o ren},
   author={Restorff, Gunnar},
   author={Ruiz, Efren},
   title={On graph $C^*$-algebras with a linear ideal lattice},
   journal={Bull. Malays. Math. Sci. Soc. (2)},
   volume={33},
   date={2010},
   number={2},
   pages={233--241},
   issn={0126-6705},
   review={\MR{2666426}},
}

\bib{errs}{article}{
   author={Eilers, S\o ren},
   author={Restorff, Gunnar},
   author={Ruiz, Efren},
   author={S\o rensen, Adam P. W.},
   title={Filtered $K$-theory for graph algebras},
   conference={
      title={2016 MATRIX annals},
   },
   book={
      series={MATRIX Book Ser.},
      volume={1},
      publisher={Springer, Cham},
   },
   isbn={978-3-319-72298-6},
   isbn={978-3-319-72299-3},
   date={2018},
   pages={229--249},
   review={\MR{3792523}},
}


\bib{errs4}{article}{
   author={Eilers, S\o ren},
   author={Restorff, Gunnar},
   author={Ruiz, Efren},
   author={S\o rensen, Adam P. W.},
   title={Geometric classification of graph $\rm C^*$-algebras over finite
   graphs},
   journal={Canad. J. Math.},
   volume={70},
   date={2018},
   number={2},
   pages={294--353},
   issn={0008-414X},
   review={\MR{3759003}},
   doi={10.4153/CJM-2017-016-7},
}

\bib{errs3}{article}{
   author={Eilers, S\o ren},
   author={Restorff, Gunnar},
   author={Ruiz, Efren},
   author={S\o rensen, Adam P. W.},
   title={The complete classification of unital graph $C^*$-algebras:
   geometric and strong},
   journal={Duke Math. J.},
   volume={170},
   date={2021},
   number={11},
   pages={2421--2517},
   issn={0012-7094},
   review={\MR{4302548}},
   doi={10.1215/00127094-2021-0060},
}


\bib{elliott}{article}{
   author={Elliott, George A.},
   title={On the classification of inductive limits of sequences of
   semisimple finite-dimensional algebras},
   journal={J. Algebra},
   volume={38},
   date={1976},
   number={1},
   pages={29--44},
   issn={0021-8693},
   review={\MR{0397420}},
   doi={10.1016/0021-8693(76)90242-8},
}

\bib{kkg}{article}{
   author={Ellis, Eugenia},
   title={Equivariant algebraic $kk$-theory and adjointness theorems},
   journal={J. Algebra},
   volume={398},
   date={2014},
   pages={200--226},
   issn={0021-8693},
   review={\MR{3123759}},
   doi={10.1016/j.jalgebra.2013.09.023},
}
\bib{eryu}{article}{
author={Ery\"{u}zl\"{u}, Menev\c{s}e},
title={Passing {$C^*$}-correspondence Relations to the Cuntz-Pimsner algebras},
journal={M\"{u}nster J. Math. },
volume={15},
number={2},
pages={441--471},
}
\bib{franks}{article}{ 
author={Franks, J.},
title={Flow equivalence of subshifts of finite type}, 
journal={Ergodic Theory Dynam. Systems},
volume={4},
date={1984},
pages={53--66},
}
\bib{goodearl}{book}{
   author={Goodearl, K. R.},
   title={von Neumann regular rings},
   edition={2},
   publisher={Robert E. Krieger Publishing Co., Inc., Malabar, FL},
   date={1991},
   pages={xviii+412},
   isbn={0-89464-632-X},
   review={\MR{1150975}},
}

\bib{haziso}{article}{
   author={Hazrat, R.},
   title={A note on the isomorphism conjectures for Leavitt path algebras},
   journal={J. Algebra},
   volume={375},
   date={2013},
   pages={33--40},
   issn={0021-8693},
   review={\MR{2998945}},
   doi={10.1016/j.jalgebra.2012.11.017},
}

\bib{haz2013}{article}{
   author={Hazrat, R.},
   title={The graded Grothendieck group and the classification of Leavitt
   path algebras},
   journal={Math. Ann.},
   volume={355},
   date={2013},
   number={1},
   pages={273--325},
   issn={0025-5831},
   review={\MR{3004584}},
   doi={10.1007/s00208-012-0791-3},
}


\bib{hazbook}{book}{
   author={Hazrat, Roozbeh},
   title={Graded rings and graded Grothendieck groups},
   series={London Mathematical Society Lecture Note Series},
   volume={435},
   publisher={Cambridge University Press, Cambridge},
   date={2016},
   pages={vii+235},
   isbn={978-1-316-61958-2},
   review={\MR{3523984}},
   doi={10.1017/CBO9781316717134},
}



\bib{hazd}{article}{
   author={Hazrat, R.},
   title={The dynamics of Leavitt path algebras},
   journal={J. Algebra},
   volume={384},
   date={2013},
   pages={242--266},
   issn={0021-8693},
   review={\MR{3045160}},
   doi={10.1016/j.jalgebra.2013.03.012},
}


\bib{hazli}{article}{
   author={Hazrat, Roozbeh},
   author={Li, Huanhuan},
   title={The talented monoid of a Leavitt path algebra},
   journal={J. Algebra},
   volume={547},
   date={2020},
   pages={430--455},
   issn={0021-8693},
   review={\MR{4040730}},
   doi={10.1016/j.jalgebra.2019.11.033},
}
\bib{hazlia}{article}{
 AUTHOR = {Hazrat, Roozbeh},
 author={Va\v{s}, Lia},
     TITLE = {{$K$}-theory classification of graded ultramatricial algebras
              with involution},
   JOURNAL = {Forum Math.},
    VOLUME = {31},
      YEAR = {2019},
    NUMBER = {2},
     PAGES = {419--463},
      ISSN = {0933-7741},
       DOI = {10.1515/forum-2017-0268},
       URL = {https://doi.org/10.1515/forum-2017-0268},
}
\bib{huang}{article}{
   author={Huang, Danrun},
   title={Automorphisms of Bowen-Franks groups of shifts of finite type},
   journal={Ergodic Theory Dynam. Systems},
   volume={21},
   date={2001},
   number={4},
   pages={1113--1137},
   issn={0143-3857},
   review={\MR{1849604}},
   doi={10.1017/S0143385701001535},
}

\bib{johsor}{article}{
author={Rune Johansen},
author={Adam P. W. S\o rensen},
title={The Cuntz splice does not preserve $*$-isomorphism of Leavitt path algebras over $\Z$},
journal={J. Algebra},
volume={220},
date={2016},
pages={3966--3983},
}

\bib{kamiput}{article}{
   author={Kaminker, Jerome},
   author={Putnam, Ian},
   title={$K$-theoretic duality of shifts of finite type},
   journal={Comm. Math. Phys.},
   volume={187},
   date={1997},
   number={3},
   pages={509--522},
   issn={0010-3616},
   review={\MR{1468312}},
   doi={10.1007/s002200050147},
}

\bib{kr1}{article}{
   author={Kim, K. H.},
   author={Roush, F. W.},
   title={Williams's conjecture is false for reducible subshifts},
   journal={J. Amer. Math. Soc.},
   volume={5},
   date={1992},
   number={1},
   pages={213--215},
   issn={0894-0347},
   review={\MR{1130528}},
   doi={10.2307/2152756},
}
\bib{kr2}{article}{
   author={Kim, K. H.},
   author={Roush, F. W.},
   title={The Williams conjecture is false for irreducible subshifts},
   journal={Ann. of Math. (2)},
   volume={149},
   date={1999},
   number={2},
   pages={545--558},
   issn={0003-486X},
   review={\MR{1689340}},
   doi={10.2307/120975},
}
\bib{KP}{article}{
   author={Kirchberg, Eberhard},
   author={Phillips, N. Christopher},
   title={Embedding of exact $C^*$-algebras in the Cuntz algebra $\scr O_2$},
   journal={J. Reine Angew. Math.},
   volume={525},
   date={2000},
   pages={17--53},
   issn={0075-4102},
   review={\MR{1780426}},
   doi={10.1515/crll.2000.065},
}

\bib{krieger}{article}{
   author={Krieger, Wolfgang},
   title={On dimension functions and topological Markov chains},
   journal={Invent. Math.},
   volume={56},
   date={1980},
   number={3},
   pages={239--250},
   issn={0020-9910},
   review={\MR{0561973}},
   doi={10.1007/BF01390047},
}
\bib{kpr}{article}{
   author={Kumjian, Alex},
   author={Pask, David},
   author={Raeburn, Iain},
   title={Cuntz-Krieger algebras of directed graphs},
   journal={Pacific J. Math.},
   volume={184},
   date={1998},
   number={1},
   pages={161--174},
   issn={0030-8730},
   review={\MR{1626528}},
   doi={10.2140/pjm.1998.184.161},
}
\bib{vitt62}{article}{
   author={Leavitt, W. G.},
   title={The module type of a ring},
   journal={Trans. Amer. Math. Soc.},
   volume={103},
   date={1962},
   pages={113--130},
   issn={0002-9947},
   review={\MR{0132764}},
   doi={10.2307/1993743},
}


\bib{lind}{book}{
   author={Lind, Douglas},
   author={Marcus, Brian},
   title={An introduction to symbolic dynamics and coding},
   series={Cambridge Mathematical Library},
   edition={2},
   publisher={Cambridge University Press, Cambridge},
   date={2021},
   pages={xix+550},
   isbn={978-1-108-82028-8},
   review={\MR{4412543}},
   doi={10.1017/9781108899727},
}


\bib{chrisp}{book}{
   author={Phillips, N. Christopher},
   title={Equivariant $K$-theory and freeness of group actions on
   $C^*$-algebras},
   series={Lecture Notes in Mathematics},
   volume={1274},
   publisher={Springer-Verlag, Berlin},
   date={1987},
   pages={viii+371},
   isbn={3-540-18277-2},
   review={\MR{0911880}},
   doi={10.1007/BFb0078657},
}



\bib{P}{article}{
   author={Phillips, N. Christopher},
   title={A classification theorem for nuclear purely infinite simple
   $C^*$-algebras},
   journal={Doc. Math.},
   volume={5},
   date={2000},
   pages={49--114},
   issn={1431-0635},
   review={\MR{1745197}},
}
\bib{raeburn}{book}{
   author={Raeburn, Iain},
   title={Graph algebras},
   series={CBMS Regional Conference Series in Mathematics},
   volume={103},
   publisher={Conference Board of the Mathematical Sciences, Washington, DC;
   by the American Mathematical Society, Providence, RI},
   date={2005},
   pages={vi+113},
   isbn={0-8218-3660-9},
   review={\MR{2135030}},
   doi={10.1090/cbms/103},
}
\bib{restorff}{article}{
   author={Restorff, Gunnar},
   title={Classification of Cuntz-Krieger algebras up to stable isomorphism},
   journal={J. Reine Angew. Math.},
   volume={598},
   date={2006},
   pages={185--210},
   issn={0075-4102},
   review={\MR{2270572}},
   doi={10.1515/CRELLE.2006.074},
}

\bib{ror}{article}{
   author={R\o rdam, Mikael},
   title={Classification of Cuntz-Krieger algebras},
   journal={$K$-Theory},
   volume={9},
   date={1995},
   number={1},
   pages={31--58},
   issn={0920-3036},
   review={\MR{1340839}},
   doi={10.1007/BF00965458},
}


\bib{rordam111}{article}{
   author={R\o rdam, M.},
   title={Classification of nuclear, simple $C^*$-algebras},
   conference={
      title={Classification of nuclear $C^*$-algebras. Entropy in operator
      algebras},
   },
   book={
      series={Encyclopaedia Math. Sci.},
      volume={126},
      publisher={Springer, Berlin},
   },
   isbn={3-540-42305-X},
   date={2002},
   pages={1--145},
   review={\MR{1878882}},
   doi={10.1007/978-3-662-04825-2\_1},
}

\bib{rordam222}{article}{
   author={R\o rdam, Mikael},
   title={Structure and classification of $C^\ast$-algebras},
   conference={
      title={International Congress of Mathematicians. Vol. II},
   },
   book={
      publisher={Eur. Math. Soc., Z\"{u}rich},
   },
   isbn={978-3-03719-022-7},
   date={2006},
   pages={1581--1598},
   review={\MR{2275660}},
}
\bib{rosen}{book}{
   author={Rosenberg, Jonathan},
   title={Algebraic $K$-theory and its applications},
   series={Graduate Texts in Mathematics},
   volume={147},
   publisher={Springer-Verlag, New York},
   date={1994},
   pages={x+392},
   isbn={0-387-94248-3},
   review={\MR{1282290}},
   doi={10.1007/978-1-4612-4314-4},
}

\bib{roscho}{article}{
   author={Rosenberg, Jonathan},
   author={Schochet, Claude},
   title={The K\"{u}nneth theorem and the universal coefficient theorem for
   Kasparov's generalized $K$-functor},
   journal={Duke Math. J.},
   volume={55},
   date={1987},
   number={2},
   pages={431--474},
   issn={0012-7094},
   review={\MR{0894590}},
   doi={10.1215/S0012-7094-87-05524-4},
}
\bib{ruto}{article}{
AUTHOR = {Ruiz, Efren},
author={Tomforde, Mark},
     TITLE = {Classification of unital simple {L}eavitt path algebras of
              infinite graphs},
   JOURNAL = {J. Algebra},
    VOLUME = {384},
      YEAR = {2013},
     PAGES = {45--83},
      ISSN = {0021-8693},
       DOI = {10.1016/j.jalgebra.2013.03.004},
       URL = {https://doi.org/10.1016/j.jalgebra.2013.03.004},
}
\bib{smith3}{article}{
   author={Paul Smith, S.},
   title={Category equivalences involving graded modules over path algebras
   of quivers},
   journal={Adv. Math.},
   volume={230},
   date={2012},
   number={4-6},
   pages={1780--1810},
   issn={0001-8708},
   review={\MR{2927354}},
   doi={10.1016/j.aim.2012.03.031},
}


\bib{steinkind}{article}{
author={Steinberg, Benjamin},
title={A note on projections in \'etale groupoid algebras and diagonal preserving homomorphisms},
eprint={arXiv:2311.05694},
}


\bib{vas}{article}{
   author={Va\v{s}, Lia},
   title={The functor $K_0^{\rm{gr}}$ is full and only weakly faithful},
   journal={Algebr. Represent. Theory},
   volume={26},
   date={2023},
   number={6},
   pages={2877--2890},
   issn={1386-923X},
   review={\MR{4681336}},
   doi={10.1007/s10468-023-10199-w},
}


\bib{kh}{article}{
   author={Weibel, Charles A.},
   title={Homotopy algebraic $K$-theory},
   conference={
      title={Algebraic $K$-theory and algebraic number theory},
      address={Honolulu, HI},
      date={1987},
   },
   book={
      series={Contemp. Math.},
      volume={83},
      publisher={Amer. Math. Soc., Providence, RI},
   },
   isbn={0-8218-5090-3},
   date={1989},
   pages={461--488},
   review={\MR{0991991}},
   doi={10.1090/conm/083/991991},
}
\bib{chuk}{book}{
   author={Weibel, Charles A.},
   title={The $K$-book},
   series={Graduate Studies in Mathematics},
   volume={145},
   note={An introduction to algebraic $K$-theory},
   publisher={American Mathematical Society, Providence, RI},
   date={2013},
   pages={xii+618},
   isbn={978-0-8218-9132-2},
   review={\MR{3076731}},
   doi={10.1090/gsm/145},
}
\bib{williams}{article}{
   author={Williams, R. F.},
   title={Classification of subshifts of finite type},
   journal={Ann. of Math. (2)},
   volume={98},
   date={1973},
   pages={120--153; errata, ibid. (2) {\bf 99 (1974), 380--381}},
   issn={0003-486X},
   review={\MR{0331436}},
   doi={10.2307/1970908},
}




\end{biblist}
\end{bibdiv}
\end{document}